\documentclass[11pt]{amsart}

\usepackage{amssymb,amsfonts, amsmath, amsthm}
\usepackage[all,arc]{xy}
\usepackage{enumerate}
\usepackage{mathtools}
\usepackage[mathscr]{euscript}
\usepackage[T1]{fontenc}
\usepackage[utf8]{inputenc}
\usepackage{array}
\usepackage{tikz}
\usepackage{mathpazo}
\usepackage{tikz-cd}
\usepackage{textcomp}
\usepackage[pagebackref, colorlinks=true, linkcolor=blue, citecolor = red]{hyperref}
 \renewcommand*{\backref}[1]{}
 \renewcommand*{\backrefalt}[4]{({%
     \ifcase #1 Not cited.%
           \or On p.~#2%
           \else On pp.~#2%
     \fi%
     })}
\usepackage[nameinlink,capitalise,noabbrev]{cleveref}
\usepackage{url}
\usetikzlibrary{patterns}
\usepackage{geometry}
\usepackage{appendix}

\usepackage{todonotes}

\newcommand{\C}{\mathscr{C}}
\newcommand{\D}{\mathscr{D}}

\newcommand{\E}{\mathscr{E}}
\renewcommand{\H}{\mathscr{H}}
\newcommand{\I}{\mathscr{I}}
\newcommand{\bI}{\mathbb{I}}

\newcommand{\s}{\mathscr{S}}
\newcommand{\T}{\mathscr{T}}
\renewcommand{\ss}{s\mathscr{S}}
\newcommand{\M}{\mathscr{M}}
\newcommand{\Hom}{\mathrm{Hom}}
\newcommand{\Map}{\mathrm{Map}}
\newcommand{\Nat}{\mathrm{Nat}}
\newcommand{\Fun}{\mathrm{Fun}}
\newcommand{\map}{\mathrm{map}}
\newcommand{\comma}{,}
\newcommand{\set}{\mathscr{S}\mathrm{et}}
\newcommand{\cat}{\C\mathrm{at}}
\newcommand{\id}{\mathrm{id}}

\newcommand{\sSet}{\mathrm{s}\set}
\newcommand{\Diag}{\mathrm{Diag}}
\newcommand{\colim}{\mathrm{colim}}
\newcommand{\Ho}{\mathrm{Ho}}
\newcommand{\mul}{\mathrm{mul}}
\renewcommand{\exp}[2]{\mathrm{exp}(#1,#2)}
\newcommand{\rotatebot}{\rotatebox{270}{$\bot$}}

\newcommand{\ds}{\displaystyle}

\newcommand{\sint}{s\hspace{-0.04in}\int}

\newcommand{\sT}{s\mathscr{T}}
\newcommand{\bT}{\mathbb{T}}
\newcommand{\sH}{s\mathscr{H}}

\newcommand{\sbT}{s\bT}
\newcommand{\sbI}{s\bI}

\newcommand{\reb}{[}
\newcommand{\leb}{]}
\newcommand{\ordered}[1]{< #1 >}
\newcommand{\Yon}{\mathrm{Yon}}
\newcommand{\comp}{\mathrm{comp}}

\newcommand{\setref}[1]{\cref{Subsec Simplicial Sets}(\ref{#1})}
\newcommand{\spaceref}[1]{\cref{Subsec Simplicial Spaces}(\ref{#1})}

\newcommand{\adjun}[4]{
\begin{tikzcd}[row sep=0.5in, column sep=0.5in]
 #1  \arrow[r, shift left=1.8, "#3"] \pgfmatrixnextcell
 #2 \arrow[l, shift left=1.6, "#4", "\bot"'] 
\end{tikzcd}
}

\newcommand{\comsq}[8]{
  \begin{tikzcd}[row sep=0.5in, column sep=0.5in]
    #1 \arrow[r, "#5"] \arrow[d, "#6"']
    \pgfmatrixnextcell #2 \arrow[d, "#7"] \\
    #3 \arrow[r, "#8"]
    \pgfmatrixnextcell #4
  \end{tikzcd}
}

\newcommand{\pbsq}[8]{
  \begin{tikzcd}[row sep=0.5in, column sep=0.5in]
    #1 \arrow[r, "#5"] \arrow[d, "#6"'] \arrow[dr, phantom, "\ulcorner", very near start]
    \pgfmatrixnextcell #2 \arrow[d, "#7"] \\
    #3 \arrow[r, "#8"']
    \pgfmatrixnextcell #4
  \end{tikzcd}
}

\newcommand{\liftsq}[8]{
  \begin{tikzcd}[row sep=0.5in, column sep=0.5in]
    #1 \arrow[r, "#5"] \arrow[d, "#6"']
    \pgfmatrixnextcell #2 \arrow[d, "#7"] \\
    #3 \arrow[r, "#8"] \arrow[ur, dashed]
    \pgfmatrixnextcell #4
  \end{tikzcd}
}

\newtheorem{theone}[equation]{Theorem}
\newtheorem{lemone}[equation]{Lemma}
\newtheorem{propone}[equation]{Proposition}
\newtheorem{corone}[equation]{Corollary}

\theoremstyle{definition}
\newtheorem{defone}[equation]{Definition}
\newtheorem{exone}[equation]{Example}

\theoremstyle{remark}
\newtheorem{remone}[equation]{Remark}
\newtheorem{notone}[equation]{Notation}

\newtheoremstyle{TheoremNum}
{}{}              
{\itshape}                      
{}                              
{\bfseries}                     
{.}                             
{ }                             
{\thmname{#1}\thmnote{ \bfseries #3}}
\theoremstyle{TheoremNum}
\newtheorem{thmn}{Theorem}

\numberwithin{equation}{section}

\title{Yoneda Lemma for Simplicial Spaces}
\author{Nima Rasekh}
\address{{\'E}cole Polytechnique F{\'e}d{\'e}rale de Lausanne, SV BMI UPHESS, Station 8, CH-1015 Lausanne, Switzerland}
\email{nima.rasekh@epfl.ch}

\date{\today}

\begin{document}

\begin{abstract}
We study the Yoneda lemma for arbitrary simplicial spaces. 
We do that by introducing {\it left fibrations} of simplicial spaces and and studying its associated model structure, 
the {\it covariant model structure}.
\par 
In particular, we prove a {\it recognition principle} for covariant equivalences over an arbitrary simplicial space 
and {\it invariance} of the covariant model structure with respect to complete Segal space equivalences.
\end{abstract}

\maketitle
 \addtocontents{toc}{\protect\setcounter{tocdepth}{1}}

\tableofcontents

 \section{Introduction}\label{Sec Introduction}
 
 \subsection{Fibrations and the Yoneda Lemma}\label{subsec:fibrations and the yoneda lemma}
 The Yoneda lemma is a fundamental result in classical category theory. It allows us to embed every category into a larger category, 
 that shares many pleasant features with the category of sets. As a concrete example, we can use the Yoneda lemma and our  
 understanding of colimits in the category of sets, to give a precise description of colimits in an arbitrary category
 \cite{maclane1998categories}, \cite{riehl2014categoricalhomotopytheory}. 
 \par 
 In recent decades there has been an effort to generalize the notion of a category to an {\it $(\infty,1)$-category}, 
 which satisfies the conditions of a category only up to coherent homotopies and is thus better suited to study 
 objects that arise naturally in homotopy theory. This first happened via several models: 
 {\it quasi-categories} \cite{boardmanvogt1973qcats}, 
 {\it complete Segal spaces} \cite{rezk2001css}, 
 {\it Kan enriched categories} \cite{bergner2007bergnermodelcat}, 
 and many other other models \cite{bergner2010survey}, \cite{bergner2018book}. 
 Later Riehl and Verity developed a model-independent approach, which encompasses most of these definitions: {\it $\infty$-cosmoi}
 \cite{riehlverity2017inftycosmos}.
 \par 
 Given the important role of the Yoneda lemma for classical categories, we
 want to study the Yoneda lemma for $(\infty,1)$-categories. 
 In fact, the Yoneda lemma and its associated theory of {\it left fibrations} has already been studied in several models of $\infty$-categories:  
 \begin{itemize}
  \item {\it Quasi-categories:} Studied first by Joyal \cite{joyal2008notes, joyal2008theory} and then Lurie \cite{lurie2009htt}. 
  Has since been reworked using different methods by
  Heuts and Moerdijk \cite{heutsmoerdijk2015leftfibrationi, heutsmoerdijk2016leftfibrationii}, Stevenson \cite{stevenson2017covariant}, 
  Cisinski \cite{cisinski2019highercategories}
  and Nguyen \cite{nguyen2019covariant}.
  \item {\it Kan enriched categories:} 
  Studied primarily in the context of the Yoneda lemma for enriched categories \cite{kelly1982enriched}, \cite{riehl2014categoricalhomotopytheory}.
  \item {\it $\infty$-Cosmoi:} Introduced and proven by Riehl and Verity \cite{riehlverity2017inftycosmos}.
  \item {\it Segal spaces:} Studied by de Brito \cite{debrito2018leftfibration} and Kazhdan, Varshavsky \cite{kazhdanvarshvsky2014yoneda}.
 \end{itemize}
 The study of Yoneda lemma for general simplicial spaces has received far less attention. 
 \par 
 The goal of this work is to focus on the Yoneda lemma and its associated theory of fibrations in the context of 
 complete Segal spaces and, more generally, arbitrary simplicial spaces. 
 
 \subsection{Why Simplicial Spaces?}\label{subsec:why simplicial spaces?}
  Given that most results here appeared in one form or another in the language of quasi-categories or $\infty$-cosmoi why 
  present the material in the language of simplicial spaces?
  The key observation is that most results about left fibrations have only been proving using the model of quasi-categories and hence rely on 
  using combinatorial techniques with simplicial sets. On the other hand the proofs with simplicial spaces rely on homotopy invariant proofs 
  with Kan complexes. 
  \par 
  Beside the theoretical satisfaction a homotopy theorist might derive from proving results using homotopy invariant techniques, 
  there are also concrete mathematical interests:
    
    \medskip 
    
   {\bf Synthetic $\infty$-Category Theory:}
    {\it Homotopy type theory} is a new foundation for mathematics that is by definition homotopy invariant and thus in many ways better suited for 
    homotopical constructions  \cite{hottbook2013}. 
    One long term goal is to use homotopy type theory to introduce a {\it synthetic} notion of $\infty$-categories. 
    A first step in this regard has been taken by Riehl and Shulman \cite{riehlshulman2017rezktypes} who introduced a notion of a {\it Rezk type} as a 
    way to define $(\infty,1)$-categories in homotopy type theory. The notion of a Rezk type is motivated by complete Segal spaces and so 
    in particular their approach to fibrations and the Yoneda lemma corresponds to fibrations of complete Segal spaces rather than quasi-categories. 

   \medskip
   
   {\bf Completeness:}
   One defining property of $(\infty,1)$-categories is {\it completeness}, first introduced by Charles Rezk \cite{rezk2001css} 
   and then rediscovered by Vladimir Voevodsky in the context of homotopy type theory, where it is introduced as the 
   {\it univalence axiom} \cite{hottbook2013}. 
   \par 
   From a foundational perspective we want to determine which results in $(\infty,1)$-category depend on the univalence axiom (i.e. require completeness) 
   and which ones hold in a more general foundation. However, we cannot directly use quasi-categories to address this problem 
   as quasi-categories are automatically complete. Rather it would require us to use technical tools such as 
   {\it flagged $\infty$-categories} \cite{ayalafrancis2018flagged}.
   On the other hand complete Segal spaces can be easily generalized to Segal spaces and so give us a direct computational way to 
   study the necessity of completeness: We simply have to check whether a result only holds for complete Segal spaces or can be generalized to 
   Segal spaces.
   
   \medskip 
   
   {\bf Fibrations of $(\infty,n)$-Categories:}
   The same way that $1$-categories have been generalized to $(\infty,1)$-categories, strict $n$-categories have been generalized to 
   $(\infty,n)$-categories. Similar to the $(\infty,1)$-categorical case there is now a long list of models
   \cite{bergner2020surveyn}, however, unlike the 
   $(\infty,1)$-categorical case many important questions about $(\infty,n)$-categories have remained unanswered. 
   \par 
   First of all it is not yet proven that the common models of $(\infty,n)$-category that appear in the literature are actually equivalent.
   For example, it is known that {\it $\Theta_n$-spaces} \cite{rezk2010thetanspaces} are equivalent to {\it $n$-fold complete Segal spaces} 
   \cite{barwick2005nfoldsegalspaces} as proven by Bergner and Rezk \cite{bergnerrezk2013comparisoni,bergnerrezk2020comparisonii}. 
   However, it is not known whether they are equivalent to {\it complicial sets} \cite{verity2008complicial} and both of those are not known 
   to be equivalent to {\it comical sets} \cite{campionkapulkinmaehara2020comical}. These are just some of the models that appear in the 
   literature and clearly illustrate the challenges that lie ahead.
   \par 
   Moreover, the theory of fibrations for none of these models has been completely developed. There are some results for 
   one model of $(\infty,2)$-categories, namely $2$-complicial sets, by Lurie \cite{lurie2009goodwillie} (there called {\it scaled simplicial sets}). 
   However, there are no results yet for $\Theta_n$-spaces and $n$-fold complete Segal spaces, both of which are natural generalizations of 
   complete Segal spaces. 
   \par 
   On the other hand each model has its own applications in various branches of mathematics. For example, $n$-fold complete Segal spaces have been the 
   primary model in the study of topological field theories and the cobordism hypothesis \cite{lurie2009cobordism},\cite{calaquescheimbauer2019cobordism} 
   and thus merit a theory of fibrations. However, given the difficulties we currently face comparing different models and the fact that the theory 
   of fibrations has not been developed for any of the models, transferring 
   results from one model to another (the way we could for $(\infty,1)$-categories)  
   is currently not possible. It is thus imperative to study fibration of $n$-fold complete Segal spaces in its own right.
   \par 
   The fact that $n$-fold complete Segal spaces are a direct generalization of complete Segal spaces thus means that an important first step 
   towards realizing this goal is to study left fibrations of complete Segal spaces. 

   \medskip 
   
   In all three examples, what is important is not just to know that a certain result holds, but rather how to prove the desired results. 
   For example, the study of fibrations of $n$-fold complete Segal spaces is expected to be a direct generalization 
   of the results for simplicial spaces proven here. 
   
   \medskip 
   
   Finally, some of the results here can already be found in work of 
   de Brito \cite{debrito2018leftfibration} and Kazhdan-Varshavsky \cite{kazhdanvarshvsky2014yoneda}. However, the most important results
   are proven at the level of arbitrary simplicial spaces and are thus far more general. 
   Two important examples of such theorems are the recognition principle for covariant equivalences 
   (\cref{the:covar equiv over simp space})
   and the invariance of the covariant model structure under CSS equivalences (\cref{The Covar invariant under CSS equiv}).

 \subsection{Main Results}\label{subsec:main results}
 The paper focuses on the study of {\it left fibrations}. A left fibration is a Reedy fibration of simplicial spaces $p:Y \to X$ 
 such that for all $n \geq 0$ the commutative square 
 \begin{center}
  \begin{tikzcd}[row sep=0.5in, column sep=0.5in]
   Y_n \arrow[r, "p_n"] \arrow[d, "\ordered{0}^*"'] & X_n \arrow[d, "\ordered{0}^*"] \\
   Y_0 \arrow[r, "p_0"] & X_0
  \end{tikzcd}
 \end{center}
 is a homotopy pullback square of spaces (\cref{def:left fibration}). It generalizes the unique lifting condition of a Grothendieck opfibration 
 for categories (\cref{def:discrete Groth fib}). It is well-established that discrete Grothendieck opfibrations model covariant functors 
 valued in sets (\cref{prop:integral adjunctions}) and by analogy we think of left fibrations as a model for covariant functors valued in spaces, 
 which guides our work throughout this paper.
 
 \medskip
 
 Unlike Grothendieck opfibrations, the study of left fibrations requires higher-categorical techniques, which is why 
 we will use the theory of model categories (\cref{Sec Some Facts about Model Categories}).
 Concretely, we will show that for each simplicial space $X$ there is a unique simplicial model structure on the category of simplicial 
 spaces over $X$, denoted $\ss_{/X}$ and called  
 the {\it covariant model structure}, such that the fibrant objects are precisely the left fibrations (\cref{The Covariant Model Structure}).
 
 \medskip
 
 One important model of an $(\infty,1)$-category is a {\it complete Segal space}, which is a simplicial space that satisfies the 
 certain lifting conditions that endows it with the structure of a category (\cref{Def Complete Segal Spaces}).
 As a first step, we thus expect that left fibrations over complete Segal spaces behave analogously to Grothendieck opfibrations over 
 categories. We will in fact take a more general step and study left fibrations over {\it Segal spaces} (\cref{Def Segal Spaces}).
 Using Segal spaces, we get following generalization of the Yoneda lemma
 (here $F(n)$ is the representable simplicial discrete space \spaceref{item:Fn}).
 
 \begin{thmn}[\ref{the:yoneda for Segal spaces}]
  Let $W$ be a Segal space and $x: F(0) \to W$ be an object. Let
  $$W_{x/} = W^{F(1)} \strut^{t} \underset{W}{\times}^{ \{ x \} } F(0) \xrightarrow{ \ t \ } W$$
  be the under-category projection, which is a left fibration (\cref{the:under CSS left fibration}).
  Then the map 
  \begin{center}
   \begin{tikzcd}[row sep=0.5in, column sep=0.5in]
    F(0) \arrow[rr, "\{\id_x\}", "\simeq"'] \arrow[dr, "\{ x \}"] & & W_{x/} \arrow[dl, "t"] \\
    & W &
   \end{tikzcd}
  \end{center}
  is a covariant equivalence over $W$ and so in particular for every left fibration $L \to W$ the induced map 
  $$\{\id_x\}^*: \Map_{/W}(W_{x/},L) \to \Map_{/W}(F(0),L) $$
  is a Kan equivalence (\cref{cor:yoneda lemma Segal spaces hom version}).
 \end{thmn}

 If the Segal space is the nerve of a category, then we can give a very precise relationship between left fibrations and functors valued in spaces.
 
 \begin{thmn}[\ref{the:simplicial groth construction}]
  Let $\C$ be a small category.
   The two simplicially enriched adjunctions  
    \begin{center}
    \begin{tikzcd}[row sep=0.5in, column sep=0.9in]
     \Fun(\C,\s)^{proj} \arrow[r, shift left = 1.8, "\sint_\C"] & 
     (\ss_{/N\C})^{cov} \arrow[l, shift left=1.8, "\sH_\C", "\bot"'] \arrow[r, shift left=1.8, "\sbT_\C"] &
     \Fun(\C,\s)^{proj} \arrow[l, shift left=1.8, "\sbI_\C", "\bot"']
    \end{tikzcd}
   \end{center}
   are Quillen equivalences. Here $\Fun(\C,\s)$ has the projective model structure (\cref{def:projective model structure}) 
   and $\ss_{/N\C}$ has the covariant model structure over $N\C$.
  \end{thmn}
 
 One new and important aspect of this work is that we generalize these results from Segal spaces to arbitrary simplicial spaces. 
 There are two critical steps: the first one is the {\it recognition principle} for covariant equivalences.
 
 \begin{thmn}[\ref{the:condition for covar equiv general rep}] 
 $p:Y \to X$ be a map of simplicial spaces. 
 For every $x: F(0) \to X$,  
 there is a natural zig-zag of diagonal equivalences (\cref{The Diagonal Model Structure})
 $$R_x \underset{X}{\times} Y \xrightarrow{ \ \simeq \ } R_x \underset{X}{\times} \hat{Y} \xleftarrow{ \ \simeq \ } F(0) \underset{X}{\times} \hat{Y}$$
 Here $i: Y \to \hat{Y}$ is a choice of a left fibrant replacement of $Y$ over $X$ and $R_x \to X$ is a contravariant fibrant replacement of 
 $\{  x \} : F(0) \to X$ (\cref{rem:right fibrations}).
\end{thmn}

\begin{thmn}[\ref{the:covar equiv over simp space}]
  Let $g: Y \to Z$ be a map over $X$ and $R_x \to X$ a choice of contravariant fibrant replacement of the map $\{x\}:F(0) \to X$ 
  (\cref{rem:right fibrations}).
  Then $g:Y \to Z$ over $X$ is a covariant equivalence if and only if for every $x : F(0) \to X$
  $$ R_x \underset{X}{\times} Y \to R_x \underset{X}{\times} Z$$
  is a diagonal equivalence (\cref{The Diagonal Model Structure}).
\end{thmn}
 
The second is the {\it invariance property} of the covariant model structure.
 
 \begin{thmn}[\ref{The Covar invariant under CSS equiv}]
   Let $f: X \to Y$ be a CSS equivalence (\cref{The Complete Segal Space Model Structure}). Then the adjunction
   \begin{center}
     \adjun{(\ss_{/X})^{cov}}{(\ss_{/Y})^{cov}}{f_!}{f^*}
   \end{center}
   is a Quillen equivalence. Here both sides have the covariant model structure.
  \end{thmn}
  
  Using the invariance property, we can prove the following:
  
  \begin{thmn}[\ref{The Covariant local of CSS}]
   The covariant model structure is a localization of the CSS model structure on $\ss_{/X}$.
  \end{thmn}
  
  \begin{thmn}[\ref{The Pullback preserves CSS equiv}]
   Base change by left fibrations preserves CSS equivalences.
  \end{thmn}
 
 \subsection{Outline} \label{Subsec Outline}
 In \cref{Sec Another Look at the Yoneda Lemma for Categories} we review the 
 classical Yoneda lemma in \cref{Subsec Tensor Product of Functors and Yoneda Lemma}, the Grothendieck construction in \cref{subsec:functors to fibrations} 
 and fibrational Yoneda lemma 
 for categories in \cref{subsec Yoneda Lemma for Fibered Categories} with an eye towards a generalization to simplicial spaces. 
 
 \medskip
 
 \cref{Sec Basics Conventions} is a review of necessary background concepts: 
 Joyal-Tierney calculus (\cref{subsec:joyal tierney calculus}), 
 spaces (\cref{Subsec Simplicial Sets}), 
 simplicial spaces (\cref{Subsec Simplicial Spaces}),  
 the Reedy model structure (\cref{Subsec Reedy Model Structure}) and 
 complete Segal spaces (\cref{Subsec Complete Segal Spaces}). 
 
 \medskip
 
 In \cref{Sec Left Fibrations and the Covariant Model Structure} we begin the study of left fibration. 
 In \cref{subsec:many faces left fib} we introduce left fibrations and give various alternative characterizations. 
 We then move on in \cref{subsec:covariant model structure} to define a model structure for 
 left fibrations, the {\it covariant model structure} (\cref{The Covariant Model Structure}). Finally, in \cref{subsec:yoneda lemma Segal spaces} 
 we study left fibrations over Segal spaces and in particular prove the {\it Yoneda lemma for Segal spaces} (\cref{the:yoneda for Segal spaces}). 
 
 \medskip
 
 In the next section, \cref{sec:grothendieck construction}, we first take a technical digression in 
 \cref{subsec:grothendieck construction over categories} and focus on the covariant model structure over 
 nerves of categories and in particular prove the {\it Grothendieck construction} in \cref{the:simplicial groth construction}. 
 We then use these new technical results in 
 \cref{Subsec The Yoneda Lemma} to prove the {\it recognition principle for covariant equivalences} (\cref{the:covar equiv over simp space}).
 
 \medskip
 
 In the final section, \cref{Sec Complete Segal Spaces and Covariant Model Structure}, we study the relation between left fibrations and 
 complete Segal spaces. In particular, in \cref{Subsec Invariance of Covariant Model Structure under CSS Equivalences} we prove the 
 {\it invariance of the covariant model structure} (\cref{The Covar invariant under CSS equiv})
 and several important implications. Finally, in \cref{Subsec Colimits Cofinality and Quillens Theorem A}  
 we apply these result to the study of colimits in Segal spaces. 
 
 \medskip
 
 There are two appendices. In \cref{Sec Some Facts about Model Categories} we review some key lemmas about model categories. 
 In \cref{Sec Comparison with Quasi-Categories} we prove that the covariant model structure for 
 simplicial spaces is Quillen equivalent to the covariant model structure for simplicial sets studied in \cite{lurie2009htt}.
 
 \subsection{Background} \label{Subsec Background}
 The main language here is the language of model categories and complete Segal spaces. So, we assume familiarity with both throughout.
 Only a few results are explicitly stated here. For a basic introduction to the theory of model categories see
 \cite{dwyersspalinski1995modelcat} or \cite{hovey1999modelcategories}. For an introduction to complete Segal spaces see the original source 
 \cite{rezk2001css}.

 \subsection{Notation} \label{subsec:notation}
 We mostly follow the notation as introduced in \cite{rezk2001css} and will be reviewed in \cref{Sec Basics Conventions}. 
 We use categories with different enrichments and use the following notation to distinguish between them.
 Fix a category $\C$ and two objects $X,Y$.
 \begin{itemize}
  \item We denote the {\it set} of maps between them by $\Hom_\C(X,Y)$. For a given object $X$ and maps 
  $g: Y \to X, Z \to X$, we will denote the set of maps $\Hom_{\C_{/X}}(Y,Z)$ by $\Hom_{/X}(Y,Z)$.
  \item There is one exception to the previous rule. If $\C$ is a category of functors, then we denote the 
  {\it set of natural transformations} from $F$ to $G$ by $\Nat(F,G)$, following conventional notation.
  \item If $\C$ is enriched over the category of simplicial sets, we denote the {\it mapping simplicial set} by $\Map_\C(X,Y)$ or, 
  if $\C$ is clear from the context, by $\Map(X,Y)$. Similar to the last one we will, instead of 
  $\Map_{\C_{/X}}(Y,Z)$, use $\Map_{/X}(Y,Z)$.
  \item If $\C$ is Cartesian closed, we denote the internal mapping object by $Y^X$.
  \item There is one exception to the previous rule. For two categories $\C$, $\D$, we denote the {\it category} of functors by 
  $\Fun(\C,\D)$, following conventional notation.
  \item If $W$ is a Segal space, then for two objects $x,y$ in $W$ there is a mapping space, which we denote by $\map_W(x,y)$  
  (\cref{def:mapping space Segal space}).
 \end{itemize}
 For a functor between small categories $F: \C \to \D$ and bicomplete category $\E$ we use the following notation for the induced 
 functors at the level of presheaf categories:
 \begin{center}
  \begin{tikzcd}[row sep=0.5in, column sep=0.5in]
   \Fun(\C,\E) \arrow[r, "F_!", shift left=1.4, bend left =30] \arrow[r, "F_*"', shift right=1.4, bend right=30]& \Fun(\D,\E) \arrow[l, "F^*"'] 
  \end{tikzcd}
 \end{center}
 Here $F^*$ is defined by precomposition, $F_!$ is the left Kan extension and $F^*$ the right Kan extension.
 \par 
 Similarly, for a given morphism $f:c \to d$ in a category $\C$ with small limits and colimits we denote the adjunctions 
 \begin{center}
  \begin{tikzcd}[row sep=0.5in, column sep=0.5in]
   \C_{/c} \arrow[r, "f_!", shift left=1.4, bend left=30] \arrow[r, "f_*"', shift right=1.4, bend right = 30]& \C_{/d} \arrow[l, "f^*"'] 
  \end{tikzcd}
 \end{center}
 where $f_!$ is the postcomposition functor, $f^*$ the pullback functor and $f_*$ is the right adjoint to $f^*$.
 
 Finally, let $\C$ be a category with final object $1$. Then we use notation $\{y\}: X \to Y$ 
 for the unique map that factors through the map $1 \to Y$ that picks out the element $y$ in $Y$.
 
 \subsection{Acknowledgements} \label{Subsec Acknowledgements}
 I want to thank my advisor Charles Rezk who has guided me through every step of the work. I want to in particular thank him for many helpful 
 remarks on previous drafts. I also want to thank Matt Ando for many fruitful conversations and suggestions.
 Finally, I want to thank the referee for many helpful comments and suggestions on how to improve and simplify the theorems.

\section{Another Look at the Yoneda Lemma for Classical Categories} \label{Sec Another Look at the Yoneda Lemma for Categories}
The Yoneda lemma is an important result in classical category theory and is thus well known among practitioners of category theory. 
A lesser known aspect of the Yoneda lemma is that it can be expressed in several different ways. 
Concretely we want to review four different faces of the Yoneda lemma, which are summarized in this table:
\begin{center}
 \begin{tabular}{|c|c|c|}
  \hline
  & {\it Hom} & {\it Tensor} \\ \hline
  {\it Functor} & $\Nat ( \Hom_{\C}(C, -), F) \xrightarrow{\cong} F(C)$ &  
  $\Hom_{\C}(C, -) \underset{\C}{\otimes} F \xrightarrow{\cong} F(C)$\\ 
  & \cref{lemma:hom functor Yoneda} & \cref{lemma:tensor functor Yoneda} \\ \hline
  {\it Fibration} & $\Fun_{/\C}(\C_{/C} , \mathscr{D}) \xrightarrow{\cong} \{ C \} \underset{\C}{\times} \mathscr{D}$ &
  $\C_{C/} \underset{\C}{\otimes} \mathscr{D} \xrightarrow{ \ \ \cong \ \ } 
  \{ C \} \underset{\C}{\otimes} \mathscr{D}$ \\
  & \cref{lemma:hom fibration Yoneda} & \cref{lemma:tensor fibration Yoneda} \\ \hline   
 \end{tabular}
\end{center}

Let us start with the most common form of the Yoneda lemma, which can be found in any introductory book on classical category theory. 
Here is a version that appears in \cite[Page 61]{maclane1998categories}.
\begin{lemone} \label{lemma:hom functor Yoneda}
 (Hom version of Yoneda for functors)
 If $F: \C \to Set$ is a functor and $C \in \C$ an object, then the natural map
 \begin{center}
  \begin{tikzcd}[row sep=0.07in, column sep=0.5in]
   \Nat ( \Hom_{\C}(C, -), F) \arrow[r, "\cong"] & F(C) \\
   \reb\alpha: \Hom_{\C}(C, -)\to F \leb \arrow[r, mapsto] &  \alpha_C (id_C) 
  \end{tikzcd}
 \end{center}
is a bijection.
\end{lemone}

 There is, however, a different way this equivalence can be phrased.  
 It relies on the {\it Hom-Tensor Adjunction}.
 
\subsection{Tensor Product of Functors and Yoneda Lemma} \label{Subsec Tensor Product of Functors and Yoneda Lemma}
 Most of the material in this subsection can be found in greater detail in \cite[VII.2]{maclanemoerdijk1994topos}.
 For this subsection let $\C$ be a fixed category and $F: \C \to Set$ and $P: \C^{op} \to \set$ be two functors.
 Then we define the tensor product as the following colimit diagram $F \otimes_{\C} P$.
 \begin{center}
  \begin{tikzcd}[row sep=0.5in, column sep=0.5in]
   \ds\coprod_{C,C' \in \C} P(C) \times Hom_{\C}(C,C') \times F(C) 
   \arrow[r, shift left=1.2, "\varphi"] \arrow[r, shift right=1.2, "\psi"'] & 
   \ds\coprod_{C \in \C} P(C) \times F(C)  \arrow[r, shorten >=1ex,shorten <=1ex] & \ds F \underset{\C}{\otimes} P
  \end{tikzcd}
 \end{center}
 where $\varphi(a,f,b) = (P(f)(a),b)$ and $\psi(a,f,b) = (a,F(f)(b))$.
 So the tensor product of two functors is the product of the values quotiented out by the mapping relations. 
 This definition generalizes the tensor product of two rings, which is the motivation for this notation.
 Similar to the case of rings this definition of a tensor product fits into a {\it hom-tensor adjunction}.
 \begin{theone}
  Let $\C$ be a category and $F: \C \to \set$ a functor. Then we have the adjunction
  \begin{center}
   \begin{tikzcd}[row sep=0.5in, column sep=1in]
    \Fun(\C^{op},\set)  \arrow[r, shift left=1.4, "- \otimes_{\C} F"] &
    \set \arrow[l, shift left=1, "\Hom_{\set}(F(-) \comma -)"] 
   \end{tikzcd}
  \end{center}
  where the left adjoint takes $P$ to $P \otimes_{\C} F$ and the right adjoint takes a set $S$
  to the functor which takes an object $C$ to $Hom_{\set}(F(C),S)$.
 \end{theone}

 \begin{remone}
  Note that we could have made the same construction for the case where $\set$ is replaced with any category which has all colimits.
  However, here we do not need to work at this level of generality. For more details on the general construction see 
  \cite[Page 358]{maclanemoerdijk1994topos}.
 \end{remone}
 
 With the tensor product at hand we can state another version of the Yoneda lemma.
 
 \begin{lemone}\label{lemma:tensor functor Yoneda}
 (Tensor version of Yoneda for functors)
 If $P: \C^{op} \to Set$ is a functor and $C \in \C$ an object, then the natural map
 \begin{center}
  \begin{tikzcd}[row sep=0.07in, column sep=0.5in]
   \Hom_{\C}(C, -) \underset{\C}{\otimes} P \arrow[r, "\cong"] & P(C) \\
   (f:C \to C',a \in P(C')) \arrow[r, mapsto] & P(f)(a) 
  \end{tikzcd}
 \end{center}
 is a bijection.  
 \end{lemone}

 This version of the Yoneda lemma has the following basic corollaries, which should look quite familiar.
 
 \begin{corone}
  Let $\C$ be a category and $C,C'$ two objects. Then we have the following isomorphism.
  $$ Hom_{\C}(C, -) \underset{\C}{\otimes} Hom_{\C}( - , C') \cong Hom_{\C}(C,C')$$
 \end{corone}

 \begin{corone}
  Let $\C$ be a category and $P, Q: \C^{op} \to Set$ be two functors.
  Then a natural transformation $\alpha: P \to Q$ is a natural equivalence if and only if 
  $$Hom_{\C}(C, -) \underset{\C}{\otimes} P \to Hom_{\C}(C, -) \underset{\C}{\otimes} Q$$
  is a bijection for every object $C \in \C$.
 \end{corone}

 \subsection{From Functors to Fibrations: The Grothendieck Construction} \label{subsec:functors to fibrations}
 We want to now translate the Yoneda lemma from a statement about functors to a statement about fibrations. 
 This requires us to translate between functors and fibrations which we will do via the Grothendieck construction.
 
 \begin{remone}
  We will need a careful understanding of the Grothendieck construction in the coming sections. 
  Thus the review in this section is self-contained. However, the ideas are in no way new and a more detailed approach can be found in many 
  places, such as \cite[I.5]{maclanemoerdijk1994topos}, or \cite[A1.1.7, B1.3.1]{johnstone2002elephanti}.
 \end{remone}
 
 \begin{defone}
  Let $\C$ be a category. Define
  $$\int_\C : \Fun(\C,\set) \to \cat_{/\C}$$ 
  as the functor that takes $F: \C \to \set$ to the category $\int_\C F \to \C$ with 
  \begin{itemize}
   \item Objects: Pairs $(c,x)$ where $c$ is an object in $\C$ and $x \in F(c)$.
   \item Morphisms: A morphisms $(c,x) \to (d,y)$ is a choice of morphism $f: c \to d$ in $\C$ such that $F(f)(x) = y$.
  \end{itemize}
  It comes with an evident projection map $\pi_F: \int_\C F \to \C$. 
 \end{defone}
 
 This functor has a left adjoint and a right adjoint that we want to define in detail.
 
 \begin{defone}
  Let $\C$ be a category. We define the functor  
  $$\C_{/-}: \C \to \cat_{/\C}$$
  that takes an object to the over-category $\C_{/c}$ and a morphism $f:c \to d$ to the post-composition $f_!: \C_{/c} \to \C_{/d}$.
  
  Similarly, define the functor 
  $$\C_{-/}: \C^{op} \to \cat_{\C/}$$
  that takes an object to the under-category $\C_{c/}$ and a morphism $f: c \to d$ to the precomposition $f^*: \C_{d/} \to \C_{c/}$.
 \end{defone}

 For a given category over $\C$, $p: \D \to \C$ define 
 $$\T_\C(p: \D \to \C): \C \to \set$$
 as the composition 
 $$\C \xrightarrow{ \ \C_{/-} \ } \cat_{/\C} \xrightarrow{ \ - \times_\C \D \ } \cat \xrightarrow{ \ \pi_0 \ } \set$$
 in other words we have $\T_\C(c) = \pi_0(\C_{/c} \times_\C \D)$.
 Similarly, define 
 $$\H_\C(p: \D \to \C): \C \to \set$$ 
 as the composition
 $$\C \xrightarrow{ \ (\C_{-/})^{op} \ } (\cat_{/\C})^{op} \xrightarrow{ \ \Hom_{/\C}(-,\D) \ } \set$$
 meaning we have $\H_\C(c) = \Hom_{/\C}(\C_{c/}, \D)$.
 We claim that $\T_\C$ is the left adjoint and $\H_\C$ is the right adjoint to $\int_\C$.
 
 \begin{propone} \label{prop:integral adjunctions}
  We have following diagram of adjunctions.
  \begin{center}
   \begin{tikzcd}[row sep=0.5in, column sep=0.9in]
    \Fun(\C,\set) \arrow[r, "\int_\C" description] & \cat_{/\C} \arrow[l, bend left = 30, "\H_\C", "\bot"'] \arrow[l, bend right=30, "\T_\C"', "\bot"]  
   \end{tikzcd}
  \end{center}
 \end{propone}

 \begin{proof}
  We first prove the right adjoint. The functor $\int_\C$ is a left Kan extension
  \begin{center}
   \begin{tikzcd}[row sep=0.5in, column sep=0.5in]
    \C ^{op} \arrow[r, "\C_{-/}"] \arrow[d, "\Yon"'] & \cat_{/\C} \\
    \Fun(\C,\set) \arrow[ur, dashed, "\int_\C"']
   \end{tikzcd}
  \end{center}
  which means it commutes with colimits and has a right adjoint given by $\H_\C(p: \D \to \C) = \Hom_\C(\C_{-/},\D)$.
  \par 
  For the left adjoint, we first show that $\T_\C$ commutes with colimits. 
  As colimits are computed point-wise this means we have to prove that for every object $c$ the functor 
  $$\pi_0(\C_{/c} \times_{\C} - ) : \cat_{/\C} \to \set$$
  commutes with colimits. This functor is a composition and so we check separately that both are left adjoints:
  \begin{enumerate}
   \item $\C_{/c} \times_\C -$ is a left adjoint because $\C_{/c} \to \C$ is a Conduch{\'e} functor \cite{conduche1972fibrations}.
   \item $\pi_0$ is the left adjoint of the inclusion functor $\set \to \cat$.
  \end{enumerate}
  Now, we prove that $\T_\C$ is the left adjoint. Every category over $\C$ is a colimit of functors $\alpha: [n] \to \C$ 
  and so the result follows from the following natural isomorphisms:
  \begin{center}
  \begin{tikzcd}
   \Hom_{/\C}(\alpha:[n] \to \C, \int_\C G) \arrow[r, "\{ 0 \}^*", "\cong"']
   & \Hom_{/\C}(\alpha \circ \{ 0 \}: [0] \to \C, \int_\C G) \arrow[r, "\alpha(0)", "\cong"']
   \arrow[d, phantom, ""{coordinate, name=Z}]
   & G(\alpha(0)) \arrow[dll,
   "\Yon_{\alpha(0)}" description, "\cong"' near start,
   rounded corners,
   to path={ -- ([xshift=2ex]\tikztostart.east)
   |- (Z) [near end]\tikztonodes
   -| ([xshift=-2ex]\tikztotarget.west)
   -- (\tikztotarget)}] \\
   \Nat(\Hom(\alpha(0),-),G) \arrow[r, "\cong"]
   & \Nat(\T_\C(\alpha:[n] \to \C),G)
  \end{tikzcd}.
 \end{center}
  Here the first isomorphism follows from the fact that a functor out of $[n]$ is determined by the image of $[0]$ and 
  the second isomorphism is just the Yoneda lemma.
 \end{proof}

 In fact $\int_\C$ has even more desirable properties.
 
 \begin{lemone} \label{lemma:int fully faithful}
  $\int_\C: \Fun(\C,\set) \to \cat_{/\C}$ is fully faithful.
 \end{lemone}

 \begin{proof}
  Let $F,G : \C \to \set$ be two functors. We need to prove that the map 
  $$\Nat(F,G) \to \Hom_{/\C}(\int_\C F , \int_\C G)$$
  is a bijection of sets. For that we will construct an inverse. Concretely, we define 
  $$\I_\C : \Hom_{/\C}(\int_\C F , \int_\C G) \to \Nat(F,G) $$
  as follows. For a given functor $H: \int_\C F \to \int_\C G$ over $\C$ we define the natural transformation 
  $\I_\C(H)_c(x) = H(c,x)$. The functoriality of $H$ implies that $\I(H)$ is natural. 
  \par 
  It remains to show these are inverses. For a given natural transformation $\alpha: F \Rightarrow G$ we have 
  $$(\I_\C \int_\C (\alpha))_c(x) = \int_\C(\alpha)(c,x) = \alpha_c(x)$$
  and on the other side 
  $$\int_\C(\I_\C(H))(c,x) = (\I_\C(H))_c(x) = H(c,x)$$
  finishing the proof.
 \end{proof}
 
 This has direct implication for $\T_\C$ and $\H_\C$.
 
 \begin{corone}
  $\T_\C$ is a localization functor and $\H_\C$ is a colocalization functor.
 \end{corone}

 We end this subsection by observing that while $\int_\C$ is fully faithful, it is in fact not essentially surjective.
 
 \begin{defone}
  A functor $p: \D \to \C$ is {\it conservative} if it reflects isomorphisms. 
 \end{defone}

 \begin{lemone} \label{lemma:Grothendieck fibration conservative}
  Let $F: \C \to \set$ be a functor. Then $\pi_F:\int_\C F \to \C$ is conservative.
 \end{lemone}

 \begin{proof}
  Let $f: (c,x) \to (d,y)$ be a morphism in $\int_\C F$ such that the underlying morphism $\pi_F(f): c \to d$ is an isomorphism. We need to show
  that $f$ in $\int_\C F$ is an isomorphism and we will do so by providing an inverse. 
  Let $f^{-1}: d \to c$ in $\C$ be the inverse of $\pi_F(f)f$ in $\C$. The inverse is now given by the lift
  $f^{-1}: (d,y) \to (c,x)$.
 \end{proof}
 
 Thus we need to restrict our attention to the essential image of $\int_\C$, which leads us to discrete Grothendieck fibrations.
 
 \subsection{Yoneda Lemma for Grothendieck Fibrations} \label{subsec Yoneda Lemma for Fibered Categories}
 In this subsection we want to use the fully faithful functor $\int_\C$ to translate both versions of the Yoneda lemma 
 from a functorial statement to a fibrational one.  
 
 This might not be a major improvement when studying $1$-categories, 
 however, in the world of higher categories functors can be difficult to study, because of the homotopy coherence. 
 On the other hand, fibrations can be defined and studied in a straightforward manner.
 Thus a fibrational approach to the Yoneda lemma is an excellent first step for a generalization to a Yoneda lemma for simplicial spaces.
 
 As we observed in \cref{lemma:int fully faithful}, $\int_\C$ is fully faithful, however, it is not essentially surjective!
 
 \begin{defone} \label{def:discrete Groth fib}
  A functor $P: \D \to \C$ is a {\it discrete Grothendieck opfibration} over $\C$ if it is in the essential image of $\int_\C$, 
  meaning there exists a functor 
  $F: \C \to \set$ and isomorphism
  $\int_\C F \xrightarrow{ \ \cong \ } \D$ 
  over $\C$.
 \end{defone}

 Fortunately, there is also an internal characterization of discrete Grothendieck opfibrations.
 
 \begin{lemone} \label{lemma:discrete Groth fib lifting} 
  A functor $P: \D \to \C$  is a discrete Grothendieck opfibration over $\C$ if and only if for 
  any map $f: C \to C'$ in $\C$ and object $D$ in $\mathscr{D}$ such that $P(D) = C$,
  there exists a {\it unique} lift $\hat{f}:D \to D'$ such that $P(\hat{f}) = f$.
 \end{lemone}

 \begin{proof}
  Let $F: \C \to \set$ be a functor. 
  Then $\int_\C F \to \C$ satisfies the lifting condition stated in the lemma. Indeed, for a morphism $f: C \to C'$ 
  and a lift $(C,x)$ where $x \in F(C)$, there is a unique lift given by the morphism $f: (C,x) \to (C',F(f)(x))$.
  \par 
  On the other hand let us assume that $P: \D \to \C$ satisfies the lifting condition of the lemma. 
  Then we will construct a functor $F: \C \to \set$ such that $\int_\C F \cong \D$ over $\C$. 
  \par 
  First note that the unique lifting condition implies that the fiber of $P$ over every given point $c$, $P^{-1}(c)$, 
  is a discrete category i.e. a set. Indeed left $f$ be a morphism in $\D$ such that $P(f) = \id_C$. Then 
  by the uniqueness assumption $f = \id$. 
  \par 
  Now define $F$ as follows:
  \begin{itemize}
   \item {\bf Objects:} For an object $c$ in $\C$ define $F(c) = P^{-1}(c)$
   \item {\bf Morphisms:} For a morphism $f: c \to c'$ define $F(f): F(c) \to F(c')$ as the map that takes $x \in F(c)$ to the 
   target of the unique lift of $f$ in $\D$. 
  \end{itemize}
  The standard projection $\pi: \int_\C F \to \C$ exactly recovers $P: \D \to \C$. 
 \end{proof}
 
 The previous characterization allows us give a contravariant version of Grothendieck opfibrations.
 
 \begin{defone}
  $P: \D \to \C$ is called a {\it discrete Grothendieck fibration} for 
  any map $f: C \to C'$ in $\C$ and object $D'$ in $\mathscr{D}$ such that $P(D') = C'$
  there exists a {\it unique} lift $\hat{f}:D \to D'$ such that $P(\hat{f}) = f$.
 \end{defone}
 
\begin{remone} \label{rem:grothendieck fib and opfib}
 Note it is very rare that a functor $p: \D \to \C$ is a discrete Grothendieck fibration as well as a discrete Grothendieck 
 opfibration. Concretely it only happens if $\D \cong \int_\C F$ over $\C$ where $F: \C \to \set$ takes every morphism to an isomorphism. 
\end{remone}
 
 \begin{remone}
  Discrete Grothendieck fibrations are a special case of more general {\it Grothendieck fibrations} that correspond to 
  functors valued in categories. See \cite{streicher2018fibration} for a readable introduction to general 
 Grothendieck fibrations.
 \end{remone}

 We now want to move on to the Yoneda lemma for Grothendieck fibrations, but for that we need the analogue of representable functors.
 
 \begin{exone} \label{ex:representable Grothendieck fib}
  Let us determine the category $\int_\C \Hom(c,-)$. Its objects are pairs $(d,f: c \to d)$ and a morphism 
  $(d,f: c \to d) \to (d',f':c \to d')$ is a morphism $g: d \to d'$ such that $gf = f' : c \to d'$. 
  Thus we just rediscovered the under-category $\C_{c/} \to \C$. 
 \end{exone}
 
 With the previous remarks at hand we can now phrase the first fibered version of the Yoneda Lemma:
 
 \begin{lemone} \label{lemma:hom fibration Yoneda}
  (Hom version of Yoneda for fibered categories)
  Let $P: \mathscr{D} \to \C$ be a discrete Grothendieck opfibration.
  Then the natural map
  \begin{center}
  \begin{tikzcd}[row sep=0.07in, column sep=0.5in]
   \Fun_{/\C}(\C_{C/} , \mathscr{D}) \arrow[r, "\cong"] & \Fun_{/\C}(\{\id_C\}, \D)  \arrow[r, equal] &[-0.35in] \{ C \} \underset{\C}{\times} \mathscr{D} \\
   \reb F: \C_{/C} \to \mathscr{D} \leb \arrow[r, mapsto] & F(\id_C) 
  \end{tikzcd}
 \end{center}
  is a bijection.
 \end{lemone}
 
 Similar to the previous part we also have a tensor version of the Yoneda lemma for fibered categories.
 First, however, we have to define a notion of tensor product for fibered categories.
 The analogue of the tensor product of two categories $\mathscr{D}$ and $\mathscr{E}$ over $\C$
 should be the pullback $\mathscr{D} \times_{\C} \mathscr{E}$.
 However, the pullback is not necessarily a set by itself
 and so we have to make the necessary adjustments. This means we have the following:
 $$ \mathscr{D} \underset{\C}{\otimes} \mathscr{E} = \pi_0 (\mathscr{D} \underset{\C}{\times} \mathscr{E})$$
 where $\pi_0$ is the set of connected components.
 With this definition we can state our last version of the Yoneda lemma
 
 \begin{lemone} \label{lemma:tensor fibration Yoneda}
  (Tensor version of Yoneda for fibered categories)
  Let $P: \D \to \C$ be a discrete Grothendieck opfibration.
  Then the natural map 
  \begin{center}
  \begin{tikzcd}[row sep=0.07in, column sep=0.5in]
   \C_{/C} \underset{\C}{\otimes} \mathscr{D} \arrow[r, "\cong"] & \{ C \} \underset{\C}{\times} \mathscr{D} \\
   (f: C' \to C \comma D') \arrow[r, mapsto] & Codomain(\hat{f}) 
  \end{tikzcd}
 \end{center}
 is a bijection (here $\hat{f}$ is the unique lift of $f$ with domain $D'$).
 \end{lemone}
 
 Our goal in the coming sections is to build the necessary machinery to generalize these statements to the setting of simplicial spaces. 
 In particular, we will define the correct analogue to discrete Grothendieck opfibrations, study their properties and prove the Yoneda lemma. 
  
 Concretely, we have following generalizations:
 \begin{center}
  \begin{tabular}{c|c|c}
   Statement & Category & Higher Category \\ \hline 
   Grothendieck Construction & \cref{prop:integral adjunctions} & \cref{the:simplicial groth construction} \\ \hline 
   Fibration & \cref{def:discrete Groth fib}/\cref{lemma:discrete Groth fib lifting} & \cref{def:left fibration} \\ \hline 
   Conservativity & \cref{lemma:Grothendieck fibration conservative}  & \cref{lemma:left fibration conservative} \\ \hline 
   Yoneda Lemma (Hom) & \cref{lemma:hom fibration Yoneda} & \cref{cor:yoneda lemma Segal spaces hom version} \\ \hline 
   Yoneda Lemma (Tensor) & \cref{lemma:tensor fibration Yoneda} & 
   \cref{the:condition for covar equiv general rep}/\cref{rem:yoneda left fib} 
  \end{tabular}
 \end{center}

\section{Basics \& Conventions} \label{Sec Basics Conventions}
 In this section we review some basic concepts that we will need in the coming sections. 
 In particular, we review Joyal-Tierney calculus (\cref{subsec:joyal tierney calculus}) \cite{joyaltierney2007qcatvssegal} as a powerful notational tool. 
 Moreover, we review notation for simplicial sets (\cref{Subsec Simplicial Sets}), simplicial spaces (\cref{Subsec Simplicial Spaces})
 and its associated Reedy model structure (\cref{Subsec Reedy Model Structure}) 
 along with two localizations of the Reedy model structure (\cref{Subsec Diagonal & Kan Model Structure}). 
 Finally, we use complete Segal spaces as our model of higher categories and thus will end the 
 section with a quick review following \cite{rezk2001css} (\cref{Subsec Complete Segal Spaces}). 
 
 \subsection{Joyal-Tierney Calculus} \label{subsec:joyal tierney calculus}
  As we primarily work with simplicial spaces it is helpful to first set up some notation that will simplify many statements.
  The notation introduced here is due to Joyal and Tierney \cite[Section 7]{joyaltierney2007qcatvssegal}.
 
 \begin{notone}
  For this subsection let $\C$ be a locally Cartesian closed bicomplete category.
 \end{notone}

 \begin{defone} \label{def:pushout product}
  Let $f: A \to B$ and $g: C \to D$ be two maps in $\C$. We define the pushout product 
  as the universal map out of the pushout
  $$f \square g: A \times D \coprod_{A \times C} B \times C \to B \times D$$
  induced by the commutative square
  \begin{center}
   \begin{tikzcd}[row sep=0.3in, column sep=0.3in]
    A \times C \arrow[r, "\id_A \times g"] \arrow[d, "f \times \id_C"] & A \times D \arrow[d] \arrow[ddr, bend left = 20, , "f \times id_D"] & \\
    B \times C \arrow[r] \arrow[drr, bend right = 20, "\id_B \times g"'] & \ds A \times D \coprod_{A \times C} B \times C \arrow[dr, "f \square g" description] & \\
    & & B \times D 
   \end{tikzcd}
   .
  \end{center}
   Moreover, for two sets of maps $\mathcal{A}$ and $\mathcal{B}$ we use the notation 
  $$\mathcal{A} \square \mathcal{B} = \{f \square g : f \in \mathcal{A}, g \in \mathcal{B} \}.$$
 \end{defone}
 
 \begin{defone} \label{def:pullback exponential} 
   For two maps $f: A \to B$ and $p: Y \to X$ we define the pullback exponential 
  $$\exp{f}{p}: Y^B \to Y^A \underset{X^A}{\times} X^B$$ 
  induced by the commutative square
  \begin{center}
   \begin{tikzcd}[row sep=0.3in, column sep=0.3in]
    Y^B \arrow[dr, "\exp{f}{p}" description] \arrow[drr, "Y^f", bend left =20] \arrow[ddr, "p^f"', bend right = 20] & & \\
    & Y^A \underset{X^A}{\times} X^B \arrow[r] \arrow[d] & Y^A \arrow[d, "p^A"] \\
    & X^B \arrow[r, "X^f"] & X^A
   \end{tikzcd}
   .
  \end{center}
  Moreover, for two sets of maps $\mathcal{A}$ and $\mathcal{X}$ we use the notation 
  $$\exp{\mathcal{A}}{\mathcal{X}} = \{\exp{f}{p} : f \in \mathcal{A}, p \in \mathcal{X} \}$$
 \end{defone}

  These two functors give us an adjunction of arrow categories: 
  \begin{center}
   \adjun{\C^{\reb 1 \leb }}{\C^{\reb 1 \leb}}{- \square f }{\exp{f}{-}}
   .
  \end{center}
  
  The key result about these two constructions is that they can help us better understand lifting conditions.
  
  \begin{notone}
   Let $\mathcal{L}$ and $\mathcal{R}$ be two sets of morphisms in $\C$. 
   If $\mathcal{L}$ has the left lifting property with respect to $\mathcal{R}$ then we use the notation
   $\mathcal{L} \pitchfork  \mathcal{R}$.
  \end{notone}

  \begin{propone} \label{prop:joyal tierney lifting}
   (\cite[Proposition 7.6]{joyaltierney2007qcatvssegal})
   Let $\mathcal{A}$, $\mathcal{B}$ and $\mathcal{X}$ be three sets of morphisms in $\C$. Then:
   $$\mathcal{A} \square \mathcal{B} \pitchfork \mathcal{X} \Leftrightarrow
   \mathcal{A} \pitchfork \exp{\mathcal{B}}{\mathcal{X}} \Leftrightarrow
   \mathcal{B} \pitchfork \exp{\mathcal{A}}{\mathcal{X}}.$$
  \end{propone}

 \subsection{Simplicial Sets} \label{Subsec Simplicial Sets}
 $\s$ will denote the category of simplicial sets, which we will call {\it spaces}.
 We will use the following notation with regard to spaces:
 \begin{enumerate} 
  \item $\Delta$ is the indexing category with objects posets $[n] =  \{ 0,1,...,n \} $ and mappings maps of posets.
  \item We will denote a morphism $[n] \to [m]$ by a sequence of numbers $\ordered{a_0, ... ,a_n}$, where $a_i$ is the image of $i \in [n]$.
  \item \label{item:morphism notation} $\Delta[n]$ denotes the simplicial set representing $[n]$ i.e. $\Delta[n]_k = \Hom_{\Delta}([k], [n])$. 
  \item $\partial \Delta[n]$ denotes the boundary of $\Delta[n]$ i.e. 
  the largest sub-simplicial set which does not include $id_{[n]}: [n] \to [n]$.
  Similarly $\Lambda[n]_l$ denotes the largest simplicial set in $\Delta[n]$ which does not include $l^{th}$ face.
  \item For a simplicial set $S$ we denote the face maps by $d_i: S_{n} \to S_{n-1} $ and the degeneracy maps by $s_i: S_n \to S_{n+1}$.
  \item \label{item:Jn} Let $I[l]$ be the category with $l$ objects and one unique isomorphisms between any two objects.
  Then we denote the nerve of $I[l]$ as $J[l]$. 
  It is a Kan fibrant replacement of $\Delta[l]$ and comes with an inclusion $\Delta[l] \rightarrowtail  J[l]$,
  which is a Kan equivalence.
  \item We say a space $K$ is {\it discrete} if for each $n$, $K_n = K_0$ and all simplicial operators are identity maps. 
 
 \end{enumerate}

 \subsection{Simplicial Spaces} \label{Subsec Simplicial Spaces}  
 $\ss = \Fun(\Delta^{op}, \s) $ denotes the category of simplicial spaces (bisimplicial sets). 
 We have the following basic notations with regard to simplicial spaces:
 \begin{enumerate}
  \item \label{item:embed} We embed the category of spaces inside the category of simplicial spaces as constant simplicial spaces 
  (i.e. the simplicial spaces $S$ such that $S_n = S_0$ for all $n$ and all simplicial operator maps are identities). 
  \item More generally we say a simplicial space is {\it homotopically constant} if all simplicial operator maps $X_n \to X_m$
  are equivalences (and in particular $X_n$ are all equivalent to $X_0$).
  \item On the other hand we say a simplicial space $X$ is a simplicial discrete space if for all $n$, the space $X_n$ is discrete. 
  \item For a given simplicial space $X$ we use the notation $s$, the {\it source map}, for $d_1: X_1 \to X_0$ and $t$,
  the {\it target map}, for $d_0: X_1 \to X_0$. This is motivated by thinking of simplicial sets as nerves of categories.
  \item \label{item:Fn} Denote $F(n)$ to be the simplicial discrete space defined as
  $$F(n)_k = \Hom_{\Delta}([k],[n]).$$
  \item Similar to \setref{item:morphism notation} we denote a morphism $F(n) \to F(m)$ by $\ordered{a_0,...,a_n}$. 
  \item $\partial F[n]$ denotes the boundary of $F(n)$. Similarly $L(n)_l$ 
  denotes the largest simplicial space in $F(n)$ which lacks the $l^{th}$ face.
  \item The category $\ss$ is enriched over spaces
  $$\Map_{\ss}(X,Y)_n = \Hom_{\ss}(X \times \Delta[n], Y).$$
  \item The category $\ss$ is also enriched over itself: 
  $$(Y^X)_{kn} = \Hom_{\ss}(X \times F(n) \times \Delta[l], Y).$$
  \item By the Yoneda lemma, for a simplicial space $X$ we have a bijection of spaces
  $$X_n \cong \Map_{\ss}(F(n),X).$$
  \end{enumerate}

 \subsection{Reedy Model Structure} \label{Subsec Reedy Model Structure}
 The category of simplicial spaces has a Reedy model structure \cite{reedy1974modelstructure}, which is defined as follows:
 \begin{itemize}
  \item[(F)] A map $f: Y \to X$ is a (trivial) Reedy fibration if for each $n \geq 0$ the following map of spaces is a (trivial) Kan fibration
  $$ \Map_{\ss}(F(n),Y) \to \Map_{\ss}(\partial F(n),Y) \underset{\Map_{\ss}(\partial F(n), X)}{\times} \Map_{\ss}(F(n), X).$$
  \item[(W)] A map $f:Y \to X$ is a Reedy equivalence if it is a level-wise Kan equivalence.
  \item[(C)] A map $f:Y \to X$ is a Reedy cofibration if it is an inclusion.
 \end{itemize}
 The Reedy model structure is very helpful as it enjoys many features that can help us while doing computations.
 In particular, it is {\it cofibrantly generated}, {\it simplicial} and {\it proper}. Moreover, it is also {\it compatible with Cartesian closure},
 by which we mean that if $i: A \to B$ and $j: C \to D$
    are cofibrations and $p: X \to Y$ is a fibration then the map $i \square j$ is a cofibration and $\exp{i}{p}$ is a fibration, 
    which are trivial if any of the involved maps are trivial.
 
 These properties in particular imply that we can apply Bousfield localizations to the Reedy model structure. 
 See \cref{Sec Some Facts about Model Categories} for more details.
 
 \subsection{Diagonal \& Kan Model Structure} \label{Subsec Diagonal & Kan Model Structure}
 In the coming sections we will use various localizations of simplicial spaces that are equivalent to the Kan model structure.
 Although these results have been studied before, we will make ample use of the notation and thus will do a careful review here.
 \par 
 First we need two adjunctions between spaces and simplicial spaces. 
 
 \begin{notone} \label{not:bunch of functors}
  Let 
  $$\Diag : \Delta \to \Delta \times \Delta$$
  be the diagonal functor that takes an object $[n]$ to the pair $([n],[n])$ and let 
  $$\pi_1, \pi_2: \Delta \times \Delta \to  \Delta$$
  be the projection functors that take an object $([n],[m])$ to the projections $[n]$ or $[m]$, respectively. Finally, let 
  $$i_1, i_2: \Delta \to \Delta \times \Delta$$
  be the inclusion functors that take an object $[n]$ to $([n],0)$ or $(0,[n])$, respectively.
 \end{notone}
 
 These functors give us three adjunctions
 
 \begin{center}
  \begin{tikzcd}[row sep=0.5in, column sep=0.5in]
   \ss \arrow[r, shift left =1, "\Diag^*"] & \s \arrow[l, shift left =1, "\Diag_*"] \arrow[r, "(\pi_1)^*", shift left =1]  
   & \ss \arrow[l, "(\pi_1)_*", shift left=1] \arrow[r, "(i_1)^*", shift left=1] & \s \arrow[l, "(i_1)_*", shift left=1]
  \end{tikzcd}
  .
 \end{center}

 \begin{remone} \label{rem:pi diag enriched adjunction}
  These adjunctions $(\Diag^*,\Diag_*)$, $((\pi_1)^*, (\pi_1)_*)$ and $((i_1)^*,(i_1)^*)$ are in fact enriched adjunctions, 
  meaning that we have a natural isomorphism of spaces
  $$\Map_{\s}(\Diag^*X,Y) \cong \Map_{\ss}(X,\Diag_*Y),$$
  $$\Map_{\ss}((\pi_1)^*X,Y) \cong \Map_{\s}(X,(\pi_1)_*(Y)),$$
  $$\Map_{\s}((i_1)^*X,Y) \cong \Map_{\ss}(X,(i_1)_*(Y)).$$
 \end{remone}
 
 \begin{remone} \label{rem:notation convention s in ss}
  Our notation convention, \spaceref{item:embed}, implies that for a space $K$ we denote the simplicial space $(\pi_1)^*(K)$ by $K$ as well.
 \end{remone}

 \begin{remone} \label{rem:pi equal i}
  Notice we have $\pi_1 \circ i_1 = \id$ and so $(\pi_1)^* \circ (i_1)^* = \id$ and more importantly $(\pi_1)_* = (i_1)^*$.
  Thus, we can also think of $(i_1)_*$ as the right adjoint to $(\pi_1)_*$.
 \end{remone}

 Before we move on we want to give very detailed descriptions of these functors. 
 For a simplicial space $X$ we have: 
 $$(\Diag^*X)_n = X_{nn},$$
 $$(\pi_1)_*(X)_n = (i_1)^*(X)_n = X_{0n}.$$
 Also, for a simplicial set $K$ we have:
 $$\Diag_*(K)_n = K^{\Delta^n},$$
 $$(\pi_1)^*(K)_n = K,$$
 $$(i_1)_*(K)_n = K^{n+1}.$$
 
 \begin{remone} \label{rem:pi and diag equivalent}
  The direct computation above in particular implies that for a given simplicial set $K$ there is a natural map 
  $$(\pi_1)^*K \to \Diag_*(K)$$
  that is a Reedy equivalence.
 \end{remone}

 We want to show that the category of simplicial spaces has two model structures that makes the two adjunctions above into 
 Quillen equivalences and that will play an important role in the coming sections.
 
 \begin{theone} \label{The Diagonal Model Structure}
  There is a unique, cofibrantly generated, simplicial model structure on $\ss$, 
  called the {\it diagonal model structure} and denoted by $\ss^{diag}$, with the following specifications.
  \begin{itemize}
   \item[(W)] A map $f:X \to Y$ is a weak equivalence if the diagonal map of spaces 
   $$\Diag^*(f): \Diag^*(X) \to \Diag^*(Y)$$ 
   is a Kan equivalence.
   \item[(C)] A map $f:X \to Y$ is a cofibration if it is an inclusion.
   \item[(F)] A map $f:X \to Y$ is a fibration if it satisfies the right lifting condition for trivial cofibrations.
  \end{itemize}
 In particular, an object $W$ is fibrant if and only if it is Reedy fibrant and a homotopically constant simplicial space.
 \end{theone}
 
 \begin{proof}
  Let $\mathcal{L}$ be the following set of cofibrations
  $$ \mathscr{L} = \{ \ordered{0}: F(0) \to F(n) : n \geq 0 \}.$$
  Then by \cref{The Localization for sS} there is a localized model structure on $\ss$ such that cofibrations are inclusions and fibrant objects are 
  Reedy fibrant simplicial spaces $K$ such that 
  $$K_n \cong \Map_{\ss}(F(n),K) \to \Map_{\ss}(F(0),K) \cong K_0$$
  is a Kan equivalence. 
  \par 
  In order to finish the proof we only need to prove that $f: X \to Y$ is a weak equivalence in the localized model structure if and only if  
  $\Diag^*(f)$ is a Kan equivalence. Fix a map of simplicial spaces $X \to Y$ and a fibrant simplicial space $K$. Then we have following diagram 
  of spaces:
  \begin{center}
   \begin{tikzcd}[row sep=0.5in, column sep=0.5in]
    \Map_{\ss}(Y,K) \arrow[r] & \Map_{\ss}(X,K) \\
    \Map_{\ss}(Y,(\pi_1)^*(K_0)) \arrow[r] \arrow[d, "\simeq"] \arrow[u, "\simeq"] &  \Map_{\ss}(X,(\pi_1)^*(K_0)) \arrow[d, "\simeq"] \arrow[u, "\simeq"] \\
    \Map_{\ss}(Y,\Diag_*(K_0)) \arrow[r] \arrow[d, "\cong"] &  \Map_{\ss}(X,\Diag_*(K_0)) \arrow[d, "\cong"] \\
    \Map_{\s}(\Diag^*Y,K_0) \arrow[r] & \Map_{\s}(\Diag^*Y,K_0)
   \end{tikzcd}
   .
  \end{center}
  The vertical maps in the top square are Kan equivalences because $K$ is fibrant and thus homotopically constant. 
  The vertical maps in the middle square are Kan equivalences by \cref{rem:pi and diag equivalent}. 
  The vertical maps in the bottom square are bijections because of the adjunction $(\Diag^*,\Diag_*)$. 
  \par 
  Now the map $f: X \to Y$ is an equivalence in this localized model structure if and only if the top horizontal map is a Kan equivalence 
  for all fibrant objects $K$
  (as the model structure is simplicial), which by the diagram above is equivalent to the bottom map being a Kan equivalence 
  for all Kan complexes $K_0$. 
  \par 
  The fact that $K_0$ is an arbitrary Kan complex and that the Kan model structure is simplicial implies that this is equivalent to 
  $$\Diag^*f: \Diag^*X \to \Diag^*Y$$
  being a Kan equivalence, which is exactly the desired statement and finishes the proof. 
 \end{proof}
 
 \begin{theone} \label{The Kan Model Structure}
  There is a unique, cofibrantly generated, simplicial model structure on $\ss$,
  called the {\it Kan model structure} and denoted by $\ss^{Kan}$, with the following specification.
  \begin{itemize}
   \item[(W)] A map $f: X \to Y$ is a weak equivalence if 
   $$(\pi_1)_*f : (\pi_1)_*X \to (\pi_1)_*Y$$ 
   is a Kan equivalence.
   \item[(C)] A map $f: X \to Y$ is a cofibration if it is an inclusion.
   \item[(F)] A map $f: X \to Y$ is a fibration if it satisfies the right lifting condition for trivial cofibrations.
  \end{itemize}
  In particular, an object $W$ is fibrant if and only if it is Reedy fibrant and the map 
  $$\Map_{\ss}(F(n),W) \to \Map_{\ss}(\partial F(n), W)$$
  is a trivial Kan fibration for $n >0$.
 \end{theone}
 
 \begin{proof}
  Let $\mathcal{L}$ be the following set of cofibrations
  $$ \mathscr{L} = \{ \partial F(n) \to F(n) : n \geq 1 \}.$$
  Then by \cref{The Localization for sS} there is a localized model structure on $\ss$ such that cofibrations are inclusions and fibrant objects are 
  Reedy fibrant simplicial spaces $K$ such that 
  $$ \Map(F(n),K) \to \Map(\partial F(n),K)$$
  is a Kan equivalence. 
  \par 
  Before we can finish the proof, we need a better understanding of the fibrant objects in this model structure. 
  A Reedy fibrant simplicial space $X$ is fibrant if and only if the map $X_n \to (X_0)^{n+1}$ is an equivalence.  
  Thus a Reedy fibrant simplicial space $X$ is fibrant in this model structure if and 
  only if the natural map $X \to (i_1)_*(X_0)$ is a Reedy equivalence.
  \par 
  In order to finish the proof we only need to prove that $f: X \to Y$ is a weak equivalence in this localized model structure if and only if  
  $(\pi_1)_*(f)$ is a Kan equivalence.
  Fix a map of simplicial spaces $X \to Y$ and a fibrant simplicial space $Z$. Then we have following diagram 
  of spaces:
  \begin{center}
   \begin{tikzcd}[row sep=0.5in, column sep=0.5in]
    \Map_{\ss}(Y,Z) \arrow[r] \arrow[d, "\simeq"] & \Map_{\ss}(X,Z) \arrow[d, "\simeq"] \\
    \Map_{\ss}(Y,(i_1)_*(Z_0)) \arrow[r] \arrow[d, "\cong"]  &  \Map_{\ss}(X,(i_1)_*(Z_0)) \arrow[d, "\cong"]  \\
    \Map_{\s}((\pi_1)_*Y,Z_0) \arrow[r] &  \Map_{\s}((\pi_1)_*X,Z_0) 
   \end{tikzcd}
   .
  \end{center}
  The vertical maps in top square are Kan equivalences as $Z$ is fibrant and so by the explanation in the previous paragraph 
  $Z \to (i_1)_*(Z_0)$ is an equivalence.
  The vertical maps in the bottom square are bijections because $(i_1)_*$ is the right adjoint to $(\pi_1)_*$ 
  as explained in \cref{rem:pi equal i}. 
  \par 
  Now the map $f: X \to Y$ is an equivalence in this localized model structure if and only if the top horizontal map is a Kan equivalence 
  for all fibrant objects $Z$
  (as the model structure is simplicial), which by the diagram above is equivalent to the bottom map being a Kan equivalence 
  for all Kan complexes $Z_0$. 
  \par 
  However $Z_0$ is an arbitrary Kan complex and the Kan model structure is simplicial. Hence this is equivalent to 
  $$(\pi_1)_*f: (\pi_1)_*X \to (\pi_1)_*Y$$
  being a Kan equivalence. 
 \end{proof}

 We are finally in a position to prove the existence of the chain of Quillen equivalences. 

 \begin{theone} \label{The Quillen equivalence of Kan model structures}
  The following three simplicially enriched adjunctions
  \begin{center}
  \begin{tikzcd}[row sep=0.5in, column sep=0.5in]
   \ss^{diag} \arrow[r, shift left =1, "\Diag^*"] & \s^{Kan} \arrow[l, shift left =1, "\Diag_*"] \arrow[r, "(\pi_1)^*", shift left =1]  
   & \ss^{Kan} \arrow[l, "(\pi_1)_*", shift left=1] \arrow[r, "(i_1)^*", shift left=1] & \s^{Kan} \arrow[l, "(i_1)_*", shift left=1] 
  \end{tikzcd}
 \end{center}
 are Quillen equivalences.
 \end{theone}

 \begin{proof}
  {\bf Quillen Adjunctions:}
  The fact that these three adjunctions are Quillen adjunctions is similar for all three cases and so we will combine the argument. 
  In all three cases we observe that the left adjoint $\Diag^*, (\pi_1)^*, (i_1)^*$ preserve inclusions and weak equivalences 
  and thus preserves cofibrations and trivial cofibrations, which imply they are left Quillen functors. 
  
  Indeed, the fact that they preserve inclusions is immediate. The fact that they preserve weak equivalences is an immediate observation 
  for $(\pi_1)^*$ and $(i_1)^*$ and for $\Diag^*$ follows from \cref{The Diagonal Model Structure}.
  
  {\bf Quillen Equivalence:}
  We move on to prove that these are Quillen equivalences. First we prove that $(\pi_1)^*$ and $(i_1)^*$ are derived inverses of each other.
  Notice that for every simplicial set $K$, $(i_1)^*(\pi_1)^*(K)=K$. So, in order to prove they are derived inverses it suffices to prove 
  there is a natural equivalence:
  $$(\pi_1)^*(i_1)^*(X) \xrightarrow{ \ \simeq \ } X$$
  This will then imply that $((\pi_1)^*, (\pi_1)_*)$ and $((i_1)^*, (i_1)_*)$ are Quillen equivalences and in fact are inverses of each other. 
  \par 
  We have $(\pi_1)^*(i_1)^*(X)_n = X_0$ and so 
  the natural map $(\pi_1)^*(i_1)^*(X) \to X$ is an equivalence in the Kan model structure on simplicial spaces. 
  \par 
  Next we move on to prove that $(\Diag^*,\Diag_*)$ is a Quillen adjunction. Based on \cref{Lemma For Quillen equiv}
  it suffices to prove that $\Diag^*$ reflects weak equivalences and the derived counit map $\Diag^*\Diag_*K \to K$ is a 
  Kan equivalence of simplicial sets. 
  \par 
  The fact that $\Diag^*$ reflects weak equivalences is part of \cref{The Diagonal Model Structure}. For the second part 
  we first observe that 
  $$(\Diag^* \Diag_*K)_n = (\Diag_*K)_{nn} = \Hom(\Delta[n] \times \Delta[n],K)$$
  and notice the counit map
  $$\Hom(\Delta[n] \times \Delta[n],K) \to \Hom(\Delta[n],K)$$
  is induced by the diagonal map $\Delta[n] \to \Delta[n] \times \Delta[n]$. 
  The diagonal map is a Kan equivalence and so the derived counit map is an equivalence.
\end{proof}
 
 \begin{remone}
  Composing the two Quillen equivalences $(\Diag^*,\Diag_*)$ and $((\pi_1)^*, (\pi_1)_*)$ we get a Quillen equivalence 
   \begin{center}
  \begin{tikzcd}[row sep=0.5in, column sep=0.5in]
  \ss^{diag} \arrow[r, shift left =1, "(\pi_1)^*\Diag^*"] & \ss^{Kan} \arrow[l, shift left =1, "\Diag_*(\pi_1)_*"]
   \end{tikzcd}
 \end{center}
 however, this Quillen equivalence is not the identity adjunction. Thus the Kan model structure and diagonal model structure on 
 simplicial spaces are Quillen equivalent, but not the same model structures.
 \end{remone}

 \subsection{Complete Segal Spaces} \label{Subsec Complete Segal Spaces}
 The Reedy model structure can be localized such that it models $(\infty,1)$-categories \cite{rezk2001css}.
 This is done in two steps.
 First we define {\it Segal spaces}.
 
 \begin{defone} \label{Def Segal Spaces}
 \cite[Page 11]{rezk2001css}
  A Reedy fibrant simplicial space $X$ is called a Segal space if the map
  $$ X_n \xrightarrow{\ \ \simeq \ \ } X_1 \underset{X_0}{\times} ... \underset{X_0}{\times} X_1 $$
  is a Kan equivalence for $n \geq 2$.
 \end{defone}
 Segal spaces come with a model structure, namely the {\it Segal space model structure}.
 
 \begin{theone} \label{The Segal Space Model Structure}
  \cite[Theorem 7.1]{rezk2001css}
  There is a simplicial closed model category structure on the category $\ss$ of simplicial spaces
  called the Segal space model category structure, and denoted $\ss^{Seg}$, with the following properties.
  \begin{enumerate}
   \item The cofibrations are precisely the monomorphisms.
   \item The fibrant objects are precisely the Segal spaces.
   \item The weak equivalences are precisely the maps $f$ such that $Map_{\ss}(f, W)$ is
   a weak equivalence of spaces for every Segal space $W$.
   \item A Reedy weak equivalence between any two objects is a weak equivalence in
   the Segal space model category structure, and if both objects are themselves
   Segal spaces then the converse holds.
   \item For two cofibrations $i$ and $j$, $i \square j$ is a cofibration, which is trivial if either of $i$ or $j$ are.
   \item The model structure is the localization of the Reedy model structure with respect to the maps
   $$ G(n) = F(1) \coprod_{F(0)} ... \coprod_{F(0)} F(1) \to F(n)$$
   for $n \geq 2$.
  \end{enumerate}
 \end{theone}
 A Segal space already has many characteristics of a category, such as objects and morphisms.

 \begin{defone} \label{def:mapping space Segal space}
 Let $W$ be a Segal space. Then an {\it object} $x$ in $W$ is a point in $W_0$.
 Moreover for two objects $x,y$ we define the {\it mapping space} as the pullback
  \begin{center}
   \begin{tikzcd}[row sep=0.5in, column sep=0.5in]
    \map_W(x,y) \arrow[r] \arrow[d] & W_1 \arrow[d] \\
    \Delta[0] \arrow[r, "(\{x\} \comma \{y\})"] & W_0 \times W_0
   \end{tikzcd}
  \end{center}
 \end{defone}
  
 Unlike classical categories, the mapping spaces of a Segal space do not come with a strict composition maps. 
 Rather there is a natural zig-zag. For more details see \cite[Section 5]{rezk2001css}.
 On the other hand we do get an actual homotopy category:
 
 \begin{defone} \label{def:homotopy category}
  Let $W$ be a Segal space. We define the {\it homotopy category} of $W$, denoted $\Ho W$, as the following category:
  \begin{enumerate}
   \item Objects of $\Ho W$ are objects of $W$
   \item For two objects $x,y$ 
   $$\Hom_{\Ho W}(x,y) = \pi_0(\map_W(x,y))$$
  \end{enumerate}
 \end{defone}
 This indeed gives us a category \cite[5.5]{rezk2001css}. Moreover, a morphism in $W$ is a weak equivalence if the morphism $[f]$ in $\Ho W$ is 
 an isomorphism.
 \par 
 Segal spaces do not give us a model of $(\infty,1)$-categories. For that we need {\it complete Segal spaces}.
 
 \begin{defone} \label{Def Of the E spaces}
  Let $J[n]$ be the fibrant replacement of $\Delta[n]$ in the Kan model structure \setref{item:Jn}.  
  We define the simplicial discrete space $E(n)$ as
  $E(n) = (\pi_2)^*J[n].$
  where $(\pi_2)^*$ was defined in \cref{not:bunch of functors}.
  In particular, $E(1)$ is the free invertible arrow.
 \end{defone}
 
 \begin{defone} \label{def:w hoequiv}
  Let $W$ be a Segal space. We define the space of weak equivalences $W_{hoequiv}$ as 
  $$W_{hoequiv} = \Map_{\s}(E(1),W).$$
  Notice the map $W_{hoequiv} \to W_1$ is an inclusion of path-components \cite[Lemma 5.8]{rezk2001css}.
 \end{defone}
 
 \begin{defone} \label{Def Complete Segal Spaces}
  A Segal space $W$ is called a {\it complete Segal space} if it satisfies one of the following equivalent conditions.
  \begin{enumerate}
   \item The inclusion map 
   $$W_0 \hookrightarrow W_{hoequiv}$$
   is a weak equivalence.
   \item The map 
    $$\ordered{0}^*: W_{hoequiv}= Map(E(1),W) \to Map(F(0),W) = W_0$$
    is a trivial Kan fibration.
    \item The map 
    $$\ordered{1}^*: W_{hoequiv}= Map(E(1),W) \to Map(F(0),W) = W_0$$
    is a trivial Kan fibration.
  \end{enumerate}
 \end{defone}

 Complete Segal spaces come with their own model structure, the {\it complete Segal space model structure}.
 
 \begin{theone} \label{The Complete Segal Space Model Structure}
  \cite[Theorem 7.2]{rezk2001css}
  There is a simplicial closed model category structure on the category $\ss$ of simplicial spaces, 
  called the complete Segal space model category structure, and denoted $\ss^{CSS}$, with the following properties.
  \begin{enumerate}
   \item The cofibrations are precisely the monomorphisms.
   \item The fibrant objects are precisely the complete Segal spaces.
   \item The weak equivalences are precisely the maps $f$ such that $Map_{\ss}(f, W)$ is
   a weak equivalence of spaces for every complete Segal space $W$.
   \item A Reedy weak equivalence between any two objects is a weak equivalence in
   the complete Segal space model category structure, and if both objects are themselves
   complete Segal spaces then the converse holds.
   \item For two cofibrations $i$ and $j$, $i \square j$ is a cofibration, which is trivial if either of $i$ or $j$ are. 
   \item The model structure is the localization of the Segal space model structure with respect to the map
   $$\ordered{0}: F(0) \to E(1).$$
  \end{enumerate}
 \end{theone}
 
 A complete Segal space is a model for an $(\infty,1)$-category.
 For a better understanding of complete Segal spaces see \cite[6]{rezk2001css} and for a comparison with other models see 
 \cite{joyaltierney2007qcatvssegal}, \cite{bergner2010survey}.
 
\section{Left Fibrations and the Covariant Model Structure} \label{Sec Left Fibrations and the Covariant Model Structure}
This section is focused on the study of left fibrations. We first focus on various characterizations of 
left fibrations (\cref{subsec:many faces left fib}). Then we show that left fibrations can be seen as fibrant objects in a model structure, 
the covariant model structure (\cref{subsec:covariant model structure}). Finally, we do a careful analysis of left fibrations over Segal spaces 
(\cref{subsec:yoneda lemma Segal spaces}).

 \begin{remone}
 {\it Historical note on left fibrations for simplicial spaces:} 
 Left fibrations for complete Segal spaces were first considered by Charles Rezk, 
 motivated by his paper on complete Segal spaces \cite{rezk2001css}, 
 however, he never published those ideas. 
 \par 
 The first record of left fibrations for Segal spaces can be found in the work of de Brito \cite[0.1]{debrito2018leftfibration} and 
 Kazhdan-Varshavsky \cite[Definition 2.1.1]{kazhdanvarshvsky2014yoneda}, 
 where both authors (independently) give the same definition of a left fibration for Segal spaces. 
 \par 
 The definition of left fibration given here was suggested to the author by Charles Rezk and 
 generalizes those definitions from Segal spaces to an arbitrary simplicial space.
 \end{remone}

\subsection{The Many Faces of Left Fibrations} \label{subsec:many faces left fib}
 In this section we first define left fibration (\cref{def:left fibration}) and then prove it can be characterized in several alternative ways:
 \cref{lemma:left fibration triv Kan fib}, 
 \cref{Lemma Other definition of Left Fibration}, 
 \cref{prop:joyal tierney left fibrations}, 
 \cref{Lemma Local def of Left Fib}. 
 There is one final characterization of left fibrations, \cref{lemma:left fib as exp}, which we relegate to the next section.
 \medskip 
 
 We want to generalize the definition of a discrete Grothendieck opfibration (\cref{def:discrete Groth fib}) to simplicial spaces. 
 The guiding principle towards a working definition is the following idea:
 
 \medskip
 \begin{center}
  {\it uniqueness in set theory $\Leftrightarrow$ contractibility in homotopy theory}
 \end{center}
 \medskip
 
 Thus we need to find an appropriate contractibility condition. 
 Using our intuition from Segal spaces for a given simplicial space $X$ we should think of the 
 space $X_0$ as the {\it space of objects}, $X_1$ as the {\it space of morphisms} and the simplicial map $s: X_1 \to X_0$ 
 as map that takes a morphism to its source. 
 \par 
 This motivates following definition:
 
 \begin{defone} \label{def:left fibration}
  A map of simplicial spaces $p: L \to X$ is called a left fibration if it is a Reedy fibration such that 
  the following square is a homotopy pullback square
  \begin{equation} \label{eq:left fib pullback}
   \pbsq{L_n}{L_0}{X_n}{X_0}{\ordered{0}^*}{p_n}{p_0}{\ordered{0}^*}
  \end{equation}
 where the horizontal maps come from the map $\ordered{0}: [0] \to [n]$ taking the point to $0 \in [n]$.
 \end{defone}
 
  \begin{remone} \label{Rem Left fibrations like functors} 
   As we observed in \cref{lemma:discrete Groth fib lifting}, Grothendieck opfibrations correspond to functors $\C \to \set$.
   Given that left fibrations are the homotopical analogue of Grothendieck opfibrations they are expected to model 
   functors valued in spaces. We will in fact prove this statement for the specific case where $X= N\C$ 
   (\cref{the:simplicial groth construction}) and use this idea as a guide towards studying left fibrations over general simplicial spaces.
   \par 
   The proof of the general case can be found for quasi-categories in \cite[Chapter 2]{lurie2009htt}, \cite{heutsmoerdijk2016leftfibrationii}.  
  \end{remone}
 
 Let us start with a simple, alternative way of characterizing left fibrations.
 
 \begin{lemone} \label{lemma:left fibration triv Kan fib}
  Let $p: Y \to X$ be a Reedy fibration. The following are equivalent:
  \begin{enumerate}
   \item $p$ is a left fibration.
   \item For each $n \geq 0$, the map 
   $$(p_n , \ordered{0}^*): Y_n \to X_n \underset{X_0}{\times} Y_0$$
   is a Kan equivalence.
   \item For each $n \geq 0$, the map 
   $$(p_n , \ordered{0}^*): Y_n \twoheadrightarrow X_n \underset{X_0}{\times} Y_0$$
   is a trivial Kan fibration.
  \end{enumerate}
 \end{lemone}
 
 \begin{proof}
  $(1 \Leftrightarrow 2)$
  This follows from the definition of a homotopy pullback.
  
  \medskip 
  
  $(2 \Leftrightarrow 3)$
  The map $\ordered{0}: F(0) \to F(n)$ is a cofibration, which implies that 
  $$\Map_{\ss}(F(n), Y) \to \Map_{\ss}(F(0), Y) \underset{\Map_{\ss}(F(0), X)}{\times} \Map_{\ss}(F(n), X)$$
  or, equivalently, the map
  $$Y_n \to X_n \underset{X_0}{\times} Y_0$$
  is a Kan fibration.
  Thus it is a weak equivalence if and only if it is a trivial Kan fibration.
 \end{proof}
 
 Next we give an alternative, inductive characterization of left fibration. 
 
 \begin{lemone} \label{Lemma Other definition of Left Fibration}
  Let $p: Y \to X$ be a Reedy fibration. The following two are equivalent:
  \begin{enumerate}
   \item The commutative square
   \begin{center}
    \pbsq{Y_n}{Y_0}{X_n}{X_0}{\ordered{0}^*}{p_n}{p_0}{\ordered{0}^*}
   \end{center}
    is a homotopy pullback square for all $n \geq 0$, meaning $p$ is a left fibration.
   \item The commutative square 
   \begin{center}
    \pbsq{Y_n}{Y_{n-1}}{X_n}{X_{n-1}}{\ordered{0,...,n-1}^*}{p_n}{p_{n-1}}{\ordered{0,...,n-1}^*}
   \end{center}
   is a homotopy pullback square for all $n \geq 0$.  
  \end{enumerate}
 \end{lemone}
 \begin{proof}
  We have the following diagram:
   \begin{center}
    \begin{tikzcd}[row sep = 0.5in, column sep = 0.5in]
     Y_n \arrow[r, "\ordered{0,...,n-1}^*"] \arrow[d, "p_n"] & Y_{n-1} \arrow[r, "\ordered{0}^*"] \arrow[d, "p_{n-1}"] & Y_0 \arrow[d, "p_0"] \\
     X_n \arrow[r, "\ordered{0,...,n-1}^*"] & X_{n-1} \arrow[r, "\ordered{0}^*"] & X_0 
    \end{tikzcd}
    .
   \end{center}
  $( 1 \Rightarrow 2)$ In this case the rectangle and the right square is a homotopy pullback and therefore the left hand square is also a 
  homotopy pullback.
  
  \medskip 
  
  $(2 \Rightarrow 1)$ For this case we use induction. The case $n=1$ is clear. If it is true for $n-1$ then this means that in the diagram above
  the right hand square is a homotopy pullback. By assumption the left hand square is a 
  homotopy pullback and so the whole rectangle has to be a homotopy pullback and we are done.
 \end{proof}
 
 Next we give a characterization of left fibrations via lifting conditions.
 
 \begin{propone} \label{prop:joyal tierney left fibrations}
  Let $p: Y \to X$ be a Reedy fibration. Then the following are equivalent:
  \begin{enumerate}
   \item $p$ is a left fibration.
   \item $p$ satisfies the right lifting property with respect to maps of the form 
   $$(\ordered{0}: F(0) \to F(n)) \square (\partial \Delta[l] \to \Delta[l]).$$
   \item The map of simplicial sets $(\pi_2)^*(\exp{\ordered{0}:F(0) \to F(n)}{p})$ (where $\pi_2$ was introduced in \cref{not:bunch of functors}) 
   satisfies the right lifting property with respect to maps of the form 
   $$\ordered{0}: \Delta[0] \to \Delta[l].$$ 
  \end{enumerate}
 \end{propone}
 
 \begin{proof}
   By \cref{lemma:left fibration triv Kan fib},
   the map $p: Y \to X$ is a left fibration if and only if
   $$(Y_n \twoheadrightarrow X_n \times_{X_0} Y_0) = (\pi_1)_*(\exp{\ordered{0}:F(0) \to F(n)}{p: Y \to X})$$ 
   is a trivial fibration, where $\pi_1$ was introduced in \cref{not:bunch of functors}. 
   This is equivalent to $(\pi_1)_*(\exp{\ordered{0}:F(0) \to F(n)}{p: Y \to X})$ having the right lifting property with 
   respect to the inclusion maps $\partial \Delta[l] \to \Delta[l]$. Using the adjunction $((\pi_1)^*, (\pi_1)_*)$, this is equivalent to  
   $\exp{\ordered{0}:F(0) \to F(n)}{p: Y \to X}$ having the right lifting property with respect to
   $(\pi_1)^*(\partial \Delta[l]) \to (\pi_1)^*\Delta[l]$, which by our notation convention (\cref{rem:notation convention s in ss}) 
   we denote by $\partial \Delta[l] \to \Delta[l]$.
   \par 
   Thus, $p:Y \to X$ is a left fibration is and only if 
   $$\partial \Delta[l] \to \Delta[l] \pitchfork \exp{\ordered{0}: F(0) \to F(n)}{ p: Y \to X}.$$  
   Now using \cref{prop:joyal tierney lifting} with the set of morphisms:
   \begin{itemize}
    \item $\mathcal{A} = \{ \ordered{0}: F(0) \to F(n) : n \geq 0 \} $
    \item $\mathcal{B} = \{ \partial \Delta[l] \to \Delta[l] : l \geq 0 \}$
    \item $\mathcal{L} = \{ \text{left fibrations} \}$
   \end{itemize}
   we have 
   $$ \mathcal{A} \pitchfork \exp{\mathcal{B}}{\mathcal{L}} \Leftrightarrow \mathcal{A} \square \mathcal{B} \pitchfork \mathcal{L} \Leftrightarrow
   \mathcal{B} \pitchfork \exp{\mathcal{A}}{\mathcal{L}}.$$
   Hence we are done.
 \end{proof}

 The fact that left fibrations are given via a right lifting property has several formal consequences.

  \begin{lemone} \label{Lemma Composition preserves Left fibrations}
  Let $f:Y \to X$ and $g:Z \to Y$ be two Reedy fibrations.
  \begin{enumerate}
   \item If $f$ and $g$ are left fibrations then $fg$ is also a left fibration.
   \item If $f$ and $fg$ are left fibrations then $g$ is also a left fibration.
  \end{enumerate}
 \end{lemone}
 
  \begin{lemone} \label{Lemma Pullback preserves Left fibrations}
  The pullback of a left fibration is a left fibration.
 \end{lemone}
  
 We can use the pullback stability of left fibrations to give several local characterizations of left fibrations.
 
 \begin{lemone} \label{Lemma Local def of Left Fib} 
  Let $p: L \to X$ be a Reedy fibration. Then the following are equivalent:
  \begin{enumerate}
   \item $p$ is a left fibration.
   \item For every map $F(n) \times \Delta[l] \to X$ the pullback $L \times_X (F(n) \times \Delta[l]) \to F(n) \times \Delta[l]$ is a 
   left fibration.
   \item For every map $F(n) \to X$ the pullback $L \times_X F(n) \to F(n)$ is a left fibration.
  \end{enumerate}
 \end{lemone}
 
 \begin{proof}
  $(1 \Rightarrow 3)$ 
  This follows from \cref{Lemma Pullback preserves Left fibrations}.
  
  \medskip 
  
  $(2 \Rightarrow 1)$
  By \cref{prop:joyal tierney left fibrations} it suffices to prove that every map has the right lifting property with respect to 
  maps $(F(0) \to F(n)) \square (\partial \Delta[l] \to \Delta[l])$. Now we have following diagram 
  \begin{center}
   \begin{tikzcd}[row sep=0.5in, column sep=0.5in]
    \ds F(n) \times \partial \Delta[l] \coprod_{F(0) \times \partial \Delta[l]} \Delta[l] \arrow[d] \arrow[r]  & 
    f^*L \arrow[r] \arrow[d] & L \arrow[d, "p"] \\ 
    F(n) \times \Delta[l] \arrow[r, "\id"] \arrow[ur, dashed] & F(n) \times \Delta[l] \arrow[r, "f"] & X
   \end{tikzcd}
   .
  \end{center}
  By assumption $f^*L \to F(n) \times \Delta[l]$ is a left fibration, which means the lift exists. 
  
  \medskip 
  
  $(3 \Rightarrow 2)$
  Fix a map $f: F(n) \times \Delta[l] \to X$. Then $f^*p: L \to F(n) \times \Delta[l]$ is a Reedy fibration and so in order to 
  show it is a left fibration we need to show it satisfies the homotopy pullback condition in \ref{eq:left fib pullback}. 
  \par 
  The map $\ordered{0}: \Delta[0] \to \Delta[l]$ gives us following pullback square
  \begin{center}
   \begin{tikzcd}[row sep=0.5in, column sep=0.5in]
    (f \circ (id \times 0))^*(L) \arrow[r, "\simeq"] \arrow[d] & f^*L \arrow[d] \\
    F(n) \arrow[r, "id \times \ordered{0}", "\simeq"'] & F(n) \times \Delta[l] 
   \end{tikzcd}
   .
  \end{center}
  The bottom map is a Reedy equivalence, which implies the top map is also a Reedy equivalence. But by assumption 
  $(f \circ (id \times \ordered{0}))^*(L)$ is a left fibration. This implies that $f^*L$ is also a left fibration. 
 \end{proof}
  
 Using the same argument as in the previous proof we can prove the analogous statement about diagonal fibrations
 that will become useful later on.
 
 \begin{corone} \label{cor:local def of diag Fib}
   Let $p: L \to X$ be a Reedy fibration. Then the following are equivalent:
  \begin{enumerate}
   \item $p$ is a diagonal fibration.
   \item For every map $F(n) \times \Delta[l] \to X$ the pullback $L \times_X (F(n) \times \Delta[l]) \to F(n) \times \Delta[l]$ is a 
   diagonal fibration.
   \item For every map $F(n) \to X$ the pullback $L \times_X F(n) \to F(n)$ is a diagonal fibration.
  \end{enumerate}
 \end{corone}

\subsection{The Covariant Model Structure} \label{subsec:covariant model structure}
 Let $X$ be a simplicial space. In this subsection we define a model structure on the over-category $\ss_{/X}$,
 the {\it covariant model structure}, which has fibrant objects left fibrations over $X$ (\cref{The Covariant Model Structure}).
 We end the subsection by giving a useful criterion for determining covariant equivalences, generalizing deformation retracts from 
 classical homotopy theory (\cref{the:def retract covariant equiv}).
 
 \begin{theone} \label{The Covariant Model Structure}
  Let $X$ be a simplicial space. 
  There is a unique simplicial left proper model structure on the over-category $\ss_{/X}$, called the covariant model structure
  and denoted by $(\ss_{/X})^{cov}$, which satisfies the following conditions:
  \begin{enumerate}
   \item The fibrant objects are the left fibrations over $X$.
   \item Cofibrations are monomorphisms.
   \item A map $f: A \to B$ over $X$ is a weak equivalence if 
   $$\Map_{\ss_{/X}}(B,L) \to \Map_{\ss_{/X}}(A,L)$$ 
   is a Kan equivalence for every left fibration $L \to X$.
   \item A weak equivalence (covariant fibration) between fibrant objects is a level-wise equivalence (Reedy fibration). 
  \end{enumerate}
 \end{theone}

 \begin{proof}
  Let $\mathcal{L}$ be the collection of maps of the following form
  $$ \mathcal{L} = \{  F(0) \xrightarrow{ \ \ordered{0} \ } F(n) \to X \}.$$
  Note that $\mathcal{L}$ is a set of cofibrations in $\ss_{/X}$ with the Reedy model structure.
  This allows us to use the theory of Bousfield localizations (\cref{The Localization for sS}) 
  with respect to $\mathcal{L}$ on the category $\ss_{/X}$.
  The resulting model structure immediately satisfies all the conditions above and in particular 
  the fibrant objects are precisely the left fibrations by \cref{lemma:fibrant objects in localized model structure}.
 \end{proof}
 
 \begin{remone}
  This model structure is also constructed in \cite[Proposition 1.10]{debrito2018leftfibration} for the particular 
  case where the base $X$ is a Segal space. 
 \end{remone}

 We can actually say more about the fibrations in the covariant model structure. 
 
 \begin{lemone} \label{lemma:fib between left fib}
  Let $p: L \to X$ and $q: L' \to X$ be two left fibrations. A map $f: L \to L'$ over $X$ is a fibration in the covariant 
  model structure if and only if it is a left fibration. 
 \end{lemone}

 \begin{proof}
  As $p$ and $q$ are fibrant, $f$ is a fibration if and only if it is a Reedy fibration (\cref{The Covariant Model Structure}). 
  The statement now follows from \cref{Lemma Composition preserves Left fibrations} as $qf = p$.
 \end{proof}

 Note the covariant model structure behaves well with respect to base change.
 
 \begin{theone} \label{The Base change covar}
  Let $f: X \to Y$ be map of simplicial spaces. Then the following adjunction 
  \begin{center}
   \adjun{(\ss_{/X})^{cov}}{(\ss_{/Y})^{cov}}{f_!}{f^*}
  \end{center}
  is a Quillen adjunction, which is a Quillen equivalence if $f$ is a Reedy equivalence.
  Here $f_!$ is the composition map and $f^*$ is the pullback map.
 \end{theone}

 \begin{proof}
  This is the special case of \cref{lemma:localized model structure base change} when 
  $\mathcal{L} = \{ \ordered{0}: F(0) \to F(n) : n \geq 0 \}$.
 \end{proof}
 
 \begin{remone}
  Later we will prove a much stronger result, namely if $f$ is an equivalence in the CSS model structure then the Quillen adjunction is actually 
  a Quillen equivalence (\cref{The Covar invariant under CSS equiv}).
 \end{remone}
 
 \begin{theone} \label{The Diag is local of Covar}
  The following is a Quillen adjunction
  \begin{center}
   \adjun{(\ss_{/X})^{cov}}{(\ss_{/X})^{diag}}{id}{id}
  \end{center}
  which is a Quillen equivalence if $X$ is a homotopically constant simplicial space.
  Here the left side has the covariant model structure and the right side has the induced diagonal model structure
  (\cref{prop:induced model structure}).
  This implies that the diagonal model structure over $X$ is a localization of the covariant model structure over $X$.
 \end{theone}
 
 \begin{proof}
  If we localize the Reedy model structure on $\ss_{/X}$ with respect to maps of the form $F(0) \to F(n) \to X$ we 
  get the covariant model structure (\cref{The Covariant Model Structure}) whereas if we localize the Reedy model structure 
  on $\ss$ with respect to maps $\ordered{0}: F(0) \to F(n)$ we get the diagonal model structure (\cref{The Diagonal Model Structure}).
  This means we can apply \cref{the:induced vs localized} to deduce that this a Quillen adjunction.
  \par 
  Now let us assume $X$ is homotopically constant. Let $j: X \to \hat{X}$ be a Reedy fibrant replacement of $X$.
  Then $\hat{X}$ is also homotopically constant, which means it is fibrant in the diagonal model structure (\cref{The Diagonal Model Structure}).
  We now have following diagram of Quillen adjunctions 
  \begin{center}
   \begin{tikzcd}[row sep=0.5in, column sep=0.5in]
    (\ss_{/X})^{cov} \arrow[r, shift left = 1.8, "id", "\bot"'] \arrow[d, shift left = 1.8, "j_!", "\rotatebot"'] & 
    (\ss_{/X})^{diag} \arrow[l, shift left = 1.8, "id"] \arrow[d, shift left = 1.8, "j_!", "\rotatebot"'] \\
    (\ss_{/\hat{X}})^{cov} \arrow[r, shift left = 1.8, "id", "\bot"'] \arrow[u, shift left = 1.8, "j^*"] & 
    (\ss_{/\hat{X}})^{diag} \arrow[l, shift left = 1.8, "id"] \arrow[u, shift left = 1.8, "j^*"] 
   \end{tikzcd}
  \end{center}
 The  vertical adjunctions are Quillen equivalences as $j$ is a Reedy equivalence by \cref{lemma:localized model structure base change}.
 The bottom horizontal adjunction is a Quillen equivalence because $\hat{X}$ is fibrant in the diagonal model structure and \cref{the:induced vs localized}.
 Thus, $2$-out-of-$3$ implies that the top adjunction is a Quillen equivalence as well.
 \end{proof}

 \begin{remone} 
  The theorem implies that every covariant equivalence is a diagonal equivalence, whereas the opposite direction is obviously not true. 
  On the other hand, in \cref{Subsec The Yoneda Lemma} we will prove that
  we can determine whether a map is a covariant equivalence by checking whether a certain collection of maps is a diagonal equivalence 
  (\cref{the:covar equiv over simp space}).
 \end{remone}

 \begin{exone} \label{Ex Covariant over the point}
  One very important instance is the case $X=F(0)$.
  The theorem shows that $\ss^{cov}$ is the same as $\ss^{diag}$.
 \end{exone}
 
 Using the covariant model structure we can one final alternative characterization of left fibrations.
 
 \begin{lemone} \label{lemma:left fib as exp}
   Let $p: L \to X$ be a Reedy fibration. Then the following are equivalent:
   \begin{enumerate}
    \item $p$ is a left fibration.
    \item For all $n \geq 0$
    \begin{equation} \label{eq:some Kan equivalence}
     \comsq{\Map_{\ss}(F(n) \times F(1), L)}{\Map_{\ss}(F(n) \times F(1), X)}{\Map_{\ss}(F(n), L)}{\Map_{\ss}(F(n), X)}{}{}{}{}
    \end{equation}
    is a homotopy pullback square.
    \item $\exp{\ordered{0}:F(0) \to F(1)}{p}$ is a trivial Reedy fibration.
   \end{enumerate}   
 \end{lemone}

  \begin{proof}
  $(1 \Rightarrow 2)$ 
  First, let us assume $p$ is a left fibration.  
   We can write $F(n) \times F(1)$ as a colimit of a diagram 
   \begin{equation} \label{eq:zig zag diagram}
    F(n+1) \leftarrow F(n) \rightarrow ... \leftarrow F(n) \rightarrow F(n+1) 
   \end{equation}
   such that all maps $F(n) \to F(n+1)$ take $0$ to $0$ (for a more detailed description of this diagram see \cite[Diagram 10.4]{rezk2001css}). 
   Thus, by applying $2$-out-of-$3$ to the diagram $F(0) \to F(n) \to F(n+1)$, the maps in \ref{eq:zig zag diagram} are covariant equivalences
   over $X$. 
   As the covariant model structure is left proper (\cref{The Covariant Model Structure})
   this implies that  $F(n) \times \{  0 \} \to F(n) \times F(1)$ is a covariant equivalence over $X$, which implies that 
   the square \ref{eq:some Kan equivalence} is a homotopy pullback square for all $n \geq 0$, as $p$ is a left fibration. 
   
   \medskip 
   
  $(2 \Rightarrow 1)$
  Let \ref{eq:some Kan equivalence} be a homotopy pullback square for all $n \geq 0$, which means 
  \begin{equation} \label{eq:some even dumber Kan equivalence}
   \Map_{\ss}(F(n) \times F(1),L) \to \Map_{\ss}(F(n),L) \times_{\Map_{\ss}(F(n),X)} \Map_{\ss}(F(n) \times F(1),X)
  \end{equation}
   a trivial Kan fibration.
   \par 
  We have a retract diagram 
  \begin{center}
   \begin{tikzcd}[row sep=0.5in, column sep=0.5in]
    F(n) \arrow[r] \arrow[d, "\ordered{0,...,n}"] & F(n) \arrow[r] \arrow[d, "\id \times \ordered{0}"]  & F(n) \arrow[d, "\ordered{0,...,n}"] \\
    F(n+1) \arrow[r, "i"] & F(n) \times F(1) \arrow[r, "r"] & F(n+1) 
   \end{tikzcd}
  \end{center}
   which, combined with the Kan equivalence \ref{eq:some even dumber Kan equivalence} implies that 
    $$\Map_{\ss}(F(n+1),L) \to \Map_{\ss}(F(n),L) \times_{\Map_{\ss}(F(n),X)} \Map_{\ss}(F(n+1),X)$$
    is a Kan equivalence. By \cref{Lemma Other definition of Left Fibration}, this implies that $p:L \to X$ is a left fibration.
    
  \medskip
  
  $(2 \Leftrightarrow 3)$
  The map $\exp{\ordered{0}:F(0) \to F(1)}{p}$ is a trivial Reedy fibration if and only if it is a level-wise weak equivalence, 
  meaning the maps \ref{eq:some Kan equivalence} are Kan equivalences for all $n \geq 0$.
 \end{proof}
 
 \begin{remone}
  One interesting implication of this lemma is that we can get the covariant model structure on $\ss_{/X}$ 
  (\cref{The Covariant Model Structure}) 
  also by localizing with respect to maps 
  $$\{ F(0) \to F(1)^n: n \geq 0 \}.$$
 \end{remone}  
 
 One important goal in the coming sections is to give a recognition principle for covariant equivalences.
 This will be done in \cref{the:covar equiv over simp space} and needs us to first discuss the simplicial Grothendieck construction 
 (\cref{the:simplicial groth construction}). 
 However, there are certain instances, motivated by classical homotopy theory, where recognizing covariant equivalences is quite easy.
 
 First we need to prove an important lemma about covariant equivalences. 
 
  \begin{lemone} \label{lemma:exp mono and left is left}
  Let $p:L \to X$ be a left fibration and $i:A \to B$, $j: C \to D$ monomorphisms.
  \begin{enumerate}
   \item If $i$ or $j$ are covariant equivalences then $i \square j$ is a trivial covariant cofibration. 
   \item $\exp{i}{p}$ is a left fibration.
  \end{enumerate}
 \end{lemone}
  
  \begin{proof}
   By \cref{prop:joyal tierney lifting}, the two statements are equivalent. Hence it suffices to prove the second statement. 
   By \cref{lemma:left fib as exp}, $\exp{i}{p}$ is a left fibration if and only if $\exp{\ordered{0}:F(0) \to F(1)}{\exp{i}{p}}$ is a trivial fibration. 
   By direct computation we have 
   $$\exp{\ordered{0}:F(0) \to F(1)}{\exp{i}{p}} = \exp{\ordered{0} \square i}{p} = \exp{i}{\exp{\ordered{0}}{p}}$$
   Again, by \cref{lemma:left fib as exp}, $\exp{\ordered{0}}{p}$ is a trivial Reedy fibration and so using the fact that the Reedy 
   model structure is compatible with Cartesian closure (\cref{Subsec Reedy Model Structure}), 
   $\exp{i}{\exp{\ordered{0}}{p}}$ is also a trivial Reedy fibration. 
   Hence we are done.
  \end{proof}

 This lemma has a useful corollary.
 
 \begin{corone}\label{Cor Exponents of Left fib by F is left fib}
  Let $L \to X$ be a left fibration and $Y$ a simplicial space, then $L^Y \to X^Y$ is a left fibration. 
 \end{corone}

 \begin{proof}
  This follows from the previous result if we use the cofibration $\emptyset \to Y$.
 \end{proof}
 
 With the technical lemma at hand, we can give a helpful characterization of covariant equivalences.
 
 \begin{theone} \label{the:def retract covariant equiv}
  Let $i:A \to B$ be a monomorphism over $X$. Then $i$ is a trivial cofibration in the covariant model structure on $\ss_{/X}$ 
  if there exists a retraction $r: B \to A$ (not necessarily over $X$) and a $H: B \times F(1) \to B$ such that 
  $H(-,0) = ir$ and $H(-,1) = \id_B$.
 \end{theone}
 
 \begin{proof}
  For the proof we simply adapt the argument in \cite[Lemma 2.1]{heutsmoerdijk2015leftfibrationi} to simplicial spaces.
  Let $Z \xrightarrow{ \ p \ } Y \to X$ be a fibration in the covariant model structure over $X$. We need to prove that the lift of the following 
  diagram exists:
  \begin{center}
   \begin{tikzcd}[row sep=0.16in, column sep=0.5in]
    A \arrow[r, "f"] \arrow[dd, "i"] & Z \arrow[dd, "p"] \arrow[dr] & \\
    & & X \\
    B \arrow[uur, dashed] \arrow[r, "g"] & Y \arrow[ur]
   \end{tikzcd}
  \end{center}
  Let $j = (\ordered{0}: F(0) \to F(1)) \square (i: A \to B)$. 
  Using the homotopy $H$ and the map $j$, we can factor this diagram as follows:
   \begin{center}
   \begin{tikzcd}[row sep=0.2in, column sep=0.4in]
    A \arrow[r] \arrow[dd, "i"] & \ds F(1) \times A \coprod_{F(0) \times A} F(0) \times B \arrow[dd, "j"] \arrow[r, "f\pi_A \coprod fr"] & Z \arrow[dd, "p"] \arrow[dr] \\
    & & & X \\
    B \arrow[r, hookrightarrow, "\id_B \times \ordered{1}"] & B \times F(1) \arrow[uur, dashed] \arrow[r, "gH"] & Y \arrow[ur] &
   \end{tikzcd}
  \end{center} 
  It suffices to show the right hand square has a lift. For that we need to show that $j$ is a trivial cofibration. 
  However, this follows immediately from the fact that a pushout product of a trivial cofibration and cofibration is a 
  trivial cofibration in the covariant model structure by \cref{lemma:exp mono and left is left}.
 \end{proof}
 
 \begin{remone}
  As already stated in the proof, this theorem was stated and proven for simplicial sets by Heuts and Moerdijk  
  \cite[Lemma 2.1]{heutsmoerdijk2015leftfibrationi}.
 \end{remone}

 We will use this result in the next subsection to study left fibrations of Segal spaces.
 
 \subsection{Yoneda Lemma for Segal Spaces} \label{subsec:yoneda lemma Segal spaces}
 In this subsection we want to use the results we have proven until now to study left fibrations over Segal spaces.
 In particular, we prove the Yoneda lemma for Segal spaces (\cref{the:yoneda for Segal spaces}).  
 \par 
 We will start by simplifying the definition of a left fibration over Segal spaces.
 
  \begin{lemone} \label{Lemma Left Fibrations over Segal Spaces}
  Let $p:Y \to X$ be a Reedy fibration and $X$ a Segal space. The following are equivalent:
  \begin{enumerate}
   \item $Y$ is a Segal space and the following is a homotopy pullback square:
   \begin{center}
   \pbsq{Y_1}{Y_0}{X_1}{X_0}{\ordered{0}^*}{f_1}{ f_0}{\ordered{0}^*}
  \end{center}
  \item $p$ is a fibration in the Segal space model structure and the following is a homotopy pullback square:
  \begin{center}
   \pbsq{Y_1}{Y_0}{X_1}{X_0}{\ordered{0}^*}{f_1}{ f_0}{\ordered{0}^*}
  \end{center}
   \item $p$ is a left fibration.
  \end{enumerate}
 \end{lemone}

 \begin{proof}
 $(1 \Leftrightarrow 2)$
 This follows immediately from the fact that a Reedy fibration between Segal spaces is a Segal fibration (\cref{The Segal Space Model Structure}).
 
 \medskip 
 
 $(3 \Leftrightarrow 1)$ 
  We have the following diagram:
  \begin{center}
   \begin{tikzcd}[row sep=0in, column sep=0.3in]
    & Y_n \arrow[drr, "\ordered{n-1,n}^*"] \arrow[ddl, "\ordered{0,...,n-1}^*"'] \arrow[dddd, "p_n" near start] & & \\
    & & & Y_1 \arrow[ddl, "0^*"'] \arrow[dddd, "p_1"]\\
    Y_{n-1} \arrow[drr, "(n-1)^*"] \arrow[dddd, "p_{n-1}"'] & & & \\
    & & Y_0 \arrow[dddd, "p_0"] & \\
    & X_n \arrow[drr, "\ordered{n-1,n}^*" near start] \arrow[ddl, "\ordered{0,...,n-1}^*"] & & \\
    & & & X_1 \arrow[ddl, "0^*"] \\
     X_{n-1} \arrow[drr, "(n-1)^*"'] & & & \\
     & & X_0 & 
   \end{tikzcd}
  \end{center}
  We have following facts about this diagram:
  \begin{itemize}
   \item The bottom square is a homotopy pullback for all $n \geq 1$ if and only if $X$ is a Segal space.
   \item Similarly, the top square is a homotopy pullback for all $n \geq 1$ if and only if $Y$ is a Segal space.
   \item The four squares around the cube are homotopy pullback square if and only if $p$ is a left fibration 
   (\cref{Lemma Other definition of Left Fibration}).
  \end{itemize}
  The result now follows from checking homotopy pullback squares. The bottom square is always a homotopy pullback square.
  \par 
  If we assume $(2)$ the then the top square and the right hand square are homotopy pullback square and using induction we can conclude
  that the other three squares are homotopy pullback square as well.
  \par 
  If we assume $(3)$ then all four square around are homotopy pullback squares which implies that the top square is a homotopy pullback square 
  as well. 
 \end{proof}

 \begin{remone}
  Note that \cite{debrito2018leftfibration} and \cite{kazhdanvarshvsky2014yoneda} use the characterization in \cref{Lemma Left Fibrations over Segal Spaces}
  as a definition of left fibrations rather than the definition we have given here (\cref{def:left fibration}). 
  Hence this lemma proves that our definition agrees with theirs when the base is a Segal space and thus is a proper generalization 
  of their definition.
 \end{remone}
 
 Note a left fibration of Segal spaces generalizes discrete Grothendieck opfibration.
 
 \begin{lemone}\label{lemma:pi zero Grothendieck fibration}
  Let $W$ be a Segal space and $p: L \to W$ be a left fibration. Then the induced functor on homotopy categories 
  $\Ho (p): \Ho (L) \to \Ho (W)$ is a discrete Grothendieck opfibration.
 \end{lemone}
 
 \begin{proof}
  By \cref{Lemma Left Fibrations over Segal Spaces} $L$ is a Segal space. Thus $\Ho(p): \Ho(L) \to \Ho(W)$ is a functor. We want to prove it is a discrete Grothendieck opfibration. 
  Let $[f]: x \to y$ be a morphism in $\Ho(W)$ and $\hat{x}$ be a lift in $\Ho(L)$. We need to prove there is a unique lift $\hat{[f]}$ 
  such that $\Ho(p)(\hat{[f]}) = [f]$. 
  \par 
  Let $f$ in $W_1$ be a representative for the class $[f] \in \pi_0(W_1)$.
  Then $(\hat{x},f)$ is a point in $L_0 \times_{W_0} W_1$. The fact that $p:L \to W$ is a left fibration implies that 
  $$L_1 \twoheadrightarrow L_0 \times_{W_0} W_1$$
  is a trivial Kan fibration 
  and so the fiber over $(\hat{x},f)$, which we denote by $F_f$, is contractible. This means $\pi_0(F_f) = \{\hat{[f]} \}$ has a single element, 
  which is precisely the unique lift. 
 \end{proof}
 
 There is an inverse argument to \cref{lemma:pi zero Grothendieck fibration}. 
 
 \begin{lemone} \label{lemma:Grothendieck opfib is left fib}
  Let $p: \D \to \C$ be a Grothendieck opfibration. Then $i_2^*(Np): i_2^*(N\D) \to i_2^*(N\C)$ is a left fibration.
 \end{lemone}

 \begin{proof}
  Notice $i_2^*(N\D)$ and $i_2^*(N\C)$ are simplicial discrete spaces, which means $i_2^*(Np)$ is a Reedy fibration. 
  Moreover, $i_2^*(N\D)$ and $i_2^*(N\C)$ are nerves of categories and hence Segal spaces. Thus by \cref{Lemma Left Fibrations over Segal Spaces} 
  we have to show that the square 
  \begin{center}
   \begin{tikzcd}[row sep=0.5in, column sep=0.5in]
    i_2^*N(\D)_1 \arrow[r, "s"] \arrow[d] & i_2^*N(\D)_0 \arrow[d] \\
    i_2^*N(\C)_1 \arrow[r, "s"] & i_2^*N(\C)_0
   \end{tikzcd}
  \end{center}
  is a pullback square. However, this is precisely the lifting condition of a Grothendieck opfibration (\cref{lemma:discrete Groth fib lifting}).
 \end{proof}
 
 We can use the connection between left fibrations and Grothendieck opfibrations to study conservativity of left fibrations.
 
 \begin{defone}
  A map of Segal spaces $p: V \to W$ is {\it conservative} if the square 
  \begin{equation} \label{eq:conservative map Segal spaces}
   \begin{tikzcd}[row sep=0.5in, column sep=0.5in]
    V_{hoequiv} \arrow[r] \arrow[d] & V_1 \arrow[d] \\
    W_{hoequiv} \arrow[r, hookrightarrow] & W_1 
   \end{tikzcd}
  \end{equation}
  is a homotopy pullback square.
 \end{defone}

 We can characterize conservativity via the homotopy category.
 
 \begin{lemone} \label{lemma:conservative map Segal spaces}
  A map of Segal spaces $p:V \to W$ is conservative if and only if the functor $\Ho p: \Ho V \to \Ho W$
  is a conservative functor.
 \end{lemone}

 \begin{proof}
  By \cref{def:w hoequiv}, $W_{hoequiv} \to W_1$ is an inclusion of path-components. Thus the square \ref{eq:conservative map Segal spaces}
  is a homotopy pullback square if and only if 
   \begin{center} 
   \begin{tikzcd}[row sep=0.5in, column sep=0.5in]
    \pi_0(V_{hoequiv}) \arrow[r] \arrow[d] & \pi_0(V_1) \arrow[d] \\
    \pi_0(W_{hoequiv}) \arrow[r, hookrightarrow] & \pi_0(W_1) 
   \end{tikzcd}
  \end{center}
  is a pullback square of sets. However, by \cref{def:homotopy category}, $\pi_0(W_{hoequiv}) \subset \pi_0(W_1)$ is the subsets of isomorphisms.
  Thus, the square is a pullback if and only if $\Ho(p)$ is conservative.
 \end{proof}

 We can finally relate conservativity and left fibrations.
 
 \begin{lemone} \label{lemma:left fibration conservative}
  Let $W$ be a Segal space and $p:L \to W$ a left fibration. Then $p$ is conservative.
 \end{lemone}

 \begin{proof}
  By \cref{Lemma Left Fibrations over Segal Spaces}, $L$ is a Segal space and so $p$ is a map of Segal spaces. 
  Thus, by \cref{lemma:conservative map Segal spaces}, $p$ is conservative if and only if the functor $\Ho p: \Ho W \to \Ho V$ is conservative. 
  However, by \cref{lemma:pi zero Grothendieck fibration}, $p$ is a discrete Grothendieck opfibration and thus is conservative.
 \end{proof}

 \begin{remone}
  This result was proven for quasi-categories in \cite[Proposition 4.9]{joyal2008theory}, which could give us the analogous argument for complete 
  Segal spaces. However, this proof generalizes the result to arbitrary Segal spaces. 
 \end{remone}

 We can use conservativity to characterize left fibrations over complete Segal spaces.
 
 \begin{lemone} \label{lemma:left fib over CSS is CSS fib}
  Let $W$ be a complete Segal space and $p: V \to W$ be a left fibration. Then $V$ is a complete Segal space.
 \end{lemone}

 \begin{proof}
  We have the diagram
  \begin{center}
   \begin{tikzcd}[row sep=0.5in, column sep=0.5in]
    V_{hoequiv} \arrow[r] \arrow[d] & V_1 \arrow[r] \arrow[d] & V_0 \arrow[d] \\
    W_{hoequiv} \arrow[r] \arrow[rr, bend right=20, "\simeq"']& W_1 \arrow[r] & W_0
   \end{tikzcd}
   .
  \end{center}
  The left hand square is a homotopy pullback square as $p$ is a left fibration and thus conservative 
  (\cref{lemma:left fibration conservative}).
  The right hand square is a pullback square because of $p$ is a left fibration. 
  Hence the whole rectangle is a homotopy pullback. 
  \par 
  Now completeness of $W$ implies that the bottom map is an equivalence (\cref{Def Complete Segal Spaces}) and, as the square is a 
  homotopy pullback, this means $V_{hoequiv} \to V_0$ is an equivalence. Hence, again by \cref{Def Complete Segal Spaces}, this means 
  that $V$ is a complete Segal space.
 \end{proof}

 We can combine \cref{Lemma Left Fibrations over Segal Spaces} and \cref{lemma:left fib over CSS is CSS fib} into following very useful result.
 
 \begin{propone} \label{prop:left fib over Segal is Segal fib}
  Let $W$ be a Segal space. Then the adjunction 
  \begin{center}
   \adjun{(\ss_{/W})^{Seg}}{(\ss_{/W})^{cov}}{\id}{\id}
  \end{center}
  is a Quillen adjunction, where the left hand side has the induced Segal space model structure and the right hand side has the covariant 
  model structure. 
  \par 
  Moreover, if $W$ is also complete then the adjunction 
  \begin{center}
   \adjun{(\ss_{/W})^{CSS}}{(\ss_{/W})^{cov}}{\id}{\id}
  \end{center}
  is a Quillen adjunction, where the left hand side has the CSS model structure and the right hand side has the covariant model structure. 
 \end{propone}
 
 \begin{proof}
  Let us focus on the case for Segal spaces first. 
  By \cref{Cor Quillen adj for localizations} it suffices to prove that the left adjoint preserves monomorphisms and the right adjoint 
  preserves Reedy fibrations and fibrant objects. It is evident that the identity functor preserves monomorphisms and Reedy fibrations.
  For the last part, first observe that if $W$ is a Segal space then a fibrant object in $(\ss_{/W})^{Seg}$ is a Segal fibration 
  (by \cref{lemma:localized model structure base change}).
  Thus we only have to prove that if $W$ is a Segal space and $p: V \to W$ is a left fibration then $p$ is a Segal fibration, 
  which is precisely the statement of \cref{Lemma Left Fibrations over Segal Spaces}.
  \par
  The case for complete Segal spaces is identical except in the last step we use the fact that a left fibration over a complete Segal space is a 
  complete Segal space fibration as shown in \cref{lemma:left fib over CSS is CSS fib}. 
 \end{proof}
 
 \begin{remone}
  The assumption that the base simplicial space is fibrant is not necessary and this proposition can be generalized 
  to arbitrary simplicial spaces as we will do in \cref{The Covariant local of CSS}. However, before we can do that we need to 
  understand invariance properties of the covariant model structure (\cref{The Covar invariant under CSS equiv}).
 \end{remone}

 Having a better understanding of left fibrations over Segal spaces, we can move on to prove the Yoneda lemma for Segal spaces. 

 \begin{defone} \label{def:under Segal space}
  Let $W$ be a Segal space and $x$ an object in $W$. Then we define the under-Segal space $W_{x/}$ as
  $$W_{x/} = W^{F(1)} \strut^{s} \underset{W}{\times}^{ \{x\} } F(0)$$
 \end{defone}

 \begin{remone}
  The definition of under-category for Segal spaces given here differs from the definition of under-category 
  for quasi-categories \cite[Proposition 1.2.9.2]{lurie2009htt}, which uses the join construction.
  However, these two definitions are in fact equivalent. Indeed, this follows from the fact that the covariant model structures of 
  quasi-categories and complete Segal spaces are equivalent, as proven in \cref{Sec Comparison with Quasi-Categories}.
 \end{remone}

 Notice the under-Segal space is in fact a Segal space.
 
 \begin{lemone} \label{lemma:under Segal is Segal}
  For a Segal space $W$ and object $x$, the under-Segal space $W_{x/}$ is a Segal space.
 \end{lemone}
 
 \begin{proof}
  We have following pullback diagram 
  \begin{center}
   \begin{tikzcd}[row sep=0.5in, column sep=0.5in]
    W_{x/} \arrow[r] \arrow[d] & W^{F(1)} \arrow[d, "(s \comma t)"] \\
    F(0) \times W \arrow[r, "\{ x \} \times \id_W"] & W \times W
   \end{tikzcd}
  \end{center}
  The right hand map is a fibration in the Segal space model structure as it is the pullback exponential of a cofibration and a fibrant object 
  (\cref{The Segal Space Model Structure}). 
  Thus the pullback is a Segal fibration. The result now follows from the fact that $W$ itself is a Segal space.
 \end{proof}

 We have shown that the projection map $W_{x/} \to W$ that takes each morphism to its source is a fibration in the Segal space 
 model structure. We want to show that it is actually left fibration.
 
 \begin{theone} \label{the:under CSS left fibration}
  Let $W$ be a Segal space and $x$ an object in $W$. Then the projection map $W_{x/} \to W$ is a left fibration.
 \end{theone}
 
 \begin{proof}
  In order to simplify notation we will denote the four vertices $F(1) \times F(1)$ by $\{ 00$, $01$, $10$, $11 \}$. 
   
  By \cref{Lemma Left Fibrations over Segal Spaces} it suffices to prove that the map 
  $$\pi: (W_{x/})_1 \twoheadrightarrow (W_{x/})_0 \underset{W_0}{\times} W_1$$
  is a trivial Kan fibration. 
  \par 
  Notice 
  $$(W_{x/})_0 \underset{W_0}{\times} W_1 \cong \Delta[0] \underset{W_0}{\times} W_1 \underset{W_0}{\times} W_1 \cong 
  \Delta[0] \underset{W_0}{\times} \Map_{\ss}(G(2), W)$$
  and 
  $$(W_{x/})_1 = \Delta[0] \underset{W_1}{\times} (W^{F(1)})_1 \cong \Delta[0] \underset{W_0}{\times} W_0 \times_{W_1} (W^{F(1)})_1 \cong 
  \Delta[0] \underset{W_0}{\times} \Map_{\ss}(F(0) \coprod_{F(1)}\strut^{\ordered{10,11}} (F(1) \times F(1)),W).$$
  Thus, it suffices to prove the map 
  $$\Map(F(0) \coprod_{F(1)} \strut^{\ordered{10,11}} (F(1) \times F(1)),W) \to \Map_{\ss}(G(2), W)$$
  is a trivial Kan fibration, or, equivalently, the map 
  \begin{equation} \label{eq:g formula}
   \ordered{00,01} \coprod_{\ordered{01}} \ordered{01,11}: G(2) \hookrightarrow F(0) \coprod_{F(1)} (F(1) \times F(1))
  \end{equation}
  is a trivial cofibration in the Segal space model structure (as $W$ is a Segal space).
  \par 
  This map factors as
  $$ G(2) \hookrightarrow F(2) \xrightarrow{\ordered{00,01,11}} F(0) \coprod_{F(1)} (F(1) \times F(1)) $$
  and so we only need to show the second map is a Segal equivalence.
  Using the bijection 
  $$F(1) \times F(1) \cong F(2) \strut^{\ordered{0,2}} \coprod_{F(1)}\strut^{\ordered{0,2}} F(2),$$ 
  the map $\ordered{00,01,11}$ is the pushout of the following diagram:
  \begin{center}
   \begin{tikzcd}[row sep=0.5in, column sep=0.5in]
    F(1) \arrow[d, hookrightarrow] & 
    F(1) \arrow[l, "\id"'] \arrow[r, "\ordered{0,2}"] \arrow[d, hookrightarrow, "="] & 
    F(2) \arrow[d, hookrightarrow, "="]  \\
    \ds F(0) \coprod_{F(1)} F(2) & F(1) \arrow[l, "\ordered{0,2}"] \arrow[r, "\ordered{0,2}"'] & F(2) 
   \end{tikzcd}
   .
  \end{center}
  So, the result follows from knowing that the left hand map is a trivial cofibration, 
  as the Segal space model structure is left proper by \cref{The Segal Space Model Structure}.
  The left hand map itself is the pushout of the following diagram:
  \begin{center}
   \begin{tikzcd}[row sep=0.5in, column sep=0.5in]
    F(0) \arrow[d, hookrightarrow, "\simeq"]& F(1) \arrow[d, hookrightarrow, "\simeq"] \arrow[l] \arrow[r, "\ordered{1,2}"] & G(2) \arrow[d, hookrightarrow, "\simeq"] \\ 
    F(0) & F(1) \arrow[l] \arrow[r, "\ordered{1,2}"'] & F(2) 
   \end{tikzcd}
   .
  \end{center}
  As all vertical arrows are equivalences in the Segal space model structure and the Segal space model structure is left proper, 
  the pushout is a Segal space equivalence as well. Hence we are done.
 \end{proof}
 
 \begin{remone}
  In order to better understand the proof it might be helpful to visualize the map \ref{eq:g formula} as:
  \begin{center}
   $\ds \ordered{00,01} \coprod_{\ordered{01}} \ordered{01,11}:$ 
   \begin{tikzcd}[row sep=0.11in, column sep=0.25in]
    0 \arrow[rr] & & 1 \arrow[dd] & & & & & 00 \arrow[rr] \arrow[ddrr,bend left = 15] \arrow[ddrr, bend right=15] & & 01 \arrow[dd] \\
     & &  & \strut \arrow[rrrr, hookrightarrow] & &  & & \strut \\
      & & 2 & & & & & & & 11 
   \end{tikzcd}
  \end{center}
 \end{remone}

 \begin{remone}
  The fact that the left fibration $W_{x/} \to W$ happened to be a Segal fibration is not a coincidence. 
  We will later see that every left fibration is indeed a Segal fibration (\cref{The Covariant local of CSS})
 \end{remone}
 
 \begin{remone} 

  The Segal space condition in \cref{the:under CSS left fibration} is necessary, as we will show with $G(2)$.
  We want to prove that $G(2)_{0/} \to G(2)$ is not a left fibration. 
  For that it suffices to observe that the map
  $$ \pi: (G(2)_{0/})_1 \to (G(2)_{0/})_0 \underset{G(2)_0}{\times} G(2)_1$$
  is not a trivial Kan fibration. 
  \par 
  Notice that 
  $$(G(2)_{0/})_0 \underset{G(2)_0}{\times} G(2)_1 = \{ 00, 01\} \underset{\{ 0,1,2 \}}{\times} \{ 00, 01, 11, 12, 22 \} 
  = \{ (00,00), (00,01), (01,11), (01,12) \}.$$
  On the other hand we have 
  $$(G(2)_{0/})_1 = \{\alpha: F(1) \times F(1) \to G(2): \alpha(0,0) = 0 \}$$
  and the map $\pi$ simply restricts $\alpha: F(1) \times F(1) \to G(2)$ to the pair 
  $(\alpha \circ \ordered{(0,0),(0,1)}, \alpha \circ \ordered{(0,1),(1,1)})$.
  \par 
  Hence, the point $(01,12)$ has no lift along $\pi$ as any choice of lift $\alpha$ necessarily satisfies $\alpha(1,1) = 2$, which is impossible.
 \end{remone}

 We are now at the point where we can prove the Yoneda lemma for Segal spaces.
 
 \begin{theone} \label{the:yoneda for Segal spaces}
  Let $W$ be a Segal space and $x$ an object. Then the map $\{id_x\}: F(0) \to W_{x/}$ is a covariant equivalence over $W$.
 \end{theone}

 \begin{proof}
  Let $\mul: F(1) \times F(1) \to F(1)$ be the map defined by $\mul(i,j) = ij$ where $i,j = 0,1$.
  Moreover, let $\mul^*: W^{F(1)} \to W^{F(1) \times F(1)}$. 
  We can adjoin it to the map $\widehat{\mul^*}: F(1) \times W^{F(1)} \to W^{F(1)}$, which fits into the following diagram:
  \begin{center}
   \begin{tikzcd}[row sep=0.5in, column sep=0.8in]
    F(0) \times W^{F(1)} \arrow[dr, "s_0s"] \arrow[d, hookrightarrow, "\ordered{0} \times \id"'] & \\
    F(1) \times W^{F(1)} \arrow[r, "\widehat{\mul^*}"] & W^{F(1)} \\
    F(0) \times W^{F(1)} \arrow[ur, "\id"'] \arrow[u, hookrightarrow, "\ordered{1} \times \id"]
   \end{tikzcd}
   .
  \end{center}
  Now restricting this diagram via the inclusion $W_{x/} \hookrightarrow W^{F(1)}$  we get the diagram 
  \begin{center}
   \begin{tikzcd}[row sep=0.5in, column sep=0.8in]
    F(0) \times W_{x/} \arrow[dr, "\{\id_x\}"] \arrow[d, hookrightarrow, "\ordered{0} \times \id"'] & \\
    F(1) \times W_{x/} \arrow[r, "\widehat{\mul^*}"] & W_{x/} \\
    F(0) \times W_{x/} \arrow[ur, "\id"'] \arrow[u, hookrightarrow, "\ordered{1} \times \id"]
   \end{tikzcd}
   .
  \end{center}
  Hence by \cref{the:def retract covariant equiv} the map $\{ \id_x \}: F(0) \to W_{x/}$ over $W$ is a covariant equivalence.
 \end{proof}
 
 Why do we call this the Yoneda lemma? The next corollary makes the connection more clear:
 
 \begin{corone} \label{cor:yoneda lemma Segal spaces hom version}
  Let $W$ be a Segal space and $L \to W$ be a left fibration. Then the map of spaces 
  $$\{id_x\}^*: \Map_{/W}(W_{x/},L) \to \Map_{/W}(F(0),L)$$
  is a trivial Kan fibration.
 \end{corone}

 \begin{proof}
  Follows from \cref{the:yoneda for Segal spaces} and the fact that the covariant model structure is simplicial (\cref{The Covariant Model Structure}).
 \end{proof}

 \begin{remone}
  As this is a very famous result, it has been proven by many people, including
  \cite[Lemma 1.31]{debrito2018leftfibration} in the context of Segal spaces.
  It also been proven by Lurie \cite{lurie2009htt}, where it follows from the straightening construction, 
  and by Joyal \cite[Chapter 11]{joyal2008theory}, using quasi-categories.
  Finally, there is also a proof  using $\infty$-cosmoi \cite[Theorem 6.0.1]{riehlverity2017inftycosmos}.
 \end{remone}
 
 We can use this result to study the relation between initial objects and representable functors. 
 Let $W$ be a Segal space. Then for two objects $x,y$ we can define the mapping space $\map_W(x,y)$ (\cref{def:mapping space Segal space}).
 We would hope that this choice is functorial, meaning we get a functor $\map_W(x,-)$. This would require an actual composition map 
 $$f_*: \map_W(x,y) \to \map_W(x,z)$$
 for any map $f: y \to z$ in $W$. However, composition of morphisms in a Segal space is only defined up to contractible ambiguity. 
 For more details on composition in Segal spaces see \cite[Section 5]{rezk2001css}
 \par 
 We will thus take a fibrational approach.
 We observed in \cref{ex:representable Grothendieck fib} that the Grothendieck opfibration associated to the representable functor 
 $\Hom(a,-): \C \to \set$ is the under-category $\C_{a/} \to \C$. This motivates following definition:
 
 \begin{defone}
  A left fibration $p:V \to W$ is called {\it representable} if there exists an object $x$ in $W$ and a Reedy equivalence 
  $f: W_{x/} \to V$ over $X$.
 \end{defone}

 Using covariant equivalences we can relate representable left fibrations with the concept of {\it initiality}.
 
 \begin{defone} \label{def:initial object}
  Let $W$ be a Segal space. An object $x$ in $W$ is called {\it initial} if the map $\{x\}:F(0) \to W$ is a covariant equivalence 
  over $W$.
 \end{defone}

  \begin{remone}
  Initial objects are a special kind of colimit as we shall see in \cref{Subsec Colimits Cofinality and Quillens Theorem A}.
  Initial objects were thus studied in the context of colimits in quasi-categories \cite[1.2.12]{lurie2009htt} \cite[10.1]{joyal2008notes}.
 \end{remone}
 
 \begin{theone} \label{the:rep equiv initial obj}
  Let $p:L \to W$ be a left fibration. Then the following are equivalent:
  \begin{enumerate}
   \item $p$ is representable.
   \item $L$ has an initial object
  \end{enumerate}
 \end{theone}
 
 \begin{proof}
   $(1) \Rightarrow (2)$
 If $p$ is representable then $L$ is Reedy equivalent to $W_{x/}$ for some object $x$ in $W$. Thus it suffices to prove $W_{x/}$ 
 has an initial object. We have following diagram.
  \begin{center}
   \begin{tikzcd}[row sep=0.5in, column sep=0.3in]
    & & (W_{x/})_{\id_x/} \arrow[d, "\pi" ,twoheadrightarrow]\\
    F(0) \arrow[dr, "x"] \arrow[urr, "\{ \id_{\id_x} \}", "\simeq"', bend left = 20] \arrow[rr, "\{ \id_x \}", "\simeq"'] & & 
    W_{x/} \arrow[dl, "p", twoheadrightarrow] \\
     & W &
   \end{tikzcd}
   .
  \end{center} 
  $W_{x/}$ is a left fibration over $W$. By  \cref{Lemma Composition preserves Left fibrations} 
  $(W_{x/})_{id_x/}$ is also a left fibration over $W$ as the composition of left fibrations is a left
  fibration. By \cref{the:yoneda for Segal spaces}, the map $\{ \id_x \}$ is a covariant equivalence over $W$. 
  By the same argument
  the map $\{ \id_{\id_x} \}$ is a covariant equivalence over $W_{x/}$, which implies it is also a covariant equivalence over $W$ 
  (\cref{The Base change covar}).
  By $2$-out-of-$3$, we get that $\pi$ is a covariant equivalence over $W$. But $\pi$ is a map between left fibrations over $W$ and thus must 
  be a trivial Reedy fibration (\cref{The Covariant Model Structure}).
  
  \medskip 
  
  $(2) \Rightarrow (1)$
  Let $\{ l \}:F(0) \to L$ be a covariant equivalence over $L$. 
  Then, by \cref{The Base change covar}, $\{ l \}:F(0) \to L$ is a covariant equivalence over $X$. 
  By \cref{the:yoneda for Segal spaces}, $F(0) \to X_{p(l)/}$ is a trivial covariant cofibration over $X$ and, by assumption, $L \to X$ 
  is a left fibration and thus a covariant fibration over $X$ and so we can lift the diagram below 
  \begin{center}
   \begin{tikzcd}[row sep=0.5in, column sep=0.5in]
    F(0) \arrow[r, "\{ x \}", "\simeq"'] \arrow[d, "\{ \id_{p(l)} \}"', "\simeq"] & L  \arrow[d, "p"] \\
    X_{p(l)/} \arrow[r] & X
   \end{tikzcd}
   .
  \end{center}
  By assumption, the top map is a covariant equivalence over $L$ and hence also over $X$ (\cref{The Base change covar}). 
  Thus, by $2$-out-of-$3$, the lift $X_{p(l)/} \to L$ is a covariant equivalence over $X$.
  As both are left fibrations, by \cref{The Covariant Model Structure}, this map is a Reedy equivalence.
 \end{proof}
 
 \begin{remone}
  Notice the second condition only depends on $L$. Thus representability of a left fibration $p:L \to W$ is independent of the map $p$ 
  and base $W$.
 \end{remone}

\section{From the Grothendieck Construction to the Yoneda Lemma} \label{sec:grothendieck construction}
 In \cref{subsec:yoneda lemma Segal spaces} we studied many important features of the covariant model structure over Segal spaces. 
 The goal is to generalize all those results to the covariant model structure over an arbitrary simplicial space. 
 An important step is to have a precise characterization of left fibrations over $F(n)$ and a computationally feasible way 
 for characterizing covariant equivalences over an arbitrary simplicial space.
 The goal of this section is to address both these concerns.
 \par 
 In \cref{subsec:grothendieck construction over categories} we prove the {\it simplicial Grothendieck construction} for categories 
 (\cref{the:simplicial groth construction}), which in particular gives us a characterization of left fibrations over $F(n)$. 
 In \cref{Subsec The Yoneda Lemma}, we will then use this characterization to prove the 
 {\it recognition principle} for covariant equivalences (\cref{the:covar equiv over simp space}).

  \begin{notone} \label{not:drop the i}
  In this section we want to study left fibration over $(i_2)^*N\C$ (\cref{not:bunch of functors}).
  In order to simplify notation, we will denote the simplicial space $(i_2)^*N\C$ by $N\C$ as well.
 \end{notone}

\subsection{Grothendieck Construction over Categories} \label{subsec:grothendieck construction over categories}
 In \cref{prop:integral adjunctions} we proved an adjunction between set-valued functors out of $\C$ and functors over $\C$, 
 which gives us an equivalence when we restrict to discrete Grothendieck opfibrations.
 \par 
 In this subsection we generalize this result and prove two Quillen equivalences between space valued functors out of $\C$ 
 and left fibrations over $N\C$. We will then use this result to give a precise characterization of left fibration over $F(n)$.
 
 \begin{defone} \label{def:projective model structure}
  Let $\C$ be a small category. We define the {\it projective} model structure on the functor category $\Fun(\C,\s)$ as follows.
  \begin{itemize}
   \item[(F)] A natural transformation $\alpha: G \to H$ is a projective fibration if and only if for every object $c$ in $\C$ the map 
   $\alpha_c: G(c) \to H(c)$ is a Kan fibration.
   \item[(W)] A natural transformation $\alpha: G \to H$ is a projective equivalence if and only if for every object $c$ in $\C$ the map 
   $\alpha_c: G(c) \to H(c)$ is a Kan equivalence.
   \item[(C)] A natural transformation is a projective cofibration if it satisfies the left lifting property with respect to all trivial projective 
   fibrations. 
  \end{itemize}
 \end{defone}
 
 The projective model structure on $\Fun(\C,\s)$ exists (\cite[Proposition A.2.8.2]{lurie2009htt}).
 Recall that for a given simplicial set $S$ we denote the constant functor as $\{ S \} : \C \to \s$ (\cref{subsec:notation}).
 
 \begin{remone} \label{rem:projective properties}
   The projective model structure has many desirable properties. 
  \begin{enumerate}
   \item It is proper.
   \item It is combinatorial.
   \item It is a simplicial model category, with simplicial enrichment given by 
   $$\Map(F,G)_n = \Nat(F \times \{ \Delta[n] \},G)$$
   \item It is compatible with Cartesian closure:
   If $A \to B$ and $C \to D$ are cofibrations then $(A \to B) \square (C \to D)$ is a cofibration, which is trivial if either is trivial.
  \end{enumerate}
 \end{remone}
 
 \begin{remone} \label{rem:functor is space simp functor in set} 
  Using the isomorphism of functor categories 
  $$\Fun(\C,\s) \cong \Fun(\Delta^{op}, \Fun(\C,\set))$$
  we can think of a space valued functor $G: \C \to \s$ as a simplicial set valued functor $G_\bullet: \C \to \set$.
  Thus we will often switch between those when required. 
 \end{remone}

 Our first step is to generalize the adjunction from \cref{prop:integral adjunctions}.
 
 \begin{defone} \label{def:sint}
  Let 
  $$\sint_\C : \Fun(\C,\s) \to \ss_{/N\C}$$ 
  be the functor that applies $\int_\C$ level-wise to the functor $G: \C \to \s$. 
 \end{defone}
 
 \begin{remone} \label{rem:sint value}
  By direct computation, the simplicial space $\sint_\C G$ is level-wise equal to 
  $$(\sint_\C G)_n = \coprod_{c_0 \to \cdots \to c_n} G(c_0)$$
 \end{remone}

 \begin{defone} \label{def:sT}
  Let 
  $$\sT_\C: \ss_{/N\C} \to \Fun(\C,\s)$$
  be the functor defined as the left Kan extension of the functor 
  $$\sT_\C(p: F(n) \times \Delta[l] \to N\C) = \Hom(p(0),-) \times \Delta[l]$$
 \end{defone}

 \begin{defone} \label{def:sH}
  Let 
  $$\sH_\C: \ss_{/N\C} \to \Fun(\C,\s)$$
  be the functor that takes a map $p: Y \to N\C$ to the functor that for an object $c$ in $\C$ is given by 
  $$\sH_\C(p:Y \to N\C)(c) = \Map_{/N\C}(N\C_{c/},Y)$$
 \end{defone}

 \begin{lemone}
  The functors $\sT, \sint, \sH$ give us two simplicially enriched adjunctions $(\sT \dashv \sint)$, $(\sint \dashv \sH)$.
 \begin{center}
   \begin{tikzcd}[row sep=0.5in, column sep=0.9in]
    \Fun(\C,\s) \arrow[r, "\sint_\C" description] & \ss_{/N\C} \arrow[l, bend left = 30, "\sH_\C", "\bot"'] \arrow[l, bend right=30, "\sT_\C"', "\bot"]  
   \end{tikzcd}
 \end{center}
 \end{lemone}
 
 \begin{proof}
  By definition $\sT$ commutes with colimits. Hence it suffices to observe that we have following natural bijections
  $$\Nat(\sT_\C(p: F(n) \times \Delta[l] \to N\C),G) \cong \Nat(\Hom(p(0),-) \times \Delta[l], G) \cong G_l(p(0)) \cong $$
  $$\Hom_{/\C}(p:F(n) \times \Delta[l] \to N\C,\sint_\C G)$$
  which proves the adjunction $\sT_\C \dashv \sint_\C$ ($G_l: \C \to \set$ is as described in \cref{rem:functor is space simp functor in set}).
  \par 
  On the other hand we have 
  $$\Hom_{/N\C}(\sint_\C(\Hom(c,-) \times \Delta[l]) \to N\C, p:Y \to N\C) \cong $$
  $$  \Hom_{/N\C}(N\C_{c/} \Delta[l] \to N\C, p:Y \to N\C) = \sH(p:Y \to N\C)_l(c) \cong $$
  $$  \Nat(\Hom(c,-) \times \Delta[l], \sH(p: Y \to N\C))$$
  which proves the adjunction $\sint_\C \dashv \sH_\C$.
  \par 
  Finally, we observe that both adjunctions respect the simplicial enrichment described in 
  \cref{Subsec Simplicial Spaces} and \cref{rem:projective properties}.
 \end{proof}
 
 We would like to prove that if a functor $G: \C \to \s$ is valued in Kan complexes then the map of simplicial spaces 
 $\sint_\C G \to N\C$ is a left fibration. We observe immediately that $\sint_\C G \to N\C$ has the right lifting property 
 with respect to maps of the form $(F(0) \to F(n)) \square (\partial \Delta[l] \to \Delta[l])$.
 This immediately gives us following lemma.
 
 \begin{lemone} \label{lemma:sint Reedy fib replacement}
  Let $G: \C \to \s$ be a functor and let $R\sint_\C G \to N\C$ be a Reedy fibrant replacement of $\sint_\C G \to N\C$. 
  Then $R\sint_\C G \to N\C$ is a left fibration.
 \end{lemone}
 
 We can use this observation to determine when $\sint_\C \alpha$ is a covariant equivalence.
 
 \begin{lemone} \label{lemma:sint reflect equiv}
  Let $\alpha: G \to H$ be a natural transformation. Then $\alpha$ is a projective equivalence if and only if 
  $\sint_\C \alpha$ is a covariant equivalence. 
 \end{lemone}

 \begin{proof}
  Let $R\sint_\C \alpha: R\sint_\C G \to R\sint_\C H$ be a Reedy fibrant replacement of $\alpha$. By the previous lemma $R\sint_\C G \to N\C$, 
  $R\sint_\C H \to N\C$ are left fibrations and so is an equivalence if and only if $(R\sint_\C G)_0 \to (R\sint_\C H)_0$ is a Kan equivalence.
  We now have following diagram
  \begin{center}
   \begin{tikzcd}[row sep=0.5in, column sep=0.5in]
    \coprod_{c_0} G(c_0) \arrow[r] \arrow[d, "\simeq"] & \coprod_{c_0} H(c_0) \arrow[d, "\simeq"] \\
    (R\sint_\C G)_0 \arrow[r] &  (R\sint_\C H)_0
   \end{tikzcd}
   .
  \end{center}
  The vertical maps are Kan equivalences as Reedy equivalences are level-wise Kan equivalences. 
  Hence the top map is an equivalence (which is equivalent to $\alpha$ being a projective equivalence) if and only if 
  the bottom map is an equivalence (which is equivalent to $\sint_\C \alpha$ be a covariant equivalence). 
 \end{proof}

 Although $\sint_\C G \to N\C$ has many desirable properties it is generally not a left fibration, because it is not a Reedy fibration. 
 Hence, we need to define an alternative, yet equivalent, functor that takes projectively fibrant functors to left fibrations.
 \par 
 The following remark can help guide us towards a working definition.
 
 \begin{remone}
  In \cref{prop:integral adjunctions} the left adjoint of $\int_\C$, denoted $\T_\C$, was defined as 
  $\T_\C(p: \D \to \C)(c) = \pi_0(\C_{/c} \times_{\C} \D)$ which exactly coincides with the left Kan extension of the functor that 
  takes the functor $p:[n] \to \C$ to the representable functor  $\Hom(p(0),-)$.
  \par 
  The same is not true for the simplicial case. Concretely, $\sT(p: Y \to N\C)(c)$ is only naturally Kan equivalent to 
  $\Diag^*(N\C_{/c} \times_{N\C} Y)$ (as will follow from \cref{the:simplicial groth construction}). 
 \end{remone}
 
 Using this remark, we can now give the definition of the desired functor, equivalent to $\sint_\C$.
 Here we use the fact that a map $F(n) \times \Delta[l] \to N\C$ corresponds to a functor $[n] \to \C$.
 
 \begin{defone} \label{def:sbI}
  Let 
  $$\sbI_\C: \Fun(\C,\s) \to \ss_{/N\C}$$
  be the functor that takes a functor $G: \C \to \s$ to the simplicial space $\sbI(G) \to N\C$ whose fiber over 
  $p: F(n) \times \Delta[l] \to N\C$ is given by 
  $$ \Nat(N([n] \times_{\C} \C_{/-}) \times \{ \Delta[l] \},G)$$
  Here $[n] \to \C$ is the functor that corresponds to the map $p:F(n) \times \Delta[l] \to N\C$ 
 \end{defone}

  It will follow from \cref{the:simplicial groth construction} that for a projectively fibrant functor $G$, 
  $\sbI_\C(G)$ is in fact a left fibration and equivalent to $\sint_\C G$, giving us a fibrant replacement. 
  \par 
   As we want to show $\sbI_\C$ is in fact a right Quillen functor, we also require its left adjoint. 
 
 \begin{defone} \label{def:sbT}
  Let 
  $$\sbT_\C: \ss_{/N\C} \to \Fun(\C,\s)$$
  be the functor that takes a map $Y \to N\C$ to the functor $\sbT(Y \to N\C)$ that takes value 
  $$\sbT_\C(Y \to N\C)(c) = \Diag^*(N\C_{/c} \times_{N\C} Y) = N\C_{/c} \times_{N\C} \Diag^*(Y)$$
 \end{defone}
 
 \begin{propone}
  There is a simplicially enriched adjunction 
  \begin{center}
   \adjun{\ss_{/N\C}}{\Fun(\C,\s)}{\sbT_\C}{\sbI_\C}
  \end{center}
  .
 \end{propone}

 \begin{proof}
  The adjunction is induced by the functor that takes $p:F(n) \times \Delta[l] \to N\C$ to 
  $\sbT_\C(p) = N([n] \times_{\C} \C_{/-}) \times \{ \Delta[l] \}$.
 \end{proof}

 We now want to prove that these adjunctions are Quillen adjunctions. For that we first need to show they have the 
 correct strictness properties.
 
 \begin{lemone} \label{lemma:props of sint}
  The functor $\sint_\C$ takes (trivial) projective cofibrations to (trivial) covariant cofibrations.
 \end{lemone}
 
 \begin{proof}
   It suffices to check $\sint_\C$ preserves the generating cofibrations and trivial cofibrations. 
   Observe that $\sint_\C$ takes the generating cofibrations 
   $$\partial \Delta[n] \times \Hom_\C(a,-) \to \Delta[n] \times \Hom_\C(a,-)$$
   to the cofibration 
   $$\partial \Delta[n] \times N(\C_{c/}) \to \Delta[n] \times N(\C_{c/})$$
   and similarly the generating trivial cofibrations 
   $$\Lambda[n]_i \times \Hom_\C(a,-) \to \Delta[n] \times \Hom_\C(a,-)$$
   to the trivial cofibration (in the covariant model structure over $N\C$):
   $$\Lambda[n]_i \times N(\C_{c/}) \to \Delta[n] \times N(\C_{c/}).$$
 \end{proof}

 \begin{lemone} \label{lemma:sbI proj fib to left fib}
  The functor $\sbI$ takes (trivial) projective fibrations to (trivial) left fibrations.
 \end{lemone}
 
 \begin{proof}
  Let $\alpha: G \to H$ be a projective fibration. We need to prove that $\sbI(\alpha)$ has the right lifting property with respect to maps 
  \begin{itemize}
   \item $(\partial F(n) \to F(n)) \square (\Lambda[l]_i \to \Delta[l])$,
   \item $(F(0) \to F(n)) \square (\partial \Delta[l] \to \Delta[l])$.
  \end{itemize}
  Using the adjunction $(\sbT \dashv \sbI)$ this is equivalent to proving that 
  \begin{itemize}
   \item $\sbT((\partial F(n) \to F(n)) \square (\Lambda[l]_i \to \Delta[l]))$,
   \item $\sbT((F(0) \to F(n)) \square (\partial \Delta[l] \to \Delta[l]))$
  \end{itemize}
  are trivial projective cofibrations in $\Fun(\C,\s)$.
  \par 
  By direct computation these are equal to
  \begin{itemize}
   \item $\sbT(\partial F(n) \to F(n)) \square \{\Lambda[l]_i\} \to \{\Delta[l]\}$,
   \item $\sbT(F(0) \to F(n)) \square \{\partial \Delta[l]\} \to \{\Delta[l]\}$.
  \end{itemize}
  The map $\{\Lambda[l]_i\} \to \{\Delta[l]\}$ is a trivial projective cofibration and $\{\partial \Delta[l]\} \to \{\Delta[l]\}$ 
  is a projective cofibration. Hence, by \cref{rem:projective properties}, 
  it suffices to prove that $\sbT(\partial F(n) \to F(n))$ is a projective cofibration 
  and $\sbT(F(0) \to F(n))$ is a projective trivial cofibration. 
  \par 
  For the first part notice for $n \geq 2$, $\sbT(\partial F(n) \to F(n))$ is the identity. For the other two cases we have
  \begin{itemize}
   \item $n=0$: $\sbT(\partial F(0) \to F(0) \xrightarrow{ \{c\} } N\C) = \{ \emptyset \} \to \Hom_\C(c,-)$,
   \item $n=1$: $\sbT(\partial F(1) \to F(1) \xrightarrow{ \{ f: c \to d \} } N\C) = \Hom(c,-) \coprod \Hom(d,-) \to N([1] \times_\C \C_{/-})$,
  \end{itemize}
  both of which are projective cofibrations. 
  \par 
  For the second part notice we have
  $$\sbT(F(0) \to F(n) \xrightarrow{ \{c_0 \to ... \to c_n \} } N\C) = \Hom(c_0,-) \to N([n] \times_\C \C_{/-})$$ 
  which is a trivial projective cofibration. 
 \end{proof}
 
  \begin{theone} \label{the:simplicial groth construction}
  Let $\C$ be a small category.
   The two simplicially enriched adjunctions  
    \begin{center}
    \begin{tikzcd}[row sep=0.5in, column sep=0.9in]
     \Fun(\C,\s)^{proj} \arrow[r, shift left = 1.8, "\sint_\C"] & 
     (\ss_{/N\C})^{cov} \arrow[l, shift left=1.8, "\sH_\C", "\bot"'] \arrow[r, shift left=1.8, "\sbT_\C"] &
     \Fun(\C,\s)^{proj} \arrow[l, shift left=1.8, "\sbI_\C", "\bot"']
    \end{tikzcd}
   \end{center}
   are Quillen equivalences. Here $\Fun(\C,\s)$ has the projective model structure and $\ss_{/N\C}$ has the covariant model structure over $N\C$.
  \end{theone}
 
  \begin{proof}
   First we show both are Quillen adjunctions. 
   By \cref{lemma:props of sint}, $\sint_\C$ preserves cofibrations and trivial cofibrations and so is a left Quillen functor.
   On the other hand, by \cref{lemma:fib between left fib}, a fibration between fibrant objects in the covariant model structure is a left fibration. 
   By \cref{lemma:sbI proj fib to left fib}, $\sbI_\C$ takes projective fibrations to left fibration, which means it 
   takes projective fibrations between fibrant objects to 
   covariant fibrations. By the same lemma, $\sbI_\C$ takes trivial projective fibrations to trivial covariant fibrations. Hence, by 
   \cref{Lemma For Quillen adj}, it is a right Quillen functor. 
   \par 
   We move on to prove they are Quillen equivalences. Notice, the composition functor 
   $\sbT_\C \circ \sint_\C: \Fun(\C,\s) \to \Fun(\C,\s)$ is a colimit preserving functor that takes $\Hom(c,-) \times \{\Delta[l]\}$
   to the functor $N(\C_{c/} \times_\C \C_{/-}) \times \{\Delta[l]\}$, which is naturally equivalent to $\Hom(c,-) \times  \{\Delta[l]\}$.
   Hence, the composition functor is naturally weakly equivalent to the identity and so a Quillen equivalence. Thus in order to prove both 
   adjunctions are Quillen equivalences by $2$-out-of-$3$ it suffices to prove $\sint_\C \dashv \sH_\C$ is a Quillen equivalence.
   \par
   By \cref{Lemma For Quillen equiv}, it suffices to prove that the derived counit map is an equivalence and $\sint_\C$ reflects weak equivalences. 
   We already proved that $\sint_\C$ reflects weak equivalences in \cref{lemma:sint reflect equiv}. 
   Let $L \to N\C$ be a left fibration. We want to prove that 
   $\sint_\C \sH_\C L \to L$ is a covariant equivalence. We will in fact prove it is a Reedy equivalence. It suffices to do so fiber-wise. 
   \par 
   Fix a map $F(n) \to N\C$ that we can represent by a diagram $c_0 \to ... \to c_n$. As $L$ is a left fibration we have an equivalence of spaces
   $$\Map_{/N\C}(F(n),L) \to \Map_{/N\C}(\{c_0\},L)$$
   Moreover, by \cref{rem:sint value}, 
   $$\Map_{/N\C}(F(n),\sint_\C \sH_\C L) \cong \sH_\C(c_0) = \Map_{/N\C}(\C_{c_0/},L).$$
   Hence in order to finish the proof, we only have to show the map 
    $$\Map_{/N\C}(N\C_{c/},L) \to \Map_{/N\C}(\{c\},L)$$
   is a Kan equivalence. However, this is precisely the statement of the Yoneda lemma for Segal spaces (\cref{the:yoneda for Segal spaces}).
  \end{proof}

 \begin{remone}
  It is interesting to note how this result compares to a similar result in \cite[Theorem C]{heutsmoerdijk2015leftfibrationi}.
  There the authors study a functor very similar to $\sint_\C$ using quasi-categories,
  however, as they are using simplicial sets, their functor $h_!$ is the diagonal of the level-wise Grothendieck construction. 
  Thus, they cannot simply take a Reedy fibrant replacement (as we did in \cref{lemma:sint reflect equiv}) to get a left fibration and thus have to apply 
  more complicated techniques. 
 \end{remone}

 The Quillen equivalence can help us find fibrant replacements.
 
 \begin{corone} \label{cor:fibrant rep via unit} 
  Let $Y \to N\C$ be a map of simplicial spaces. Then the derived unit map $Y \to \sbI_\C R \sbT_\C Y$ is the covariant fibrant 
  replacement of $Y \to N\C$. 
 \end{corone}

 There is one key example which we want to consider more explicitly. 
 
  Let $\C = [n]$. Then $N\C = F(n)$  (using \cref{not:drop the i}) and so the result implies that we have Quillen equivalences 
   \begin{center}
    \begin{tikzcd}[row sep=0.5in, column sep=0.9in]
     \Fun(\reb n \leb,\s)^{proj} \arrow[r, shift left = 1.8, "\sint_{\reb n \leb}"] & 
     (\ss_{/F(n)})^{cov} \arrow[l, shift left=1.8, "\sH_\C", "\bot"'] \arrow[r, shift left=1.8, "\sbT_{\reb n \leb}"] &
     \Fun( \reb n \leb ,\s)^{proj} \arrow[l, shift left=1.8, "\sbI_{\reb n \leb}", "\bot"']
    \end{tikzcd}
    .
   \end{center}

 This result has important corollaries that we will use extensively.
 
 \begin{corone} \label{cor:covariant equiv over F n}
  A map of simplicial spaces $X \to Y$ over $F(n)$ is a covariant equivalence if and only if for all maps 
  $\ordered{0,...,i}:F(i) \to F(n)$ the induced map 
  $$F(i) \times_{F(n)} X \to F(i) \times_{F(n)} Y$$
  is a covariant equivalence for all $0 \leq i \leq n$.
 \end{corone}

 \begin{proof}
 By direct computation $N([n]_{/i}) \to N([n])$ is the exactly the map $\ordered{0,...,i}:F(i) \to F(n)$.
 The result now follows from \cref{the:simplicial groth construction}.
 \end{proof}
 
 \begin{corone} \label{cor:left fib over NC colimit}
  Every left fibration $L \to F(n)$ is Reedy equivalent to a colimit of left fibrations of the form 
  $(\ordered{0,...,i} \circ \pi_1) : F(i) \times \Delta[l] \to F(n)$. 
 \end{corone}

 \begin{proof}
  Let $L \to F(n)$ be a left fibration. Then by \cref{the:simplicial groth construction} there exists a functor 
  $G: [n] \to \s$ and a Reedy equivalence $L \simeq \sint_{[n]} G$ over $F(n)$ (concretely we can take $G = \sH_\C(L)$). 
  But $G$ is a simplicial presheaf and so there is an isomorphism $G \cong \colim (\Hom([i],-) \times \Delta[l])$ and so 
  $$L \simeq \sint_{[n]} G \cong \colim (\sint_{[n]} \Hom([i],-) \times \Delta[l]) \cong
  \colim ((\ordered{0,...,i} \circ \pi_1) : F(i) \times \Delta[l] \to F(n))$$
  giving us the desired result. 
 \end{proof}
 
 \begin{corone}
  Let $L \to F(n) \times \Delta[l]$ be a left fibration. Then there is a Reedy equivalence 
  $$L \simeq \colim (
  (\ordered{0,...,i} \times \id_{\Delta[j]}) \circ \pi_1) [F(i) \times \Delta[l]) \times \Delta[j] \to F(n) \times \Delta[l]] $$
  over $F(n) \times \Delta[l]$.
 \end{corone}
 
 \begin{proof}
  The projection map $\pi_1: F(n) \times \Delta[l] \to F(n)$ is a Reedy equivalence and so by \cref{The Base change covar} gives us a Quillen equivalence 
  \begin{center}
   \adjun{(\ss_{/F(n) \times \Delta[l]})^{cov}}{(\ss_{/F(n)})^{cov}}{(\pi_1)_!}{(\pi_1)^*}
  \end{center}
   which implies that there exists a left fibration $\hat{L} \to F(n)$ and a Reedy equivalence $L \simeq \hat{L} \times \Delta[l]$ over $F(n) \times \Delta[l]$.
   The result now follows from the previous corollary.
 \end{proof}

  In the coming sections we will need the contravariant version of fibrations, {\it right fibrations}.
  
  \begin{remone} \label{rem:right fibrations}
     Until now we have focused on the covariant approach to fibrations. However, there is also a contravariant analogue. 
   Instead of repeating all the previous arguments, we will introduce following table that can help us translate all previous results.
   \renewcommand{\arraystretch}{2.0}
   \begin{center}
    \begin{tabular}{|c|c|}
     \hline 
     \ Left Fibration \ & \ Right Fibration \\  
     \ $(p_n, \ordered{0}):   Y_n \xrightarrow{ \ \simeq \ } X_n \times_{X_0} Y_0$ \ & 
     \ $(p_n, \ordered{n}):   Y_n \xrightarrow{ \ \simeq \ } X_n \times_{X_0} Y_0$ \ \\ \hline
     \ Covariant Model Structure \ & \ Contravariant Model Structure \ \\ \hline 
     \ Under-Segal Space \ & \ Over-Segal Space \ \\
     \ $W_{x/} = F(0) \strut^{\{x\}} \times^{s}_W W^{F(1)}$ \ & \ $W_{/x} = W^{F(1)} \strut^{t}\times^{\{x\}} F(0)$ \ \\ \hline  
     \ initial object \ & \ final object \ \\ \hline
     \ $\sint_\C: \Fun(\C,\s)^{proj} \to (\ss_{/N\C})^{cov}$ \ & \ $\sint^{op}_\C:\Fun(\C^{op},\s)^{proj} \to (\ss_{/N\C})^{contra}$ \ \\ \hline
    \end{tabular}
   \end{center}
  \end{remone}
 
  Having defined left and right fibrations, we can use our previous results to generalize
   \cref{rem:grothendieck fib and opfib} from categories to simplicial spaces.
 
   \begin{theone} \label{The Left is right if over line}
   Let $p: L \to X$ be a left fibration. Then the following are equivalent:
   \begin{enumerate}
    \item $p$ is a right fibration.
    \item For every map $f: F(1) \to X$ the map $f^*L \to F(1)$ is a right fibration.
    \item $p$ is a diagonal fibration.
    \item For every map $f: F(1) \to X$ the map $f^*L \to F(1)$ is a diagonal fibration.
   \end{enumerate}
  \end{theone}
  \begin{proof}
   By \cref{Lemma Local def of Left Fib}, \cref{cor:local def of diag Fib} 
   $p$ is a left, right or diagonal fibration if and only if $f^*p: f^*L \to F(n)$ is such a fibration for every map $F(n) \to X$.
   However, by \cref{the:simplicial groth construction} for every left fibration $L \to F(n)$ there exists a functor $G: [n] \to \s$ 
   and Reedy equivalence $\sint_{[n]} G \cong L$ over $F(n)$.
   Thus for the remainder of the proof we will assume $X = F(n)$ and our left fibration is of the form $\sint_{[n]} G \to F(n)$ for a functor 
   $G:[n] \to \s$, which itself is isomorphic to a simplicial object $G_\bullet:[n] \to \set$.
   
   \medskip
   
   {\it (1 $\Leftrightarrow$ 2)}
   By \cref{lemma:Grothendieck opfib is left fib}, $\sint_{[n]} G \to F(n)$ is a right fibration if and only if $\int_{[n]} G_l \to [n]$ is 
   a Grothendieck fibration for all 
   $G_l:[n] \to \set$. By \cref{rem:grothendieck fib and opfib}, this is equivalent to $G:[n] \to \s$ taking every morphism to an isomorphism. 
   This is equivalent to $G|_{\{i,i+1\}}: [1] \to \s$ taking the morphism to an isomorphism for all $0 \leq i < n$, 
   which, again by \cref{rem:grothendieck fib and opfib}, is equivalent to $\sint_{[1]} G|_{\{i,i+1\}} \to F(1)$ being a 
   right fibration for all $0 \leq i < n$.
   Hence, it is equivalent to the statement that for every map $f: F(1) \to X$ the map $f^*L \to F(1)$ is a right fibration.
   
   \medskip
   
   {\it (1 $\Leftrightarrow$ 3)}
   Every diagonal fibration is a left fibration and right fibration. On the other hand, if $\sint_{[n]} G \to F(n)$ is a left and right 
   fibration then, by the argument in the previous paragraph, $G: [n] \to \s$ takes values in isomorphisms. Thus the evident natural transformation 
   $$\{ G(0) \} \Rightarrow G$$
   from the constant functor $\{ G(0) \}: [n] \to \s$ is an isomorphism, which by \cref{the:simplicial groth construction} means 
   $\sint_{[n]} \{G(0)\} \to \sint_{[n]} G$ is an isomorphism over $F(n)$. 
   \par 
   However, by direct computation $\sint_{[n]} \{G(0)\} = G(0) \times F(n) \xrightarrow{ \ \pi_2 \ } F(n)$, which is a diagonal fibration, 
   as $G(0)$ is a Kan complex. Hence, $\sint_{[n]} G \to F(n)$ is a diagonal fibration as well.
   
   \medskip
   
   {\it (2 $\Leftrightarrow$ 4)}
   We can use the same argument as in the previous part. 
  \end{proof}

 \subsection{The Yoneda Lemma} \label{Subsec The Yoneda Lemma}
 We are finally in a position to prove the recognition principle for covariant equivalences.
 The proof has three main steps:
 \begin{enumerate}
  \item Prove how right and left fibrations interact: \cref{The Pullback preserves covar equiv}.
  \item Characterize covariant equivalences between left fibrations: \cref{The Condition for covar equiv between left fib}.
  \item Prove the recognition principle for covariant equivalences: \cref{the:covar equiv over simp space}.
 \end{enumerate}
   
  However, before we can start we need one technical lemma.
  
 \begin{lemone} \label{lemma:pullback right fib reduction}
  Let $X$ be a simplicial space and $\mathcal{L} = \{ A \to N(\C) \}$ be a set of monomorphisms in $\ss$.
  Let $(\ss_{/X}, \M_X)$ be the left Bousfield localization model structure of the induced Reedy model structure 
  with respect to the set of monomorphism $\mathcal{L} = \{ A \to N(\C) \to X \}$.
  Then the following are equivalent:
  \begin{enumerate}
   \item
   For every right fibration $p: R \to X$ the adjunction 
   \begin{center}
    \adjun{(\ss_{/X})^{\M_X}}{(\ss_{/X})^{\M_X}}{p_!p^*}{p_*p^*}
   \end{center}
   is a Quillen adjunction. 
   \item 
   For every object $c$ in $\C$ and map $ i: A \to N(\C)$ in $\mathcal{L}$ the pullback map 
   $$N(\pi)^*(i): N(\pi)^*(A) \to N(\C_{/c})$$
   is a trivial cofibration in $(\ss_{/X})^{\M_X}$. Here $\pi: \C_{/c} \to \C$ is the projection map.
  \end{enumerate}

 \end{lemone}

 \begin{proof}
 
 $ (1 \Rightarrow 2)$ 
 This is just the special case of $(1)$ applied to the right fibration $N\C_{/c} \to N\C$.
 
 \medskip 
 
 $(2 \Rightarrow 1)$
 The proof consists of several reduction steps.
  
  \medskip
  
  {\it (I) Reduce to Fibrant Objects:}
  First, by \cref{Cor Quillen adj for localizations} it suffices to show $p_!p^*$ preserves cofibrations,  $p_*p^*$ preserves Reedy fibrations 
  and fibrant objects $Y \to X$. The fact that $p_!p^*$ preserves cofibrations and that $p_*p^*$ preserves Reedy fibrations follows from the fact
  that $(p_!p^*,p_*p^*)$ is a Quillen adjunction when both sides just have the induced Reedy model structure, as the 
  Reedy model structure is right proper (\cref{Subsec Reedy Model Structure}). 
  So, we only have to prove that for every 
  fibrant object $Y \to X$, $p_!p^*(Y) \to X$ is also fibrant.
  
  \medskip 
  
  {\it (II) Reduce to Local Objects:}
  Next notice that $p_!p^*(Y) \to X$ is fibrant if and only if it is a Reedy fibrant and local with respect to maps $ A \to N(\C) \to X$, where 
  $i: A \to N(\C)$ is in $\mathcal{L}$. Again the Reedy fibrancy follows from the previous paragraph and so it suffices to prove that $p_!p^*(Y) \to X$ is local.
  Thus we need to prove that 
  $$ i^*: \Map_{/X}(N(\C),p_!p^*(Y)) \to \Map_{/X}(A,p_!p^*(Y))$$
  is a Kan equivalence.
  
  \medskip
  
  {\it (III) Reduce to Local Trivial Cofibration:}
  Using the fact that $p_!p^*$ has a left adjoint this is equivalent to 
  $$ (p_!p^*i)^*: \Map_{/X}(p_!p^*N(\C),Y) \to \Map_{/X}(p_!p^*A,Y)$$
  being a Kan equivalence. 
  
  As the model structure $\M_X$ is simplicial and $Y \to X$ is an arbitrary fibrant object, this is equivalent to 
  $$p_!p^*i : p_!p^*A \to p_!p^*N(\C)$$
  being a weak equivalence in $(\ss_{/X})^{\M_X}$.
  
  \medskip
  
  {\it (IV) Reduce to Categorical base:} Now notice we have a Quillen adjunction 
  \begin{center}
   \adjun{(\ss_{/N(\C)})^{\M_{N\C}}}{(\ss_{/N(\C)})^{\M_X}}{j_!}{j^*}
  \end{center}
  for every map $j: N\C \to X$. Thus in order to prove that $p_!p^*i$ is a weak equivalence over $X$ it suffices to prove 
  it is a weak equivalence over $N\C$.
  
  \medskip 
  
  {\it (V) Reduce to Representable Right Fibrations:}
  By definition $p_!p^*(N\C) = R \times_X N\C \to N\C$, which is a right fibration as right fibrations are preserved by pullbacks 
  (\cref{Lemma Pullback preserves Left fibrations}).
  However, by the contravariant analogue of \cref{cor:left fib over NC colimit}, 
  every right fibration over $N\C$ is Reedy equivalent to $\colim (N\C_{/c} \times \Delta[l])$.
  Thus we can reduce the argument to proving that 
  $$A \times_{N\C} N\C_{/c} \times \Delta[l] \to N\C_{/c} \times \Delta[l]$$
  is a weak equivalence in $\ss_{/N\C}$.
  
  \medskip 
  
  {\it (VI) Reduce to desired condition:}
  Finally, we again observe that $\M_{N\C}$ is a simplicial model structure on $\ss_{/N\C}$ which means $\Delta[l]$ is contractible. 
  Thus the previous condition is equivalent to 
    $$N(\pi)^*(A) = A \times_{N\C} N\C_{/c} \to N\C_{/c} $$
    being a trivial cofibration in $\ss_{/N\C}$ in the $\M_{N\C}$ model structure.
 \end{proof}
 
 \begin{remone} \label{rem:pullback right fib reduction}
  We can use the same argument to prove an analogous result for pulling back along left fibrations. 
  Concretely, 
  \begin{center}
    \adjun{(\ss_{/X})^{\M_X}}{(\ss_{/X})^{\M_X}}{p_!p^*}{p_*p^*}
   \end{center}
   is a Quillen adjunction for every left fibration $p: L \to X$ if and only if  
   $$N(\pi)^*(i): N(\pi)^*(A) \to N(\C_{c/})$$
   is a trivial cofibration in $(\ss_{/X})^{\M_X}$ for every object $c$.
 \end{remone}

 We can now use this result to give the desired connection between right fibrations and covariant equivalences.
 
 \begin{theone} \label{The Pullback preserves covar equiv}
  Let $p: R \to X$ be a right fibration. Then the adjunction
  \begin{center}
   \adjun{(\ss_{/X})^{cov}}{(\ss_{/X})^{cov}}{p_!p^*}{p_*p^*}
  \end{center}
  is a Quillen adjunction where both sides have the covariant model structure.
 \end{theone}

 \begin{proof}
  The covariant model structure is given by localization with respect to maps $F(0) \to F(n) \to X$.
  The over-category $[n]_{/i} \to [n]$ is given by the map of simplicial spaces $\ordered{0,...,i}: F(i) \to F(n)$.
  Thus, by the previous lemma, we only need to prove that the pullback map 
  $F(0)= F(0) \times_{F(n)} F(i) \to F(i)$
  is a covariant equivalence over $F(n)$. However, that true by definition. 
 \end{proof}
 
 \begin{remone}
  By \cref{rem:pullback right fib reduction} and the analogous argument to \cref{The Pullback preserves covar equiv}, 
  for every left fibration $p: L \to X$, we get a Quillen adjunction $(p_!p^*,p_*p^*)$ between contravariant model structures.
 \end{remone}

 \begin{remone}
  This result has been proven independently in the context of quasi-categories by 
  Lurie \cite[Proposition 4.1.2.15]{lurie2009htt},
  Joyal \cite[Theorem 11.9]{joyal2008theory} 
  and Nguyen \cite[Proposition 4.12]{nguyen2019covariant}.
 \end{remone}

 We can also use this lemma to also prove a relationship between right fibrations and complete Segal spaces.
 
 \begin{theone} \label{the:pullback of right fib over CSS equiv} 
  Let $W$ be a Segal space and $p: R \to W$ be a right or left fibration. Then the adjunction
  \begin{center}
   \adjun{(\ss_{/W})^{Seg}}{(\ss_{/R})^{Seg}}{p^*}{p_*}
  \end{center}
  is a Quillen adjunction where both sides have the induced Segal space model structure (\cref{prop:induced model structure}).
  \par 
  If $W$ is also complete, then the same statement holds for the adjunction
  \begin{center}
   \adjun{(\ss_{/W})^{CSS}}{(\ss_{/R})^{CSS}}{p^*}{p_*}
  \end{center}
  where now both sides have the induced complete Segal space model structure (\cref{prop:induced model structure}).
 \end{theone}

 \begin{proof}
  We will assume $p$ is a right fibration. The argument for left fibration follows similarly, using 
  the adjustment in \cref{rem:pullback right fib reduction}.
  
  Let $W$ be a Segal space. We can extend the adjunction above as follows 
  \begin{center}
   \begin{tikzcd}[row sep=0.5in, column sep=0.5in]
    (\ss_{/W})^{Seg} \arrow[r, shift left = 1.8, "p^*"] & 
    (\ss_{/R})^{Seg} \arrow[r, shift left = 1.8, "p_!"] \arrow[l, shift left=1.8, "p_*", "\bot"'] & 
    (\ss_{/W})^{Seg} \arrow[l, shift left=1.8, "p^*", "\bot"'] 
   \end{tikzcd}
   .
  \end{center}
  In order to show that $(p^*,p_*)$ is a Quillen adjunction, we have to prove $p^*$ preserves cofibrations and trivial Segal cofibrations. 
  It is evident that $p^*$ preserves cofibrations, as they are just monomorphisms. Moreover, 
  by \cref{prop:induced model structure}, $p_!$ reflects trivial Segal cofibration. 
  Hence $p^*$ preserves trivial cofibrations if and only if $p_!p^*$ preserves trivial cofibrations, which is equivalent to proving that 
  \begin{center}
   \adjun{(\ss_{/W})^{Seg}}{(\ss_{/W})^{Seg}}{p_!p^*}{p_*p^*}
  \end{center}
  is a Quillen adjunction, 
  where both sides have the Segal space model structure. We can thus apply \cref{lemma:pullback right fib reduction}.
  
  By \cref{the:induced vs localized} the induced Segal space model structure over a Segal space is just given by localizing with respect to the maps
  $G(n) \to F(n) \to W$. Thus we only need to check the map $G(n) \to F(n)$ satisfies the desired condition in \cref{lemma:pullback right fib reduction}. 
  We know the over-category over $i$ is given by $\ordered{0,...,i}: F(i) \to F(n)$.
  Thus we only need to show that 
  $G(i) = G(n) \times_{F(n)} F(i) \to F(i)$
  is an equivalence in the Segal space model structure, which is true by definition.
  \par 
  Now let us assume also in addition that $W$ is complete. By the explanation given at the beginning of the proof, it suffices to show that 
  \begin{center}
   \adjun{(\ss_{/W})^{CSS}}{(\ss_{/W})^{CSS}}{p_!p^*}{p_*p^*}
  \end{center}
  is a Quillen adjunction, where both sides have the complete Segal space model structure. 
  This means we can again use \cref{lemma:pullback right fib reduction}.
  
  By \cref{the:induced vs localized} the induced complete Segal space model structure 
  is given by localizing with respect to maps $G(n) \to F(n) \to W$ and 
  $F(0) \to E(1) \to W$. We already observed that $G(n) \to F(n)$ satisfies the condition of \cref{lemma:pullback right fib reduction} 
  so we only need to prove the same statement for the map $F(0) \to E(1)$.
  \par 
  However, $E(1) = N(I[1])$, where $I[1]$ is the category with two objects and one unique isomorphism (\cref{Def Of the E spaces}).
  By direct computation $I[1]_{/0} = I[1]_{/1} = I[1]$ and so the projection map from the over-category is just the identity map.
  Hence we are done.
 \end{proof}

 This theorem has following useful corollary.
 
 \begin{corone}
  Let following diagram be given
  \begin{center}
   \begin{tikzcd}[row sep=0.5in, column sep=0.5in]
    p^*X \arrow[r, "p^*f"] \arrow[d] & L \arrow[d, "p"] \\
    X \arrow[r, "\simeq"', "f"] & W
   \end{tikzcd}
  \end{center}
   where $W$ is a complete Segal spaces, $p:L \to W$ is a left or right fibration and $f$ is a complete Segal space equivalence. 
   Then $p^*f: p^*X \to L$ is a complete Segal space equivalence.  
 \end{corone}

 \begin{remone}
  The assumptions in the previous theorem seem too strong, the result should also hold if the base simplicial space $W$ is not
  a Segal space. This is in fact correct and we will prove this in \cref{The Pullback preserves CSS equiv} / \cref{The Pullback preserves Seg equiv}. 
  However, before we can do that we need to understand 
  the invariance of left fibrations with respect to complete Segal space equivalences, which is the goal of 
  \cref{The Covar invariant under CSS equiv}. 
 \end{remone}
 
 We can now move on to the second step and characterize covariant equivalences between left fibrations.
 
 \begin{theone} \label{The Condition for covar equiv between left fib}
  Let $L \to X$ and $L' \to X$ be left fibrations and $f:L \to L'$ a map over $X$. 
  The following are equivalent:
  \begin{enumerate}
   \item $f$ is a covariant equivalence.
   \item $f$ is a Reedy equivalence.
   \item $f$ is a Kan equivalence.
   \item $f$ is a fiber-wise Reedy equivalence ($f \times_X F(0): L \times_X F(0) \to L' \times_X F(0)$ is a Reedy equivalence for every map $F(0) \to X$).
   \item $f$ is a fiberwise Kan equivalence ($f \times_X F(0): L \times_X F(0) \to L' \times_X F(0)$ is a Kan equivalence for every map $F(0) \to X$).
   \item $f$ is a fiberwise diagonal equivalence ($f \times_X F(0): L \times_X F(0) \to L' \times_X F(0)$ is a diagonal equivalence for every map $F(0) \to X$).
  \end{enumerate}
 \end{theone}

 \begin{remone} \label{Rem Equivalences of left fibrations as functors} 
  By \cref{the:simplicial groth construction}, a left fibration $L \to N\C$ is Reedy equivalent to $\sint_\C G \to N\C$ for some functor 
  $G: \C \to \s$. Thus a map of left fibrations $L \to L'$ over $N\C$ is an equivalence if and only if the corresponding natural 
  transformation $G \to G'$ is an equivalence.
  \par 
  \cref{The Condition for covar equiv between left fib} can thus be seen as a generalization of this observation to an arbitrary simplicial space $X$:
  We are comparing two left fibrations over $X$, by comparing their fibers, which we should think of as their ``values".
 \end{remone}

 \begin{proof}
 {\it (1 $\Leftrightarrow$ 2)} Follows from the definition of localization as left fibrations are the fibrant objects
 in the covariant model structure (\cref{The Covariant Model Structure}).
 
 \medskip
 
 {\it (2 $\Leftrightarrow$ 3)} Clearly $(2)$ implies $(3)$. For the other side 
 let $f$ be a Kan equivalence, then $f_0: L_0 \to L'_0$ is a Kan equivalence of spaces.
 This implies that in the diagram  
 \begin{center}
  \begin{tikzcd}[row sep=0.5in, column sep=0.5in]
   Y_n \arrow[d, "\simeq"] \arrow[r, "f_n"] & Z_n \arrow[d, "\simeq"] \\
   Y_0 \underset{X_0}{\times} X_n \arrow[r,"(f_0 \comma id)", "\simeq"'] & Z_0 \underset{X_0}{\times} X_n
  \end{tikzcd}
 \end{center}
 the two vertical maps and the bottom horizontal map are Kan equivalences.
 Thus $f_n: Y_n \to Z_n$ is a Kan equivalence as well, which implies that $f$ is a Reedy equivalence.
 
 \medskip 
 
 {\it (2 $\Leftrightarrow$ 4)} This is the generalization of \cref{Cor Fiberwise equiv for Kan fib} to simplicial spaces. 
 
 \medskip 
 
 {\it (3 $\Leftrightarrow$ 5)} This is precisely the statement of \cref{Cor Fiberwise equiv for Kan fib}.
  
  \medskip 
  
 {\it (4 $\Leftrightarrow$ 6 )} By \cref{Lemma Pullback preserves Left fibrations}, $F(0) \times_X L \to F(0)$ is a left fibration, 
 which by \cref{Ex Covariant over the point} 
 means that $L \times_X F(0)$ is diagonally fibrant. Thus $f \times_X F(0): L \times_X F(0) \to L' \times_X F(0)$ is a 
 Reedy equivalence if and only if it is a diagonal equivalence (\cref{The Diagonal Model Structure}).
  \end{proof}
  
 \begin{remone}
  Note it suffices to prove that there exists $x: F(0) \to X$ such that $f \times_X F(0)$ fiberwise Kan equivalence for every path component of 
  $X_0$. Indeed, if $x$ and $y$ are in the same path-component then we have an equivalence of fibers 
  $f_0 \underset{X_0}{\times}^x \Delta[0] \simeq f_0 \underset{X_0}{\times}^y \Delta[0]$.
 \end{remone}

  We can now move on to the general case.
    
\begin{theone} \label{the:condition for covar equiv general rep}  
 $p:Y \to X$ be a map of simplicial spaces. 
 For every $x: F(0) \to X$,  
 there is a natural zig-zag of diagonal equivalences 
 $$R_x \underset{X}{\times} Y \xrightarrow{ \ \simeq \ } R_x \underset{X}{\times} \hat{Y} \xleftarrow{ \ \simeq \ } F(0) \underset{X}{\times} \hat{Y}$$
 Here $i: Y \to \hat{Y}$ is a choice of a left fibrant replacement of $Y$ over $X$ and $R_x \to X$ is a contravariant fibrant replacement of 
 $\{  x \} : F(0) \to X$.
\end{theone}
 
 \begin{proof}
  Fix a covariant fibrant replacement  $i: Y \to \hat{Y}$ over $X$.
  Then we have the following zig-zag of equivalences
  $$ Y \underset{X}{\times} R_x  \xrightarrow{ \ \ cov\simeq \ \ } \hat{Y} \underset{X}{\times} R_x  \xleftarrow{ \ \ contra\simeq \ \ } 
  \hat{Y} \underset{X}{\times} F(0)$$
  By \cref{The Pullback preserves covar equiv} the first map is a covariant equivalence because $R_x \to X$ is a right fibration. 
  By the covariant version of the same lemma the second map is a contravariant equivalence because $\hat{Y} \to X$ is a left fibration.
  So, by \cref{The Diag is local of Covar}, both are diagonal equivalences. 
 \end{proof}
  
  In the case $X$ is a Segal space, we can replace the zig-zag of equivalences with an actual map.
  
 \begin{remone} \label{rem:yoneda left fib}
  Let $X$ be a Segal space. Let $p: Y \to X$ be a map of simplicial spaces and 
  $Y \xrightarrow{ \ i \ } \hat{Y} \xrightarrow{ \ \hat{p} \ } X$ its fibrant replacement. 
  Then, by \cref{the:yoneda for Segal spaces}, $\{\id_X \}:F(0) \to F(0) \times_X X^{F(1)}$ is a covariant fibrant replacement of $\{x\}:F(0) \to X$ 
  over $X$. Now we have following diagram:
  \begin{center}
   \begin{tikzcd}[row sep=0.5in, column sep=0.8in]
    Y \ ^p\underset{X}{\times}^s (X^{F(1)} \ ^t\underset{X}{\times}^{\{x\}} F(0)) 
    \arrow[r, "i \underset{X}{\times} X^{F(1)} \underset{X}{\times} F(0)", "\simeq"'] &
    \hat{Y} \ ^{\hat{p}}\underset{X}{\times}^s (X^{F(1)} \ ^t\underset{X}{\times}^{\{x\}} F(0)) \arrow[d, "sec"', bend right = 30, dashed]  &
    \hat{Y} \ ^{\hat{p}}\underset{X}{\times}^{\{x\}} F(0) \arrow[l, "\hat{Y} \underset{X}{\times} \{ id_X \}"', "\simeq"] \\
    & \hat{Y}^{F(1)} \ ^{t (\hat{p}^{F(1)})}\underset{X}{\times}^{\{x\}} F(0) \arrow[u, twoheadrightarrow, "\simeq"] \arrow[ur, "s"'] & 
   \end{tikzcd}
  \end{center}
   By \cref{lemma:left fib as exp}, the map $\hat{Y}^{F(1)} \to \hat{Y} \times_X X^{F(1)}$ is a trivial Reedy fibration and so we can pick a section 
   $$sec: \hat{Y} \ ^{\hat{p}}\underset{X}{\times}^s X^{F(1)} \ ^t\underset{X}{\times}^{\{x\}} F(0) \to 
   \hat{Y}^{F(1)} \ ^{t (\hat{p}^{F(1)})}\underset{X}{\times}^{\{x\}} F(0).$$
   By \cref{the:condition for covar equiv general rep}, $\hat{Y} \underset{X}{\times} \{ id_X \}$ is a diagonal equivalence and so, by $2$-out-of-$3$,
   $s$ is a diagonal equivalence. 
   Hence 
   $$s \circ sec \circ (i \underset{X}{\times} X_{/x}): 
   Y \underset{X}{\times} X_{/x} \to    \hat{Y} \underset{X}{\times} F(0)$$ 
   is the desired diagonal equivalence.
 \end{remone}

 We can finally prove the recognition principle for covariant equivalences.
 
 \begin{theone} \label{the:covar equiv over simp space}
  Let $g: Y \to Z$ be a map over $X$ and $R_x \to X$ a choice of contravariant fibrant replacement of the map $\{x\}:F(0) \to X$.
  Then $g:Y \to Z$ over $X$ is a covariant equivalence if and only if for every $x : F(0) \to X$
  $$ R_x \underset{X}{\times} Y \to R_x \underset{X}{\times} Z$$
  is a diagonal equivalence.
 \end{theone}

  \begin{proof}  
  Let $\hat{g}: \hat{Y} \to \hat{Z}$ be a left fibrant replacement. By \cref{The Covariant Model Structure}, $g:Y \to Z$ is a covariant equivalence
  if and only if $\hat{g}:\hat{Y} \to \hat{Z}$ is a Reedy equivalence.
  We now have following diagram:
  \begin{center}
   \begin{tikzcd}[row sep=0.5in, column sep=0.5in]
    Y \underset{X}{\times} R_x \arrow[r, "g \times id"] \arrow[d, "\simeq", "i \times id"'] & 
    Z \underset{X}{\times} R_x \arrow[d, "\simeq"', "j \times id"]\\
    \hat{Y} \underset{X}{\times} R_x \arrow[r, "\hat{g} \times id"] & \hat{Z} \underset{X}{\times} R_x \\
    \hat{Y} \underset{X}{\times} F(0) \arrow[r, "\hat{g} \times id"] \arrow[u, "\simeq"'] & \hat{Z} \underset{X}{\times} F(0) \arrow[u, "\simeq"]
   \end{tikzcd}
   .
   \end{center} 
   By \cref{the:condition for covar equiv general rep}, all vertical maps are diagonal equivalences and so the top horizontal map 
   is a diagonal equivalence if and only if the bottom horizontal map is one. 
   But the bottom map is a diagonal equivalence for every $x: F(0) \to X$ if and only if $\hat{Y} \to \hat{Z}$ is a Reedy equivalence
   (\cref{The Condition for covar equiv between left fib}). Hence, we are done.
 \end{proof}
 
 In the case of Segal spaces the equivalence takes on a very simple form.
 
 \begin{corone} \label{cor:yoneda lemma Segal spaces tensor}
  Let $X$ be a Segal space and $f:Y \to Z$ a map over $X$. Then $f$ is a covariant equivalence if and only if 
  $$Y \times_X X_{/x} \to Z \times_X X_{/x}$$
  is a diagonal equivalence for every object $x$.
 \end{corone}

 \begin{proof}
  By \cref{the:yoneda for Segal spaces}, $X_{/x} \to X$ is the contravariant fibrant replacement of $\{ x \}: F(0) \to X$. 
  The result now follows from \cref{the:covar equiv over simp space}.
 \end{proof}

 \begin{remone}
  This result has also been proven by Heuts and Moerdijk \cite[Proposition G]{heutsmoerdijk2015leftfibrationi} using quasi-categories.
 \end{remone}

 \begin{remone}
  \cref{cor:yoneda lemma Segal spaces tensor} should very much remind us of the behavior of $\sbT_\C$, which takes a map 
  $Y \to N\C$ to the functor $\sbT_\C(Y)$ with value $\sbT_\C(Y)(c) = \Diag^*(Y \times_{N\C} N\C_{/c})$. 
  The functor $\sbT_\C$ was only defined over nerves of categories, but \cref{cor:yoneda lemma Segal spaces tensor} suggests that 
  a map of Segal spaces $Y \to X$ should via the covariant model structure 
  correspond to a functor with value $\Diag^*(Y \times_{X} X_{/x} )$. 
 \end{remone}

\section{Complete Segal Spaces and Covariant Model Structure} \label{Sec Complete Segal Spaces and Covariant Model Structure}
  In this section we want to study the relation between complete Segal space and left fibrations. 
  In \cref{Subsec Invariance of Covariant Model Structure under CSS Equivalences} we generalize results from 
  \cref{subsec:yoneda lemma Segal spaces} to arbitrary simplicial spaces, using the invariance of the covariant model 
  structure (\cref{The Covar invariant under CSS equiv}). 
  In \cref{Subsec Colimits Cofinality and Quillens Theorem A} we apply our previous results to study colimits in Segal spaces.
 
 \subsection{Invariance of the Covariant Model Structure}
 \label{Subsec Invariance of Covariant Model Structure under CSS Equivalences}
  Until now we have seen several results that suggest a deep connection between complete Segal spaces and left fibrations, 
  in particular over a Segal space (\cref{subsec:yoneda lemma Segal spaces}). 
  In this subsection we want to prove that these results generalize to an arbitrary simplicial space. 
  The key input is the {\it invariance theorem} for the covariant model structure, which proves that the 
  covariant model structure is invariant under equivalences in the complete Segal space model structure. 
  \begin{theone} \label{The Covar invariant under CSS equiv}
   Let $f: X \to Y$ be a CSS equivalence. Then the adjunction
   \begin{center}
     \adjun{(\ss_{/X})^{cov}}{(\ss_{/Y})^{cov}}{f_!}{f^*}
   \end{center}
   is a Quillen equivalence. Here both sides have the covariant model structure.
  \end{theone} 
  
  \begin{remone}
   As the proof is quite long here is an overview of the essential steps:
   \begin{enumerate}
    \item By the diagram in \ref{eq:quillen equivalence}, we can reduce the proof to  
    fibrant replacement maps $i: X \to \hat{X}$.
    \item We first prove the derived counit map is an equivalence in \ref{eq:counit diagram}.
    \item We then move on to the derived unit map.
    By the small object argument the proof reduces to checking for Reedy (\ref{item:reedy}), Segal (\ref{item:segal}) and 
    completeness (\ref{item:complete}) maps, of which only the 
    Segal (\ref{item:segal}) maps require a longer argument. 
    \item By direct computation we can reduce the case for Segal maps to proving that $F(0) \to G(n)$ is a covariant equivalence 
    over $G(n)$ (\ref{eq:proof map}). 
    \item We prove this by induction in (\ref{eq:induction G}).
   \end{enumerate}
  \end{remone}

  \begin{proof}
   Let 
   \begin{center}
    \begin{tikzcd}[row sep=0.6in, column sep=0.6in]
     X \arrow[r, "f", "\simeq_{CSS}"'] \arrow[d, "i"', "\simeq_{CSS}"] & Y \arrow[d, "i'", "\simeq_{CSS}"'] \\
     \hat{X} \arrow[r, "\hat{f}"', "\simeq_{Ree}"] & \hat{Y} 
    \end{tikzcd}
   \end{center}
   be a fibrant replacement of $f$ in the complete Segal space model structure. Then all maps in the diagram are equivalences 
   in the CSS model structure and so the bottom horizontal map is a Reedy equivalence as $\hat{X}$ and $\hat{Y}$ are 
   themselves complete Segal spaces (\cref{The Complete Segal Space Model Structure}). 
   \par 
   This diagram gives us following diagram of adjunctions:
    \begin{equation} \label{eq:quillen equivalence}
    \begin{tikzcd}[row sep=0.5in, column sep=0.5in]
     (\ss_{/X})^{cov} \arrow[r, shift left = 1.8, "f_!", "\bot"'] \arrow[d, shift left = 1.8, "i_!", "\rotatebot"'] & 
     (\ss_{/Y})^{cov} \arrow[l, shift left = 1.8, "f^*"] \arrow[d, shift left = 1.8, "(i')_!", "\rotatebot"'] \\
     (\ss_{/\hat{X}})^{cov} \arrow[r, shift left = 1.8, "\hat{f}_!", "\bot"'] \arrow[u, shift left = 1.8, "i^*"] & 
     (\ss_{/\hat{Y}})^{cov} \arrow[l, shift left = 1.8, "\hat{f}^*"] \arrow[u, shift left = 1.8, "(i')^*"] 
    \end{tikzcd}
    .
  \end{equation}
  By \cref{The Base change covar}, all four are Quillen adjunctions and the bottom horizontal Quillen adjunction is a Quillen equivalence.
  Thus by $2$-out-of-$3$ it suffices to prove that the two vertical Quillen adjunctions are Quillen equivalences. 
  As both Quillen adjunctions are given by a fibrant replacement map, it suffices to prove the left vertical Quillen adjunction $(i_!,i^*)$ 
  is a Quillen equivalence. 
  \par 
  We will prove that the derived unit and derived counit maps are weak equivalences.  
  First we prove that the derived counit map $i_!i^*L \to L$ is a covariant equivalence for every left fibration $p:L \to \hat{X}$
  (notice the derived counit map is the actual counit map as all objects are cofibrant). 
  \par 
  The counit map comes from the diagram 
  \begin{equation} \label{eq:counit diagram}
   \begin{tikzcd}[row sep=0.5in, column sep=0.5in]
    i^*L \arrow[r] \arrow[d] & L \arrow[d, "p"] \\
    X \arrow[r, "i"] & \hat{X}
   \end{tikzcd}
   .
  \end{equation}
  As $p$ is a left fibration over a complete Segal space and $i$ is a CSS equivalence, it follows from 
  \cref{the:pullback of right fib over CSS equiv} that $i^*L \to L$ is a
  CSS equivalence. 
  Finally, by \cref{prop:left fib over Segal is Segal fib}, a CSS equivalence over the CSS $\hat{X}$ is also a 
  covariant equivalence over $\hat{X}$ finishing the proof.
  
  Next we prove that that the derived unit map is an equivalence. For that we first observe that by the small object argument 
  \cite[Proposition 10.5.16]{hirschhorn2003modelcategories}
  the map $i: X \to \hat{X}$  is a transfinite composition
  of pushouts of coproducts of the following three classes of maps 
  \begin{enumerate}
   \item \label{item:reedy} {\bf Reedy:} $(\partial F(n) \to F(n)) \square (\Lambda[l]_i \to \Delta[L])$ 
   \item \label{item:segal} {\bf Segal:} $(G(n) \to F(n)) \square (\partial \Delta[l] \to \Delta[l])$
   \item \label{item:complete} {\bf Complete:} $(F(0) \to E(1)) \square (\partial \Delta[l] \to \Delta[l])$
  \end{enumerate}

  and thus it suffices to prove that the derived unit map is an equivalence when $f$ is one of these three classes of maps. 
  \par 
  The case for Reedy maps follows from \cref{The Base change covar}. 
  The case for complete maps follows from the fact that in the diagram of Quillen adjunctions 
  \begin{center}
   \begin{tikzcd}[row sep=0.5in, column sep=0.5in]
     \Fun([0], \s )^{proj} \arrow[r, shift left = 1.8, "\sint_{[0]}", "\bot"'] \arrow[d, shift left = 1.8, "\{ 0 \}_!", "\rotatebot"'] & 
     \ss^{cov} \arrow[l, shift left = 1.8, "\sH_{[0]}"] \arrow[d, shift left = 1.8, "\{ 0 \}_!", "\rotatebot"'] \\
     \Fun( I[1] , \s )^{proj} \arrow[r, shift left = 1.8, "\sint_{I[1]}", "\bot"'] \arrow[u, shift left = 1.8, "\{ 0 \}^*"] & 
     (\ss_{/E(1)})^{cov} \arrow[l, shift left = 1.8, "\sH_{I[1]}"] \arrow[u, shift left = 1.8, "\{ 0 \}^*"] 
    \end{tikzcd}
   \end{center}
   the horizontal maps are Quillen equivalences (\cref{the:simplicial groth construction}) 
   and the left hand is also a Quillen equivalence ($\{ 0 \}: [0] \to I[0]$ is an equivalence of categories) and so by $2$-out-of-$3$ 
   the left hand vertical adjunction is also a Quillen equivalence.  
  Hence we only need to focus on the case of Segal maps.
  \par 
  To simplify notation we denote the map $(G(n) \to F(n)) \square (\partial \Delta[l] \to \Delta[l])$ by 
  $j_n:G(n,l) \to F(n,l)$. Let $L \to G(n,l)$ be a left fibration. We want to prove the derived unit map $L \to j_n^*R(j_n)_!L$ is a Reedy equivalence.
  We have following diagram:
  \begin{center}
   \begin{tikzcd}[row sep=0.5in, column sep=0.5in]
    L \arrow[ddr, "p"', twoheadrightarrow, bend right = 20] \arrow[dr] \arrow[drrr, bend left=8, "R(j_n)_!" near start, "\simeq_{N\C}" near end]  &[-0.3in] & &  \\[-0.3in]
    & j_n^*(\pi_1)^*(\sbI_{[n]}R\sbT_{[n]} L) \arrow[r, hookrightarrow] \arrow[d, twoheadrightarrow] & 
    (\pi_1)^*\sbI_{[n]}R\sbT_{[n]} L \arrow[r] \arrow[d, twoheadrightarrow] & 
    \sbI_{[n]}R\sbT_{[n]} L \arrow[d, twoheadrightarrow] \\
    & G(n,l) \arrow[r, "j_n", hookrightarrow] & F(n,l) \arrow[r, "\pi_1"] & F(n)
   \end{tikzcd}
   .
  \end{center}
   Here $R \sbT_{[n]}L$ is the projective fibrant replacement of $\sbT_{[n]}L$.
   The map $R(j_n)_!: L \to \sbI_{[n]}R\sbT_{[n]} L$ is the covariant fibrant replacement over $F(n)$ (\cref{cor:fibrant rep via unit}) and thus 
   we only need to prove that $L \to  j_n^*(\sbI_{[n]}R\sbT_{[n]} L \times \Delta[l])$ is a covariant equivalence over $G(n,l)$. 
   By \cref{The Condition for covar equiv between left fib} it suffices to prove they are fiber-wise diagonally equivalent.
   \par 
   Fix an object $m: F(0) \to G(n,l)$. We have following chain of equivalences
   By direct computation for an object $m: F(0) \to G(n,l)$ we have 
   \begin{align*}
    (\pi_1j_n)^*(\sbI_{[n]}R\sbT_{[n]} L ) \times_{G(n,l)} F(0) 
    & \cong \sbI_{[n]}R\sbT_{[n]} L  \times_{F(n)} F(0) & \text{definition of pullback}\\
    & \cong \Nat(N([0] \strut^{ \{ m \} }\times_{[n]} [n]_{/-}) , R\sbT_{[n]} L) &  \text{\cref{def:sbI}} \\
    & \simeq \Nat(\Hom_{[n]}(m,-) , R\sbT_{[n]} L) & \text{computation}\\
    & \simeq R\sbT_{[n]} L(m) & \text{\cref{lemma:hom functor Yoneda}} \\ 
    & \simeq \sbT_{[n]}(m) & \text{\cref{def:projective model structure}} \\
    & \cong \Diag^*(L \times_{F(n)} F(n)_{/m}) & \text{\cref{def:sbT}}
   \end{align*} 
   Thus we need to prove the map 
   \begin{equation} \label{eq:proof map}
    L \times_{G(n,l)} F(0) \to L \times_{F(n)} F(n)_{/m} \cong L \times_{G(n,l)} G(n,l) \times_{F(n)} F(n)_{/m} 
   \end{equation}
   is a diagonal equivalence. 
   \par 
   It suffices to prove that $F(0) \to G(n,l) \times_{F(n)} F(n)_{/m}$ is a contravariant equivalence over $G(n,l)$.
   Indeed, in that case, by \cref{The Pullback preserves covar equiv}, the map \ref{eq:proof map} is also a contravariant equivalence 
   (as $L \to G(n,l)$ is a left fibration) and hence a diagonal equivalence by \cref{The Diag is local of Covar}.
   \par 
   By direct computation $F(n)_{/m} \to F(n)$ is given by $\ordered{0,...,m}: F(m) \to F(n)$ and so we have a bijection 
    $$F(0) \to  G(n,l) \times_{F(n)} F(n)_{/m} \cong G(m,l).$$ 
   This map  is the homotopy pushout (in the Reedy model structure) of the diagram 
   \begin{center}
    \begin{tikzcd}[row sep=0.5in, column sep=0.5in]
     F(0) \times \Delta[l] \arrow[d,hookrightarrow] & 
     F(0) \times \partial \Delta[l] \arrow[d,hookrightarrow] \arrow[r] \arrow[l] & 
     F(0) \times \partial \Delta [l] \arrow[d,hookrightarrow] \\
     G(n) \times \Delta[l] & G(n) \times \partial \Delta[l] \arrow[r] \arrow[l] & F(n) \times \partial \Delta [l]  
    \end{tikzcd}
    .
   \end{center}
    Thus it suffices to prove the vertical maps are contravariant equivalences over $G(n,l)$ 
    (as the contravariant model structure is left proper by \cref{The Covariant Model Structure}). 
    \par 
    The contravariant model structure is simplicial (\cref{The Covariant Model Structure}) and so it suffices to prove that 
    $\ordered{0}: F(0) \to G(m)$ is a contravariant equivalence over $G(n,l)$. 
    The map is a composition of maps $g: G(i) \to G(i+1)$ and thus it suffices to show $g$ is a contravariant equivalence 
    over $G(n,l)$.
    By \cref{The Base change covar}, 
    we can reduce that to proving that  
    $g:G(i) \to G(i+1)$ is a contravariant equivalence over $G(i+1)$.
     \par 
    Finally, we have following pushout square:
    \begin{equation} \label{eq:induction G}
     \begin{tikzcd}[row sep=0.5in, column sep=0.5in]
       F(0) \arrow[r, "\ordered{0}", "\simeq"'] \arrow[d, "\ordered{m-1}"'] & F(1) \arrow[d, "\ordered{m-1,m}"] \\
      G(i) \arrow[r, "\simeq"'] & G(i+1)
     \end{tikzcd}
    \end{equation}
     The top horizontal map is a covariant equivalence by definition, which implies that the bottom horizontal map 
     is a contravariant equivalence over $G(i+1)$.  
  \end{proof}

  \begin{remone}
   This result is also proven by Lurie \cite[Remark 2.1.4.11]{lurie2009htt}, however, there it relies on translating the problem
   into the world of simplicial categories and then proving it there, which we managed to avoid.
   On the other side, it is also proven by Heuts and Moerdijk \cite[Proposition F]{heutsmoerdijk2015leftfibrationi} with simplicial sets, 
   using a very similar approach. 
  \end{remone}
  
  \begin{remone}
   Interestingly enough the result does not hold if we replace ``CSS equivalence" with covariant or contravariant equivalence.
   For that it suffices to look at the simple case of $F(0) \to F(1)$, as the covariant model structure over $F(0)$ is just 
   the diagonal model structure, which is certainly not equivalent to the covariant or contravariant model structure over $F(1)$.
  \end{remone}
  
  \begin{remone}
   There are maps which are not CSS equivalences, but still induce a Quillen equivalence of covariant model structures.
   For more details see \cite[Subsections 4.4.5 and 5.1.4]{lurie2009htt}.
  \end{remone}
  
   We can use the results of this subsection to give a more explicit description of covariant fibrant replacement.
      
   \begin{exone} \label{Ex General Left Fibrant Rep of Point}
   Let $X$ be a simplicial space and $ i: X \to \hat{X}$ be a chosen CSS fibrant replacement of $X$.
   Then for any point $x: F(0) \to X$ the covariant fibrant replacement is given by $F(0) \times_{\hat{X}} \hat{X}^{F(1)} \times_{\hat{X}} X$.
   Indeed, according to \cref{the:yoneda for Segal spaces}, $F(0) \times_{\hat{X}} \hat{X}^{F(1)} \to \hat{X}$ 
   is the covariant fibrant replacement of $x: F(0) \to \hat{X}$
   and then by \cref{The Covar invariant under CSS equiv} the covariant equivalence is preserved by pulling back along $i$.
  \end{exone}

  In \cref{prop:left fib over Segal is Segal fib} we proved that a left fibration over a complete Segal space is a complete Segal space fibration. 
  Using the invariance theorem, \cref{The Covar invariant under CSS equiv}, we can now generalize it to left fibrations over every simplicial space. 
  
  \begin{theone} \label{The Covariant local of CSS}
   Let $X$ be a simplicial space. Then the following adjunction
   \begin{center}
    \adjun{(\ss_{/X})^{CSS}}{(\ss_{/X})^{cov}}{id}{id}
   \end{center}
   is a Quillen adjunction.
   Here the left hand side has the induced CSS model structure (Definition \ref{prop:induced model structure})
   and the right hand side has the covariant model structure.
   \par 
   This implies that the covariant model structure over $X$ is a localization of the induced CSS model structure over $X$.
  \end{theone}
  \begin{proof}
   We want to prove that the left adjoint preserves cofibrations and trivial cofibrations. 
   Clearly it preserves cofibrations as they are just the monomorphisms. 
   Let $i: Y \to Z$ be a trivial CSS cofibration over $X$.
   Then, \cref{The Covar invariant under CSS equiv} gives us a Quillen equivalence
   \begin{center}
    \adjun{(\ss_{/Y})^{cov}}{(\ss_{/Z})^{cov}}{i_!}{i^*}.
   \end{center}
   Thus, in particular, the counit map $i_!i^*Z \to Z$ is a covariant equivalence over $Z$. 
   However, we have $i_!i^*(\id_Z) = i: Y \to Z$, which means $i$ is a covariant equivalence over $Z$. 
   Finally, by \cref{The Base change covar}, $i$ is also a covariant equivalence over $X$.
  \end{proof}

  \begin{remone}
    Note that we can use the same proofs with the contravariant model structure to show that the contravariant model structure
    is a localization of the CSS model structure as well.
  \end{remone}
  The result above has the following very important corollary.
  \begin{corone} \label{Cor Left Fib is CSS Fib}
   Every left (and right) fibration is a CSS fibration.
  \end{corone}

  \begin{remone}
   This result generalizes \cite[Subsection 1.4]{debrito2018leftfibration}, which proved that left fibrations over Segal spaces are 
   CSS fibrations. 
   Lurie proves the same result over an arbitrary simplicial set \cite[Theorem 3.1.5.1]{lurie2009htt}, but relies on the 
   straightening construction.
  \end{remone}
  
  In \cref{the:pullback of right fib over CSS equiv} we proved that pulling back along right and left fibrations over a CSS preserves CSS equivalences. 
  Using the invariance theorem, \cref{The Covar invariant under CSS equiv}, we can generalize this to arbitrary right and left fibrations.
  
   \begin{theone} \label{The Pullback preserves CSS equiv}  
   Let $p:R \to X$ be a right or left fibration. Then the adjunction
    \begin{center}
     \adjun{(\ss_{/X})^{CSS}}{(\ss_{/R})^{CSS}}{p^*}{p_*}
    \end{center}
    is a Quillen adjunction. Here both sides have the induced CSS model structure 
    (\cref{prop:induced model structure}).
  \end{theone}
  \begin{proof}
   We will assume that $p$ is a right fibration, the case for left fibrations follows similarly. 
   
   In order to prove $(p^*,p_*)$ is a Quillen adjunction is suffices to prove $p^*$ preserves cofibrations and 
   trivial cofibrations. Evidently, $p^*$ preserves cofibrations as they are just inclusion. 
   Hence we are left with proving that for a given trivial CSS cofibration $Y \to Z$ over $X$, 
   the map $Y \times_X R \to Z \times_X R$ is a CSS equivalence over $R$, which, by \cref{prop:induced model structure},
   is equivalent to $Y \times_X R \to Z \times_X R$ being a CSS equivalence. 
   
   By \cref{The Covar invariant under CSS equiv} we have following diagram 
   \begin{center}
    \begin{tikzcd}[row sep=0.5in, column sep=0.5in]
     R \arrow[dr, "j" near end , "\simeq" description] 
     \arrow[ddr, "p"', twoheadrightarrow, near end, bend right =25] \arrow[drr, "i_R" near end, "\simeq" description, bend left =25] 
     & & \\
     & X \times_{\hat{X}} \hat{R} \arrow[r] \arrow[d, twoheadrightarrow] & \hat{R} \arrow[d, "\hat{p}", twoheadrightarrow] \\
     & X \arrow[r, "i_X" near start, "\simeq" description] & \hat{X}
    \end{tikzcd}

   \end{center}
   where $\hat{p}:\hat{R} \to \hat{X}$ is a CSS fibrant replacement of $p: R \to X$ (and a right fibration) 
   and the map $j: R \to X \times_{\hat{X}} \hat{R}$ is a Reedy equivalence. 
   
   We can use the Reedy equivalence $j$ to form following diagram: 
  \begin{center}
   \begin{tikzcd}[row sep=0.5in, column sep=0.5in]
    Y \underset{X}{\times} R \arrow[r] \arrow[d, "\simeq", "Y \underset{X}{\times} j"'] & 
    Z \underset{X}{\times} R \arrow[d, "\simeq"', "Z \underset{X}{\times} j"] \\
    Y \underset{X}{\times} (X \underset{\hat{X}}{\times} \hat{R}) \arrow[r] \arrow[d, "\cong"] & 
    Z \underset{X}{\times} (X \underset{\hat{X}}{\times} \hat{R}) \arrow[d, "\cong"'] \\
    Y \underset{\hat{X}}{\times} \hat{R} \arrow[r] & Z \underset{\hat{X}}{\times} \hat{R}
   \end{tikzcd}
  .
  \end{center}
  The vertical maps in the top square are Reedy equivalences as the Reedy model structure is right proper (\cref{Subsec Reedy Model Structure}).
  The vertical map in the bottom square are isomorphisms by definition of pullbacks. 
  This implies that in order to show the top map is an equivalence it suffices to show the bottom map is an equivalence. 
  However, $\hat{p}: \hat{R} \to \hat{X}$ is a right fibration over a complete Segal space 
  and so the result follows from \cref{the:pullback of right fib over CSS equiv}.
 \end{proof}
 
 The exact same proof can be used to prove the following theorem.
 
  \begin{theone} \label{The Pullback preserves Seg equiv}
   Let $p:R \to X$ be a right or left fibration. Then the adjunction
    \begin{center}
     \adjun{(\ss_{/X})^{Seg}}{(\ss_{/R})^{Seg}}{p^*}{p_*}
    \end{center}
    is a Quillen adjunction. Here both sides have the induced Segal space model structure 
    (\ref{prop:induced model structure}).
  \end{theone}

 \begin{remone}
  This same result is stated in \cite[Remark 11.10]{joyal2008theory} in the language of quasi-categories, however without a proof.
 \end{remone}
 
 The theorem has following helpful corollary.
 
 \begin{corone}
  Let $X \to Y$ be a CSS equivalence and $F \to Y$ either a right or left fibration over $Y$.
  Then the map $X \times_Y F \to F$ is also a CSS equivalence.
 \end{corone}
 
 This result is indeed helpful, as the CSS model structure is {\it not} right proper i.e. generally weak equivalences are not preserved
 by pullbacks.
 We can easily see this in the following example.
 
 \begin{exone}
  The map $G(2) \to F(2)$ is a Segal equivalence. 
  Let $F(1) \to F(2)$ be the unique map that takes $0$ to $0$ and $1$ to $2$.
  Note that this map is a CSS fibration but neither a left fibration nor a right fibration.
  Now the pullback
  $$F(1) \underset{F(2)}{\times} G(2) \to F(1)$$
  is clearly not a Segal equivalence as the left hand side is just $F(0) \coprod F(0)$.
 \end{exone}
 
 \subsection{Colimits, Cofinality and Quillen's Theorem A} \label{Subsec Colimits Cofinality and Quillens Theorem A}
  One intricate subject in the theory of $(\infty,1)$-categories is the study of limits and colimits. 
  Accordingly, there now many sources dedicated to the study of colimits of quasi-categories, such as \cite{joyal2008notes}, \cite{lurie2009htt}.
  Thus if we are interested in understanding colimits in a complete Segal space, we can 
  translate those results from quasi-categories to complete Segal spaces using the work of Joyal and Tierney \cite{joyaltierney2007qcatvssegal}, 
  or specialize the model independent approach to colimits via $\infty$-cosmoi 
  to the case of complete Segal spaces \cite{riehlverity2017inftycosmos}.
  
  In \cref{subsec:why simplicial spaces?} we discussed how  we want to understand whether results about $(\infty,1)$-categories, 
  such as various results about their colimits, still hold if we drop the completeness condition. 
  This cannot be directly translated from the corresponding results about quasi-categories or 
  $\infty$-cosmoi and needs to be proven directly. 
  
  In this final section we will apply our knowledge about left fibrations of Segal spaces to prove several classical results about 
  colimits in the context of Segal spaces. In particular, we prove that a cocone out of a diagram corresponds to a map out of its colimits
  (\cref{The Maps out of cone point give us map out of colimit}), 
  final maps give us equivalent colimits (\cref{The Cofinal maps give same cocones}) and Quillen's theorem A 
  for simplicial spaces (\cref{The Quillen Theorem A}) and in particular Segal spaces (\cref{cor:Quillen Thm A Segal}).

  \begin{remone}
   In this section we focus on using left fibrations to study colimits. 
   We can analogously study limits via right fibrations. 
  \end{remone}
  
  \begin{defone} \label{Def Cocones under diagram}
  Let $X$ be a Segal space and $p:K \to X$ be a map of simplicial spaces. We define the Segal space of {\it cocones under $K$},
  denoted by $X_{p/}$, as the pullback
  \begin{center}
   \begin{tikzcd}[row sep=0.5in, column sep=0.5in]
    X_{p/} \arrow[r] \arrow[d, "\pi"] & (X^K)^{F(1)} \arrow[d, "(s \comma t)"] \\
    X \cong F(0) \times X \arrow[r, "\{ p \} \times \Delta"] & X^K \times X^K
   \end{tikzcd}
  \end{center}
  where $\Delta: X \to X^K$ is the map induced by the final map $K \to F(0)$.
 \end{defone}
 
 \begin{remone}
  The definition here differs from the definition usually given for quasi-categories \cite[Proposition 1.2.9.2]{lurie2009htt}.
  However, the two definitions are in fact equivalent. Indeed, this follows from the fact that the covariant model structures of 
  quasi-categories and complete Segal spaces are equivalent, as proven in \cref{Sec Comparison with Quasi-Categories}.
 \end{remone}

   \begin{exone}
   If $X$ is a Segal space and $K = F(0)$ then $p$ is determined by a choice of point $x$ in $X$ and we have $X_{p/} = X_{x/}$, 
   the Segal space of objects under $x$, as defined in \cref{def:under Segal space}.  
  \end{exone}
  
 \begin{lemone} \label{Lemma Cocones gives us left fibrations}
  Let $X$ be a Segal space and $p: K \to X$ be a map of simplicial spaces. 
  The projection map 
  $$\pi : X_{p/} \to X$$
  is a left fibration.
 \end{lemone}
 \begin{proof}
  In the following pullback diagram
  \begin{center}
   \begin{tikzcd}[row sep=0.5in, column sep=0.5in]
    X_{p/} \arrow[r] \arrow[d, twoheadrightarrow] \arrow[dr, phantom, "\ulcorner", very near start]
    & F(0) \underset{X^K}{\times} (X^K)^{F(1)} = (X^K)_{p/}\arrow[d, twoheadrightarrow] \\
    X \arrow[r]
    & X^K
  \end{tikzcd}
  \end{center}
 the right vertical map is a left fibration, by \cref{the:under CSS left fibration}, and so the left vertical map must be a left fibration as well.
 as left fibrations are closed under pullbacks (\cref{Lemma Pullback preserves Left fibrations}).
 \end{proof}
 
 \begin{defone} \label{Def Of limits}
   Let $X$ be a Segal space and $p:K \to X$ a map of simplicial spaces. 
   We say $p$ has a colimit if the Segal space $X_{p/}$ has an initial object (\cref{def:initial object}).
  \end{defone}
  
  \begin{remone}
   This approach to colimits has already been studied for quasi-categories \cite[Definition 1.2.13.4]{lurie2009htt}.
  \end{remone}

  The next lemma can help us better understand the definition of a colimit.
  
  \begin{lemone} \label{lemma:limit conditions}
   Let $X$ be a Segal space and $p: K \to X$ be a map of simplicial spaces. Then the following are equivalent:
   \begin{enumerate}
    \item $p$ has a colimit.
    \item There is a covariant equivalence $\{ \sigma \}: F(0) \to X_{p/}$ over $X$.
    \item There is a Reedy equivalence 
    $$X_{v/} \to X_{p/}$$
    where $v$ is an object in $X$.
   \end{enumerate}
  \end{lemone}
  
  \begin{proof}
   
   $(1 \Rightarrow 2)$
   If $p$ is a colimit, then $X_{p/}$ has an initial object $\sigma$ which, by \cref{def:initial object}, means 
   $\{ \sigma \} : F(0) \to X_{p/}$ is a covariant equivalence over $X$.
   
   \medskip
   
   $(2 \Rightarrow 3)$
   Assume we have a covariant equivalence $\{ \sigma \} : F(0) \to X_{p/}$ over $X$ and let $\{ v \} = \pi \circ \{ \sigma \}: F(0) \to X$.
   Then the commutative square
   \begin{center}
    \liftsq{F(0)}{X_{p/}}{X_{v/}}{X}{\simeq}{\simeq}{}{}
   \end{center}
   has a lift. Indeed, $X_{p/} \to X$ is a left fibration (\cref{Lemma Cocones gives us left fibrations}) 
   and so a fibration in the covariant model structure and 
   $F(0) \to X_{v/}$ is a covariant equivalence over $X$ (\cref{the:yoneda for Segal spaces}).
   As the top and left hand maps are covariant equivalences over $X$, the map $X_{v/} \to X_{p/}$ is also a covariant equivalence over $X$.
   Finally, as both $X_{v/}$ and $X_{p/}$ are left fibrations over $X$, the covariant equivalence is in fact a Reedy equivalence 
   (\cref{The Covariant Model Structure}).
   
   \medskip 
   
   $(3 \Rightarrow 1)$ By \cref{the:rep equiv initial obj}, 
   $X_{v/}$ has an initial object and so by the Reedy equivalence $X_{p/}$ also has an initial object. 
  \end{proof}
  
  We call the object $\sigma$ in $X_{p/}$ the {\it universal cocone} and, by abuse of language, the object $v$ in $X$ the {\it colimit}.
  
   For the remainder of this section we want to use our knowledge of left fibrations to study colimits of Segal spaces. 
   First, we prove the Segal space analogue to the fact that maps out of a colimit are equivalent to maps of cocones.

  \begin{theone} \label{The Maps out of cone point give us map out of colimit}
   Let $X$ be a Segal space and $p: K \to X$ be a map of simplicial spaces and assume it has universal
   cocone $\sigma: F(0) \to X_{p/}$ with colimit $v$.
   Then for any object $y$ in $X$ we have a natural equivalence 
   $$\comp : map_{X}(v,y) \xrightarrow{ \ \ \simeq \ \ } map_{X^K}(p, \{ y \}).$$
  \end{theone}
  
  \begin{proof}
    By \cref{lemma:limit conditions}, we have a Reedy equivalence $X_{v/} \to X_{p/}$ over $X$.
    Fix an object $y: F(0) \to X$, which gives us a point $y: \Delta[0] \to X_0$. Then, using \cref{def:mapping space Segal space},
    we get a Kan equivalence 
    $$\comp: \map_X(v,y) = \Map_{/X}(F(0),X_{v/}) \to \Map_{/X}(F(0),X_{p/}) = \map_{X^K}(p,\{ y \} )$$
    and hence we are done.
  \end{proof}

   We now move on to study computational aspects of colimits in Segal spaces. 
   In particular, we introduce final maps and prove they give us equivalent colimits. 
  
   \begin{defone} \label{Def Cofinal map}
    A map $f: X \to Y$ is called {\it final} if $f$ is a contravariant equivalence over $Y$.
    Similarly, $f: X \to Y$ is called {\it initial} if $f$ is a covariant equivalence over $Y$.
   \end{defone} 
   
   \begin{remone}
    The notion of final maps of quasi-categories was first studied in \cite[8.11]{joyal2008notes}.
    They also have been studied by Lurie \cite[Definition 4.1.1.1]{lurie2009htt} where they are called cofinal.
   \end{remone}

   Initial and final maps are a generalization of initial and final objects.
   
  \begin{exone}
   A map $\{x\}: F(0) \to X$ is initial in the sense of \cref{Def Cofinal map} if and only if 
   the object $x$ is initial in the sense of \cref{def:initial object}.
  \end{exone}

  \begin{lemone} \label{Lemma Condition for final maps}
   Let $f: X \to Y$ be a map of simplicial spaces. The following are equivalent:
   \begin{enumerate}
    \item $f$ is a final map.
    \item For any map $g:Y \to Z$ the map $f$ is a contravariant equivalence over $Z$.
    \item For any right fibration $R \to Y$ the induced map 
    $$\Map_{/Y}(Y,R) \to \Map_{/Y}(X,R)$$
    is a Kan equivalence.
   \end{enumerate}
  \end{lemone}
  \begin{proof}
     {\it (1 $\Rightarrow$ 2)}
    As $f: X \to Y$ is a covariant equivalence over $Y$, $g_!(X \to Y)$ is a covariant equivalence over $Z$ (\cref{The Base change covar}).
    
    \medskip 
    
    {\it (2 $\Rightarrow$ 1)}
    This is just a special case where $g = \id_Y$.
    
    \medskip 
    
    {\it (1 $\Leftrightarrow$ 3)}
    Follows from the definition of contravariant equivalence (\cref{The Covariant Model Structure}).
  \end{proof}
  
  \begin{remone} \label{Rem Cofinal maps are terrible}
   Although final maps are defined as certain contravariant equivalences, they do not always behave similar to weak equivalences
   in the model categorical sense.
   In particular, they do not satisfy the $2$-out-of-$3$ property. For example, in the chain
   $$F(0) \xrightarrow{ \ordered{0} } F(1) \xrightarrow{ \ \ \ordered{0,0} \ \ } F(0)$$
   the map $\ordered{0,0}$ and the composition $\ordered{0}$ are final, but $\ordered{0}: F(0) \to F(1)$ is not. 
  \end{remone}
  
  \begin{corone} \label{Cor CSS equiv is final}
   If $f:X \to Y$ is a CSS equivalence then it is final.
  \end{corone}
  \begin{proof}
   This follows directly from \cref{The Covariant local of CSS}.
  \end{proof}
  
  Before we move on let us note that this gives us one exception to \cref{Rem Cofinal maps are terrible}.
  
  \begin{lemone}
   Let $X \xrightarrow{ \ f \ } Y \xrightarrow{ \ g \ } Z$ be a chain of maps such that $g$ is a CSS equivalence.
   Then $f$ is a final map if and only if $gf$ is a final map.
  \end{lemone}
  \begin{proof}
   First note that $g$ is a final map. So, if $f$ is final then simple composition implies
   that $gf$ is also a final map.
   
   On the other side, let us assume $gf= g_!(f)$ is a final map. 
   Then, by \cref{The Covar invariant under CSS equiv}, the following adjunction is a Quillen equivalence:
   \begin{center}
    \adjun{(\ss_{/Y})^{contra}}{(\ss_{/Z})^{contra}}{g_!}{g^*}
   \end{center}
   which implies that $f: X \to Y$ is a contravariant equivalence over $Y$ if and only if $g_!(f): X \to Y$ is a contravariant equivalence over $Z$, 
   and so $f$ is a final map as well. 
  \end{proof}
  
  Having discussed final maps we can now show how it allows us to simplify colimit diagrams.
  
 \begin{theone} \label{The Cofinal maps give same cocones}
  Let $g: A \to B$ be a final map and $X$ be a Segal space. Then for any map $f: B \to X$ the induced map 
  $$X_{f/} \to X_{fg/}$$
  is a Reedy equivalence. 
 \end{theone}
 
 \begin{proof}
  Fix an object $x$ in the Segal space $X$.
  Using adjunctions between products and exponentials, we have following bijections
  \begin{equation} \label{eq:bijections}
   \Map(F(1), X^B) \cong \Map(F(1) \times B, X) \cong \Map(B,X^{F(1)})
  \end{equation}
   Let $\Map^{res}(F(1) \times B,X)$ be the full subspace of $\Map(F(1) \times B,X)$ consisting of maps $H: F(1) \times B \to X$ such that 
   $H|_{\{0\}} = f: B \to X$ and $H|_{\{1\}} = \{x\}: B \to X$.  Then restricting the two bijection in \ref{eq:bijections} gives us 
   the bijections

   $$\Map_{/X}(F(0), X_{f/}) \cong \Map^{res}(F(1) \times B, X) \cong \Map_{/X}(B,X_{/x}).$$
 
  This action is functorial and hence the map $g: A \to B$ gives us a commutative diagram:
  \begin{center}
   \begin{tikzcd}[row sep=0.5in, column sep=0.5in]
    \Map_{/X}(F(0), X_{f/}) \arrow[r, "g_*"] \arrow[d, "\cong"] & \Map_{/X}(F(0), X_{fg/}) \arrow[d, "\cong"] \\
    \Map_{/X}(B,X_{/x}) \arrow[r, "f^*", "\simeq"'] &  \Map_{/X}(A,X_{/x})
   \end{tikzcd}
   .
  \end{center}
   By the explanation above the vertical maps are bijections. Moreover, $X_{/x} \to X$ is a right fibration (\cref{the:under CSS left fibration}) and 
   so, by \cref{Def Cofinal map}, the bottom map is a Kan equivalence. Hence the top map is a Kan equivalence for every map $x: F(0) \to X$. 
   As $g_*: X_{f/} \to X_{fg/}$ is a map of left fibrations over $X$, this implies that it is a Reedy equivalence (\cref{The Condition for covar equiv between left fib}).
 \end{proof}

 \begin{corone}
  Let $X$ be a Segal space and $g: A \to B$ be a final map of simplicial spaces.
  Then a map $f: B \to X$ has a colimit if and only if $gf: A \to X$ has a colimit 
  and in that case they are equivalent objects in $X$. 
 \end{corone}

 We end this section by giving a useful criterion for classifying final maps of simplicial spaces, motivated by 
 Quillen's Theorem A \cite{quillen1972higherktheories}.
 
  \begin{theone} \label{The Quillen Theorem A}
   Let $f: X \to Y$ be a map of simplicial spaces and $\{L_y \to Y\}_{\{y\}:F(0) \to Y}$ a collection of covariant fibrant replacements of 
   $\{y\}:F(0) \to Y$. The following are equivalent:
   \begin{enumerate}
    \item $f$ is a final map.
    \item For any $y:F(0) \to Y$, the simplicial space $L_y \times_Y X$ is diagonally contractible.
   \end{enumerate}
  \end{theone}
  \begin{proof}
    By \cref{the:covar equiv over simp space}, $f:X \to Y$ is a contravariant equivalence over $Y$ if and only if 
    $$L_y \times_Y f: L_y \times_Y X \to L_y \times_Y Y \cong L_y$$ 
    is a diagonal equivalence. 
    By assumption $F(0) \to L_y$ is a covariant equivalence over $Y$ and thus, by \cref{The Diag is local of Covar}, a diagonal equivalence.
    Hence, $L_y \times_Y f$ is a diagonal equivalence if and only if $L_y \times_Y X$ is diagonally contractible.   
  \end{proof}
   
   In the case of Segal spaces we can simplify the statement.
   
   \begin{corone} \label{cor:Quillen Thm A Segal}
    Let $Y$ be a Segal spaces and $f: X \to Y$ be a map of simplicial spaces. Then $f$ is final if and only if for every object $y$ in $Y$ 
    the simplicial space $Y_{y/} \times_Y X$ is diagonally contractible. 
   \end{corone}

  \begin{remone}
   This result was proven for quasi-categories in \cite[Theorem 4.1.3.1]{lurie2009htt} where it is attributed to Joyal.
   Thus, \cref{cor:Quillen Thm A Segal} generalizes Quillen's theorem $A$ to Segal spaces.
  \end{remone}
  
\appendix

\section{Some Facts about Model Categories} \label{Sec Some Facts about Model Categories}

 We primarily used the theory of model categories to tackle issues of higher category theory.
 In this section we will not introduce model categories as they are already several excellent sources. 
 For instance, we refer the reader to \cite{dwyersspalinski1995modelcat} for a short introduction to this subject and to 
 \cite{hovey1999modelcategories}, \cite{hirschhorn2003modelcategories} for a more detailed discussion.
 Here we will only state some technical lemmas we have used throughout. 
 \par 
 First, we make ample use of following very classical result about trivial Kan fibrations.
 
 \begin{lemone} \label{Lemma Triv Kan fib trivial fibers}
  Let $p: S \to T$ be a Kan fibration in $\s$. Then $p$ is a trivial Kan fibration if and only if
  each fiber of $p$ is contractible. 
 \end{lemone}
 
 For a readable proof of this statement see \cite[Section 38]{rezk2017qcats}.
 This lemma has the following important corollary.
 
 \begin{corone} \label{Cor Fiberwise equiv for Kan fib}
  Let $p:S \to K$ and $q:T \to K$ be two Kan fibrations. A map $f: S \to T$ over $K$ is a Kan equivalence 
  if and only if for each point $\{ k \}: \Delta[0] \to K$ the map between fibers 
  $$S \underset{K}{\times} \{ k \} \to T \underset{K}{\times} \{ k \}$$
  is a Kan equivalence.
 \end{corone}

 Next we have two important results that help us study Quillen adjunctions.
 
 \begin{lemone} \label{Lemma For Quillen adj}
  (\cite[Proposition 7.15]{joyaltierney2007qcatvssegal}, \cite[Proposition 8.5.4]{hirschhorn2003modelcategories}) 
  Let $\mathcal{M}$ and $\mathcal{N}$ be two model categories and 
  \begin{center}
   \adjun{\mathcal{M} }{\mathcal{N}}{F}{G}
  \end{center}
  be an adjunction of model categories, then the following are equivalent:
  \begin{enumerate}
   \item $(F,G)$ is a Quillen adjunction.
   \item $F$ takes cofibrations to cofibrations and $G$ takes fibrations between
   fibrant objects to fibrations.
   \item $G$ preserves trivial fibrations and takes fibrations between fibrant objects to fibrations. 
  \end{enumerate}
 \end{lemone}
 
 \begin{lemone} \label{Lemma For Quillen equiv}
 (Special case of 
  \cite[Proposition 7.17, Proposition 7.22]{joyaltierney2007qcatvssegal})
  Let
  \begin{center}
   \adjun{\mathcal{M}}{\mathcal{N} }{F}{G}
  \end{center}

  be a Quillen adjunction of model categories. Then the following are equivalent:
  \begin{enumerate}
   \item $(F,G)$ is a Quillen equivalence.
   \item $F$ reflects weak equivalences between cofibrant objects and 
   the derived counit map $FLG(n) \to n$ is an equivalence for every fibrant-cofibrant object $n \in \mathcal{N}$
   (Here $LG(n)$ is a cofibrant replacement of $G(n)$ inside $\mathcal{M}$).
   \item $G$ reflects weak equivalences between fibrant objects and 
   the derived unit map $m \to GRF(m)$ is an equivalence for every fibrant-cofibrant object $m \in \mathcal{M}$
   (Here $RF(m)$ is a fibrant replacement of $F(m)$ inside $\mathcal{N}$).
  \end{enumerate}
 \end{lemone}

 We move on to the main topic of this appendix, namely the existence of two localized model structures on 
 the category of simplicial spaces over a fixed simplicial space and their comparison.
 
 \begin{propone} \label{prop:induced model structure}
  \cite[Theorem 7.6.5]{hirschhorn2003modelcategories}
   Let $\mathcal{M}$ be a model structure on $\ss$.
   Let $X$ be a simplicial space. There is a simplicial model structure on $\ss_{/X}$, which we call the {\it induced model structure}
   and denote by $(\ss_{/X})^{\mathcal{M}}$, and which satisfies following conditions:
   \begin{itemize}
    \item[(F)] A map $f: Y \to Z$ over $X$ is a (trivial) fibration if $Y \to Z$ is a (trivial) fibration 
    \item[(W)] A map $f:Y \to Z$ over $X$ is a weak equivalence if $Y \to Z$ is a weak equivalence
    \item[(C)] A map $f:Y \to Z$ over $X$ (trivial) cofibration if $Y \to Z$ is a (trivial) cofibration.
   \end{itemize}
 \end{propone}
 
  \begin{remone}
  The induced model structure can be defined for any model category and not just for model structures on $\ss$, but for our work 
  there is no need for further generality.
 \end{remone}
 
 \begin{theone} \label{The Localization for sS}
  Let $X$ be a simplicial space and $\mathcal{L}$ be a set of monomorphisms in $\ss_{/X}$.
  There exists a cofibrantly generated, simplicial model category structure on $\ss_{/X}$ with the following properties:
  \begin{enumerate}
   \item The cofibrations are exactly the monomorphisms.
   \item The fibrant objects (called $\mathcal{L}$-local objects) are exactly the Reedy fibrations $W \to X \in \ss$ such that 
   $$\Map_{/X}(B, W) \to \Map_{/X}(A, W)$$
   is a weak equivalence of spaces for all maps $f: A \to B$ over $X$ in $\mathcal{L}$.
   \item The weak equivalences (called $\mathcal{L}$-local weak equivalences) are exactly the
   maps $g: Y \to Z$ over $X$ such that for every $\mathcal{L}$-local object $W \to X$ , the induced map
   $$\Map_{/X}(Z, W) \to \Map_{/X}(Y, W)$$
   is a weak equivalence.
   \item Let $f: Y \to Z$ be a map over $X$. 
   \begin{itemize}
    \item If $f$ is a Reedy weak equivalence over $X$ then it is a $\mathcal{L}$-local weak equivalence.
    \item If $f$ is a $\mathcal{L}$-local fibration then it is a Reedy fibration.
   \end{itemize}
   The opposite of both statements hold if $Y$ and $Z$ are $\mathcal{L}$-local.
  \end{enumerate}
  We call this model category the {\bf localized model structure}.
 \end{theone}
 
 The model structure is given as the left Bousfield localization of the induced Reedy model structure on $\ss_{/X}$.
 For a careful and detailed proof of the existence of left Bousfield localizations see \cite[Theorem 4.1.1]{hirschhorn2003modelcategories}
 (notice we are using the fact that the induced Reedy model structure on $\ss_{/X}$ is proper and cellular
 \cite[Proposition 12.1.6]{hirschhorn2003modelcategories}).
 For a nice summary of this proof that goes over the main steps see \cite[Proposition 9.1]{rezk2001css}.

 \begin{remone}
  Notice, we can in particular take $X$ to be the final object in which case the theorem gives us a localization model structure 
  of the Reedy model structure on $\ss$.
 \end{remone}
 
 Note any such model structure is invariant under Reedy equivalences.
 
 \begin{lemone} \label{lemma:localized model structure base change}
  Let $\mathcal{L}$ be a set of monomorphisms in $\ss$ and let $f: X \to Y$ be a map of simplicial spaces. 
  Then the adjunction 
  \begin{center}
   \adjun{(\ss_{/X})^{\mathcal{M}_X}}{(\ss_{/Y})^{\mathcal{M}_Y}}{f_!}{f^*}
  \end{center}
  is a Quillen adjunction which is a Quillen equivalence if $f$ is a Reedy equivalence. 
  Here the left hand side has the localization model structure with respect to maps $A \to B \to X$ for all maps $A \to B$ in $\mathcal{L}$ and 
  the right hand side has the localization model structure with respect to maps $A \to B \to Y$ for all maps $A \to B$ in $\mathcal{L}$.
 \end{lemone}

 \begin{proof}
  First we use \cref{Lemma For Quillen adj} to prove it is a Quillen adjunction. Clearly the left adjoint $f_!$ preserves monomorphisms 
  and thus weak equivalences. On the other hand, fibrations between fibrant objects are just Reedy fibrations, which are preserved by $f^*$. 
  Thus we only need to prove that $f^*$ preserves fibrant objects. 
  However, the fibrant objects are just Reedy fibrations which satisfy a right lifting property with respect to maps $A \to B$ in 
  $\mathcal{L}$ and the class of such maps is clearly closed under pullback.
  \par 
  Next we assume that $f$ is a Reedy equivalence and prove that the adjunction is a Quillen equivalence. We 
  will prove that the derived unit and derived counit maps are equivalences. 
  Let $p: Z \to X$ be an arbitrary map. The derived unit map, $f^*Rf_!$, is given by taking the Reedy fibrant replacement of the map 
  $Z \to X \to Y$ in $\ss_{/Y}$ and then pulling it back along $f: X \to Y$. 
  This can be depicted as the following diagram:
  \begin{center}
   \begin{tikzcd}[row sep=0.5in, column sep=0.3in]
    Z \arrow[dr, "p"'] \arrow[rrrrr, "i", bend left = 30, "\simeq"'] \arrow[rr, "u"] & & 
    f^*\hat{Z} \arrow[dl, "f^*f_!(\hat{p})"] \arrow[rrr, "\simeq"] & & & \hat{Z} \arrow[d, "f_!(\hat{p})"] \\
    & X \arrow[rrrr, "f", "\simeq"'] & & & & Y
   \end{tikzcd}
  \end{center}
  where $i: Z \to \hat{Z}$ is the Reedy fibrant replacement. 
  The map $f^*\hat{Z} \to \hat{Z}$ is a Reedy equivalence as Reedy equivalences are preserved by pullbacks. 
  Thus, by $2$-out-of-$3$, the derived unit map $u: Z \to f^*\hat{Z}$ is a Reedy weak equivalence.
  \par 
  We move on to the derived counit map. As all objects are cofibrant the derived counit map is simply given by the actual counit map.
  Let $W \to Y$ be an arbitrary map. Then we have the diagram 
  \begin{center}
   \begin{tikzcd}[row sep=0.5in, column sep=0.5in]
    f^*W \arrow[r, "\simeq"] \arrow[d] & W \arrow[d] \\
    X \arrow[r, "f", "\simeq"'] & Y 
   \end{tikzcd}
   .
  \end{center}
  As the Reedy model structure is right proper (\cref{Subsec Reedy Model Structure}) and thus Reedy weak equivalences are preserved by pullback, 
  the counit map $f^*W \to W$ is a Reedy equivalence and so also also an equivalence in the localized model structure 
  (\cref{The Localization for sS}). 
 \end{proof}

 Our precise understanding of the fibrant objects in the localized model structure allows us to simplify the conditions in 
 \cref{Lemma For Quillen adj}.
 
 \begin{corone} \label{Cor Quillen adj for localizations}
  Let $X$ be a simplicial space and
  let $(\ss_{/X},\mathcal{M})$ and $(\ss_{/X},\mathcal{N})$ be two localized model structures of the induced Reedy model structure.
  Then an adjunction 
  \begin{center}
   \adjun{(\ss_{/X})^{\mathcal{M}}}{(\ss_{/X})^{\mathcal{N}}}{F}{G}
  \end{center}
  is a Quillen adjunction if it satisfies the following conditions:
  \begin{enumerate}
   \item $F$ takes cofibrations to cofibrations.
   \item $G$ takes fibrants to fibrants.
   \item $G$ takes Reedy fibrations to Reedy fibrations.
  \end{enumerate}
 \end{corone}
 
 \begin{remone} \label{rem:localized vs induced localized} 
 Let $X$ be a simplicial space and $\mathcal{L}$ be a set of monomorphisms over $X$. Then we can now construct two 
 model structures on $\ss_{/X}$ using $\mathcal{L}$:
 \begin{enumerate}
  \item Using \cref{The Localization for sS} we can construct a localized model structure on the induced model structure $\ss_{/X}$.
  This is the localized model structure on $\ss_{/X}$.
  \item We can project $\mathcal{L}$ to $\ss$ to get a set of monomorphisms in $\ss$ and then use \cref{The Localization for sS} 
  to construct a localized model structure on $\ss$ and finally take the induced model structure on $\ss_{/X}$.
  We call this the {\it induced localized model structure} on $\ss_{/X}$.
 \end{enumerate}
 \end{remone}
 
 We want to understand how these two model structures compare to each other.
 Before we can do that we need a precise characterization of the fibrant objects in the localized model structure on $\ss_{/X}$.
 
 \begin{lemone} \label{lemma:fibrant objects in localized model structure}
  Let $X$ be a simplicial space and $\mathcal{L}$ a set of monomorphisms. 
  An object $p: Y \to X$ in $\ss_{/X}$ is fibrant in the localized model structure if and only if 
  it is a Reedy fibration and for every morphism $f: A \to B$ in $\mathcal{L}$ the commutative square
  \begin{center}
   \begin{tikzcd}[row sep=0.5in, column sep=0.5in]
    \Map(B,Y) \arrow[r, "f^*"] \arrow[d, "p_*"] & \Map(A,Y) \arrow[d, "p_*"] \\
    \Map(B,X) \arrow[r, "f^*"] & \Map(A,X)
   \end{tikzcd}
  \end{center}
  is a homotopy pullback square of spaces. 
 \end{lemone}

 \begin{proof}
  The square is a homotopy pullback square if and only if the horizontal map below is an equivalence 
  \begin{center}
   \begin{tikzcd}[row sep=0.5in, column sep=0.5in]
    \Map(B,Y) \arrow[rr] \arrow[dr, "p_*"'] & & \Map(A,Y) \times_{\Map(A,X)} \Map(B,X) \arrow[dl, "\pi_2"] \\
    & \Map(B,X) & 
   \end{tikzcd}
   .
  \end{center}
 By \cref{Cor Fiberwise equiv for Kan fib}, this is equivalent to proving that for each map $g: B \to X$ the induced map 
 $$\Map_{/X}(B,Y) \to \Map_{/X}(A,Y)$$
 is an equivalence, which is exactly the condition of being fibrant in the localized model structure (\cref{The Localization for sS}). 
 \end{proof}

 \begin{theone}\label{the:induced vs localized}
  Let $X$ be a simplicial space and $\mathcal{L}$ be a set of monomorphisms in $\ss_{/X}$. 
  Then the adjunction
  \begin{center}
   \adjun{(\ss_{/X})^{loc\mathcal{M}} }{(\ss_{/X})^{\mathcal{M}}}{id}{id}
  \end{center}
  is a Quillen adjunction, which is a Quillen equivalence if $X$ is fibrant. In fact, in this case the two model structures are isomorphic.
  Here the left hand side has the localized model structure and the right hand side has the induced localized model structure 
  (\cref{rem:localized vs induced localized}).
 \end{theone}

 \begin{proof}
  Both sides have the same set of cofibrations. In order to finish the proof it suffices to show that every fibrant object in 
  the induced localized model structure on $\ss_{/X}$ is fibrant in the localized model structure on $\ss_{/X}$
  and that the opposite holds if $X$ is fibrant in $\ss$.
  \par 
  Let $p: Y \to X$ be a Reedy fibration. Then we have following diagram:
   \begin{center}
   \begin{tikzcd}[row sep=0.5in, column sep=0.5in]
    \Map(B,Y) \arrow[r, "f^*"] \arrow[d, "p_*"] & \Map(A,Y) \arrow[d, "p_*"] \\
    \Map(B,X) \arrow[r, "f^*"] & \Map(A,X)
   \end{tikzcd}
   .
  \end{center}
  If $p$ is fibrant in the induced localized model structure on $\ss_{/X}$, 
  then the square above is a homotopy pullback square as it is a simplicial model structure 
  and $f: A \to B$ is a trivial cofibration in the localized model structure on $\ss$. However, by 
  \cref{lemma:fibrant objects in localized model structure}, this is equivalent to $p: Y \to X$ being fibrant in the 
  localized model structure on $\ss_{/X}$. This finishes one side and proves the adjunction above is a Quillen adjunction.
  \par 
  On the other hand, let us assume $X$ is fibrant in the localized model structure on $\ss$ 
  and $p: Y \to X$ is fibrant in the localized model structure on $\ss_{/X}$. 
  The fibrancy of $p: Y \to X$ implies, again by \cref{lemma:fibrant objects in localized model structure}, that the square 
    \begin{center}
   \begin{tikzcd}[row sep=0.5in, column sep=0.5in]
    \Map(B,Y) \arrow[r, "f^*", "\simeq"'] \arrow[d, "p_*"] & \Map(A,Y) \arrow[d, "p_*"] \\
    \Map(B,X) \arrow[r, "f^*", "\simeq"'] & \Map(A,X)
   \end{tikzcd}
  \end{center}
  is a homotopy pullback square and the fibrancy of $X$ implies the bottom map is a Kan equivalence. 
  Hence the top map is a Kan equivalence as well. This means that $Y$ is fibrant in the localized model structure on $\ss$. 
  Thus, by \cref{The Localization for sS}, the map $p: Y \to X$ is a fibration in the localized model structure on $\ss$, 
  as Reedy fibrations between fibrant objects are fibrations in the localized model structure. 
 \end{proof}

\section{Comparison with Quasi-Categories} \label{Sec Comparison with Quasi-Categories}

 In this part we confirm that the covariant model structure for simplicial sets coincides with the covariant model structure 
 on simplicial spaces, by proving they are Quillen equivalent, via two different Quillen equivalences. 
 \par 
 The trick is to realize that that the Quillen equivalences between quasi-categories and complete Segal spaces 
 constructed by Joyal and Tierney \cite{joyaltierney2007qcatvssegal} descend to Quillen equivalences between their 
 respective covariant model structures.
 \par 
 We will thus start by giving a quick review of the relevant results in \cite{joyaltierney2007qcatvssegal}
 and review the relevant definitions of the covariant model structure on simplicial sets. We will only focus on specific results 
 that we need in this section and refer the reader to the vast literature for any details \cite{heutsmoerdijk2015leftfibrationi},
 \cite{lurie2009htt}.
 
 \begin{notone} \label{not:sSet vs s}
  As before we use $\s^{Kan}$ to denote the category of simplicial sets with the Kan model structure. 
  On the other hand we use $\sSet^{Joy}$ for the category of simplicial sets with the Joyal model structure.
 \end{notone}
 
 \begin{defone} \label{def:left fibration qcats}
  \cite[Definition 2.0.0.3]{lurie2009htt}
  A map $f:S \to T$ of simplicial sets is a {\it left fibration} if it satisfies the right lifting property with respect to 
  all horn conclusions of the form 
  $\ds \Lambda^n_i \to \Delta[n]$, where $ 0 \leq i < n$.
 \end{defone}
 
 \begin{defone} \label{Def Covariant model structure on QCat}
  \cite[Definition 2.1.4.5, Proposition 2.1.4.7, Proposition 2.1.4.8]{lurie2009htt}
  Let $S \in \sSet$ be a simplicial set. There is a left proper, combinatorial, simplicial model structure on $\sSet_{/S}$, 
  called the covariant model structure
  and denoted by $(\sSet_{/S})^{cov}$, which satisfies following conditions:
  \begin{enumerate}
   \item A map $T \to U$ over $S$ is a cofibration if it is a monomorphism.
   \item The fibrant objects are the left fibrations
  \end{enumerate}
 \end{defone}
 
 We can also characterize the fibrations between fibrant objects.
 
 \begin{corone} \label{cor:cov fib between left fib}
  \cite[Corollary 2.2.3.14]{lurie2009htt}
  A map between left fibrations is a fibration in the covariant model structure if and only if it is a left fibration.
 \end{corone}

 There are several important theorems about the covariant model structure we are going to need later on.
 \begin{theone} \label{The Covariant local of Joyal}
  \cite[Theorem 3.1.5.1]{lurie2009htt}  
  Let $S$ be a simplicial set. Then the following adjunction
  \begin{center}
   \adjun{(\sSet_{/S})^{Joyal}}{(\sSet_{/S})^{cov}}{id}{id}
  \end{center}
  is a Quillen adjunction, where the left hand side has the Joyal model structure 
  and the right hand side has the covariant model structure.
  This implies that the covariant model structure is a localization of the Joyal model structure.
 \end{theone}
 
 \begin{theone} \label{the:covar model structure invariant categorical equiv}
  \cite[Proposition 2.1.4.10, Remark 2.1.4.11]{lurie2009htt}
  Let $f: S \to T$ be a map of simplicial sets. Then the adjunction 
  \begin{center}
   \adjun{(\sSet_{/S})^{cov}}{(\sSet_{/T})^{cov}}{f_!}{f^*}
  \end{center}
  is a Quillen adjunction, which is a Quillen equivalence if $f$ is a categorical equivalence.
  Here both sides have the covariant model structure.
 \end{theone}

 We now move on to review the main results in \cite{joyaltierney2007qcatvssegal}.

 \begin{theone} \label{The CSS and Qcat equiv with ip}
  \cite[Theorem 4.11]{joyaltierney2007qcatvssegal} 
  Let $p_1^*: \sSet \to \ss$ be the functor which associates to each simplicial set $S$ the simplicial discrete space $p_1^*(S)$ 
  defined as $(p_1^*(S))_{nl} = S_n$ and 
  let $i_1^*: \ss \to \sSet$ be its right adjoint. 
  Then this gives us the following adjunction
  \begin{center}
   \adjun{(\sSet)^{Joy}}{(\ss)^{CSS}}{p_1^*}{i_1^*}
  \end{center}
  which is a Quillen equivalence, where $\sSet$ has the Joyal model structure and $\ss$ has the CSS model structure. 
 \end{theone}

 \begin{theone} \label{The CSS and Qcat equiv with t}
  \cite[Theorem 4.12]{joyaltierney2007qcatvssegal} 
  Let $t_! : \ss \to \sSet$ be the left Kan extension of the map which is defined on the generators $F(n) \times \Delta[l]$ as
  $t_!(F(n) \times \Delta[l]) = \Delta[n] \times J[l]$. Let $t^!: \sSet \to \ss$ be the right adjoint of this construction, 
  i.e. $t^!(S)_{nl} = \Hom_{\ss}(\Delta[n] \times J[l], S)$. Then this defines a Quillen equivalence
  \begin{center}
   \adjun{(\ss)^{CSS}}{(\sSet)^{Joy}}{t_!}{t^!}
  \end{center}
  with $\ss$ having the CSS model structure and $\sSet$ having the Joyal model structure.
 \end{theone}
 
 Although there are two different adjunctions, they have nice interactions which we will need later on.
 
 \begin{lemone} \label{lemma:gS CSS equiv}
  \cite[Proposition 4.10]{joyaltierney2007qcatvssegal}
  Let $S$ be a quasi-category. Then the natural map 
  $$g_S: p_1^*S \to t^!S$$
  is an equivalence in the CSS model structure.
 \end{lemone}
  
  There is an analogous statement for complete Segal spaces.
  
 \begin{lemone} \label{lemma:h X categorical equivalence} 
  Let $X$ be a complete Segal space, then the natural map 
  $$h_X: i_1^*X \to t_!X$$
  is a categorical equivalence. 
 \end{lemone}

 We now move on to the main topic of this section: the two Quillen equivalences.
 
 \begin{remone}
  Let 
  \begin{center}
   \adjun{\C}{\D}{F}{G}
  \end{center}
  be an adjunction of categories and $C$ an object in $\C$. Then we get an adjunction 
  \begin{center}
   \adjun{\C_{/C}}{\D_{/FC}}{F}{u^*G}
  \end{center}
  where the left adjoint takes a map $f: D \to C$ to $Ff: FD \to FC$ and the right adjoint takes a map 
  $f:D \to FC$ to the pullback $u^*(G(f)): u^*G(D) \to C$, where $u: C \to GFC$ is the unit map.
 \end{remone}

 \begin{theone} \label{The Covariant equivalences ip}
 Let $S$ be a simplicial set. 
 The adjunction
  \begin{center}
   \adjun{(\sSet_{/S})^{cov}}{(\ss_{/p_!^*S})^{cov}}{p_1^*}{u^*i_1^*}
  \end{center}
 is a Quillen equivalence, where both sides have the covariant model structure.
 \end{theone}

 \begin{theone} \label{The Covariant equivalences t}
  The adjunction 
  \begin{center}
   \adjun{(\ss_{/X})^{cov}}{(\sSet_{/t_! X})^{cov}}{t_!}{u^*t^!}
  \end{center}
  is a Quillen equivalence, where both sides have the covariant model structure.
 \end{theone}

 We will prove these two theorems in three steps:
 \begin{enumerate}
  \item Prove they are Quillen adjunctions: \cref{lemma:ip Quillen adj}, \cref{lemma:t Quillen adj}.
  \item Reduce the proof to complete Segal spaces and quasi-categories: \cref{rem:ip and t reduce to fibrants}.
  \item Prove they are Quillen equivalences: \cref{lemma:ip Quillen equiv qcat}, \cref{lemma:t Quillen equiv CSS}.
 \end{enumerate}

 We start with the Quillen adjunctions.
 
 \begin{lemone} \label{lemma:ip Quillen adj}
  The adjunction
  \begin{center}
   \adjun{(\sSet_{/S})^{cov}}{(\ss_{/p_1^*S})^{cov}}{p_1^*}{u^*i_1^*}
  \end{center}
  is a Quillen adjunction, where both sides have the covariant model structure.
 \end{lemone}
 \begin{proof}
  We use \cref{Lemma For Quillen adj}.
  Clearly, $p_1^*$ takes cofibrations to cofibrations as they are just inclusions. 
  So, all that is left is to show that $u^*i_1^*$ takes fibrations between fibrant objects to fibrations.
  By \cref{lemma:fib between left fib} a fibration between fibrant objects is just a left fibration. 
  Thus it suffices to prove that $u^*i_1^*$ preserves left 
  fibrations. The map $u^*$ just pulls back along the unit, which preserves left fibrations, as it is given by a right lifting property 
  (\cref{def:left fibration qcats}).
  So it suffices to prove that $i_1^*$ preserves left fibrations.
  \par 
  Let $p:L \to p_1^*S$ be a left fibration. We have to show that $i_1^*(p)$ satisfies the right lifting property with respect to horns 
  $\Lambda[n]_i \to \Delta[n]$ where $0 \leq i < n$ (\cref{def:left fibration qcats}). 
  Using the adjunction $(p_1^*,i_1^*)$ This is equivalent to $p$ having the 
  right lifting property with to the maps $j: L(n)_i \to F(n)$, which means we have to prove $j$ is a trivial cofibration in $\ss$ with respect 
  to the covariant model structure. 
  \par 
  From \cref{cor:covariant equiv over F n} and the fact that 
  $\ordered{0,...,k} \times_{\id_{F(n)}} j: F(k) \times_{F(n)} L(n)_i \to F(k)$ is a diagonal equivalence (both sides are
  diagonally contractible), we deduce that $j$ is a covariant equivalence over $F(n)$ and so is a covariant
  equivalence over $p_1^*S$ as well (\cref{The Base change covar}).
 \end{proof}
 
 \begin{lemone} \label{lemma:t Quillen adj}
  The adjunction 
  \begin{center}
   \adjun{(\ss_{/X})^{cov}}{(\sSet_{/t_!X})^{cov}}{t_!}{u^* t^!}
  \end{center}
  is a Quillen adjunction where both sides have the covariant model structure.
 \end{lemone}
 \begin{proof}
  We will show the adjunction satisfies the three conditions of \cref{Lemma For Quillen adj}.
  Clearly, $t_!$ takes cofibrations to cofibrations as they are just monomorphisms.
  Thus we only need to prove that $u^*t^!$ preserves fibrant objects and fibrations between fibrant objects.
  By \cref{cor:cov fib between left fib}, a fibration between fibrant objects is a left fibration, thus it suffices to prove 
  that $u^* t^!$ preserves left fibrations. 
  \par 
  Let $p: L \to t_!X$ be a left fibration. Then, by \cref{The Covariant local of Joyal}, it is a fibration in the Joyal model structure and so 
  $u^*t^!(p)$ is a CSS fibration, by \cref{The CSS and Qcat equiv with t}, 
  and so in particular a Reedy fibration. Thus we only have to prove that for every map $F(n) \to X$ the induced map 
  $$\ordered{0}^*: \Map_{/X}(F(n),u^*t^!(p) ) \to \Map_{/X}(F(0), u^*t^!(p)) $$ 
  is a Kan equivalence.
  By adjunction, this is equivalent to proving that $\ordered{0}: \Delta[0] = t_!(F(0)) \to t_!(F(n))= \Delta[n]$ is a covariant equivalence over $t_!X$.
  However, it is a well-established fact that the map $\ordered{0}: \Delta[0] \to \Delta[n]$ is a covariant equivalence over $\Delta[n]$ 
  and so in particular over $t_!X$ 
  (by \cref{The Base change covar}). For an elegant proof of this fact see \cite[Lemma 2.5]{heutsmoerdijk2015leftfibrationi}.
 \end{proof}
 
 Next we will reduce the proof to the case of complete Segal spaces and quasi-categories.
 
 \begin{remone} \label{rem:ip and t reduce to fibrants}
  Let $S$ be a simplicial set and choose a quasi-category fibrant replacement $i:S \to \hat{S}$ and 
  let $X$ be a simplicial space and choose a complete Segal space fibrant replacement $j: X \to \hat{X}$.
  Then we have following diagram of Quillen adjunctions:
  \begin{center}
   \begin{tikzcd}[row sep=0.5in, column sep=0.5in]
    (\sSet_{/S})^{cov} \arrow[r, shift left = 1.8, "p_1^*", "\bot"'] \arrow[d, shift left = 1.8, "i_!", "\rotatebot"'] & 
    (\ss_{/p_1^*S})^{cov} \arrow[l, shift left = 1.8, "u^*i_1^*"] \arrow[d, shift left = 1.8, "(p_1^*(i))_!", "\rotatebot"'] 
    & & 
    (\ss_{/X})^{cov} \arrow[r, shift left = 1.8, "t_!", "\bot"'] \arrow[d, shift left = 1.8, "j_!", "\rotatebot"'] & 
    (\sSet_{/t_!X})^{cov} \arrow[l, shift left = 1.8, "u^*t^!"] \arrow[d, shift left = 1.8, "(t_!(j))_!", "\rotatebot"'] 
    \\
    (\sSet_{/\hat{S}})^{cov} \arrow[r, shift left = 1.8, "p_1^*", "\bot"'] \arrow[u, shift left = 1.8, "i^*"] & 
    (\ss_{/p_1^*\hat{S}})^{cov} \arrow[l, shift left = 1.8, "u^*i_1^*"] \arrow[u, shift left = 1.8, "(p_1^*(i))^*"] 
    & & 
    (\ss_{/\hat{X}})^{cov} \arrow[r, shift left = 1.8, "t_!", "\bot"'] \arrow[u, shift left = 1.8, "j^*"] & 
    (\sSet_{/t_!\hat{X}})^{cov} \arrow[l, shift left = 1.8, "u^*t^!"] \arrow[u, shift left = 1.8, "(t_!(j))^*"] 
   \end{tikzcd}
  \end{center}
  All vertical Quillen adjunctions are Quillen equivalences 
  (\cref{the:covar model structure invariant categorical equiv}, \cref{The Covar invariant under CSS equiv}) 
  thus the top horizontal Quillen adjunctions are Quillen equivalences 
  if and only if the bottom Quillen adjunctions are Quillen equivalences. 
 \end{remone}

 We are now ready to move on to the last step.
 
 \begin{lemone}\label{lemma:ip Quillen equiv qcat}
  Let $S$ be a quasi-category. Then the adjunction 
   \begin{center}
   \adjun{(\sSet_{/S})^{cov}}{(\ss_{/p_!^*S})^{cov}}{p_1^*}{u^*i_1^*}
  \end{center}
  is a Quillen equivalence, where both sides have the covariant model structure.
  \begin{center}
  \end{center}
 \end{lemone}

 \begin{proof}
  We prove that the derived unit and counit maps are equivalences. 
  \par 
  First we prove the derived counit map is an equivalence. 
  Let $p: L \to p_1^*S$ be a left fibration. We need to prove that the map $p_1^*u^*i_1^*L \to L$ is a covariant equivalence over $p_1^*S$.
  By \cref{The CSS and Qcat equiv with ip}, $u: S \to i_1^* p_1^* S$ is the identity map. 
  Hence we only need to prove that $p_1^*i_1^*L \to L$ is a covariant equivalence over $p_1^*S$. 
  However, this follows immediately from the fact that $p_1^*i_1^*L \to L$ is a complete Segal space equivalence (by 
  \cref{The CSS and Qcat equiv with ip}) and hence a covariant equivalence over $p_1^*S$ (\cref{The Covariant local of CSS}).

  \medskip  
  
  We move on to prove the derived unit map is an equivalence. Let $p:L \to S$ be a left fibration. Then $p_1^*L \to p_1^*S$ is not 
  a left fibration and so we need to find left fibrant replacement. 
  We have following diagram:
  \begin{center}
   \begin{tikzcd}[row sep=0.5in, column sep=0.5in]
    p_1^*L \arrow[dr] \arrow[drr, bend left=20, "g_L", "\simeq"'] \arrow[ddr, bend right=20] &[-0.2in] & \\[-0.2in]
    & (g_S)^*t^!L \arrow[d, twoheadrightarrow] \arrow[r, "\simeq"] & t^!L \arrow[d, twoheadrightarrow] \\
    & p_1^*S \arrow[r, "g_S", "\simeq"'] & t^!S 
   \end{tikzcd}
  \end{center}
  By \cref{lemma:t Quillen adj}, $t^!L \to t^!S$ is a left fibration and so $(g_S)^*(t_!L) \to p_1^*S$ is a left fibration 
  (\cref{Lemma Pullback preserves Left fibrations}). Moreover, 
  $(g_S)^*t^!L \to t^!L$ is a CSS equivalence, as pulling back along left fibrations preserves CSS equivalences (\cref{The Pullback preserves CSS equiv})
  and $g_L$ is a CSS equivalence by \cref{lemma:gS CSS equiv}. Thus $p_1^*L \to (g_S)^*t^!L$ is a CSS equivalence over $p_1^*S$ and so a 
  covariant equivalence over $p_1^*S$ (by \cref{The Covariant local of CSS}). 
  Hence, $p_1^*L \to (g_S)^*t^!L$ is the derived unit map.
  \par 
  Thus, in order to prove the derived unit is an equivalence we only need to show that 
  $$L \to i_1^*p_1^*L \to (i_1^*)(g_S)^*t^!L$$
  is a covariant equivalence over $p_1^*S$. However, both maps are just the identity map and hence we are done. 
 \end{proof}

 \begin{lemone} \label{lemma:t Quillen equiv CSS}
  Let $X$ be a complete Segal space. 
  Then the adjunction 
  \begin{center}
   \adjun{(\ss_{/X})^{cov}}{(\sSet_{/t_! X})^{cov}}{t_!}{u^*t^!}
  \end{center}
  is a Quillen equivalence, where both sides have the covariant model structure.
 \end{lemone}
 
 \begin{proof}
  We have the following chain of Quillen adjunctions:
  \begin{center}
   \begin{tikzcd}[row sep=0.5in, column sep=0.5in]
    (\sSet_{/i_1^* X})^{cov} \arrow[r, shift left = 1.6, "p_1^*"] &
    (\ss_{/p_1^*i_1^*X})^{cov} \arrow[r, shift left = 1.6, "c_!"] \arrow[l, shift left = 1.6, "u^*i_1^*"] &
    (\ss_{/X})^{cov} \arrow[r, shift left = 1.6, "t_!"] \arrow[l, shift left = 1.6, "c^*"] &
    (\sSet_{/t_!X})^{cov} \arrow[l, shift left = 1.6, "u^*t^!"] 
   \end{tikzcd}
  \end{center}
  Here $c: p_1^*i_1^*X \to X$ is the counit map of the adjunction.
  \par 
  The first adjunction is a Quillen equivalence by \cref{lemma:ip Quillen equiv qcat}. 
  The middle one is a Quillen equivalence by \cref{The Covar invariant under CSS equiv}, as $c$ is an equivalence of 
  complete Segal spaces (as proven in \cref{The CSS and Qcat equiv with ip}). 
  Finally, the composition of these three adjunctions gives us the Quillen adjunction:
   \begin{center}
   \adjun{(\sSet_{/i_1^* X})^{cov}}{(\sSet_{/t_!X})^{cov}}{(h_X)_!}{(h_X)^*}.
  \end{center}
  It is induced by the map $h_X: i_1^*X \to t_!X$ which is a categorical equivalence, by \cref{lemma:h X categorical equivalence}, 
  and so a Quillen equivalence.
  \par 
  Thus by $2$-out-of-$3$ the adjunction $(t_!,u^*t^!)$ is a Quillen equivalence. 
 \end{proof}

 The Quillen equivalence above has an interesting corollary.
 
 \begin{corone}
  The covariant model structure on $(\sSet_{/S})^{cov}$ is a localization of the Joyal model structure with respect 
  to the set of maps $\Delta[0] \to \Delta[n] \to S$.
 \end{corone}
 
 \begin{remone}
  Essentially we proved that the two Quillen equivalences that Joyal and Tierney introduced remain an equivalence
  after we localize both sides. Theoretically, we could have just proven these theorems using the fact that
  localizing with respect to the "same" maps on both sides preserves Quillen equivalences. However, the issue is that we didn't have
  a good enough understanding of the localization of the Joyal model structure 
  (i.e. it is not clear which maps we are localizing with respect to).
  It is just after this proof that we get a clear sense of the localizing maps.
 \end{remone}

 \bibliographystyle{alpha}
 \bibliography{main}
 
\end{document}